%% file: main.tex
\documentclass[letterpaper,12pt]{article}

\usepackage{amsmath}
\usepackage{amsthm}
\usepackage{graphicx}
\usepackage{float}
\usepackage[labelsep=period]{caption}
\usepackage{authblk}
\usepackage{enumerate}
\usepackage{amssymb}
\usepackage{colonequals}
\usepackage{bm}
\usepackage{setspace}
\usepackage{rotating}
\usepackage[margin=1in]{geometry}
\usepackage{scalerel}
\usepackage{changepage}
\usepackage{comment}
\usepackage{scrextend}
\def\ess{\scalerel*{\includegraphics{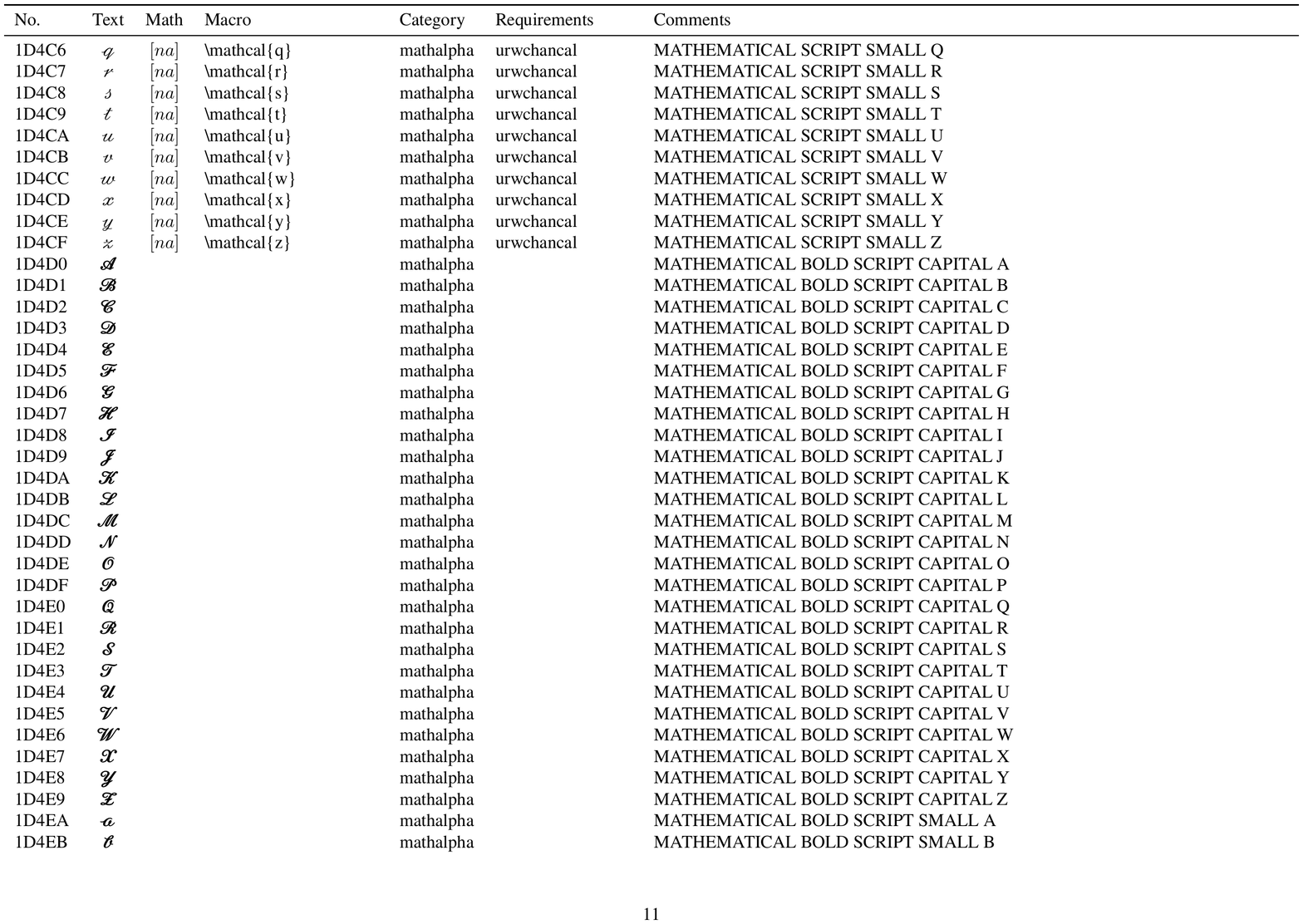}}{m}}
\def\boxslash{\scalerel*{\includegraphics{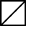}}{B}}
\usepackage{units}
\usepackage{tikz}
\renewcommand{\parallel}[0]{%
\begin{tikzpicture}%
\draw (0.4ex,0.0ex) -- (0.8ex,1.7ex);%
\draw (0.85ex,0.0ex) -- (1.25ex,1.7ex);%
\end{tikzpicture}%
}
\numberwithin{equation}{section}
\DeclareMathOperator{\vol}{vol}
\DeclareMathOperator{\lub}{l.u.b}
\newtheorem{theorem}{Theorem}

\newtheorem{lemma}{Lemma}
\newtheorem{proposition}{Proposition}[section]
\newtheorem{corollary}{Corollary}[subsection]
\theoremstyle{definition}
\newtheorem*{definition}{Definition}
\newtheorem*{RelDens}{Definition \textnormal {(relative density)}}
\newtheorem*{IntDens}{Definition \textnormal {(intrinsic density)}}
\newtheorem*{GlobDens}{Definition \textnormal {(global density of infinite packings)}}
\newtheorem*{FundProb}{Fundamental Problem of Sphere Packings}
\newtheorem*{Kep2}{\textit{Kepler's Conjecture} \textnormal {(2nd version)}}
\newtheorem{example}{Example}[section]
\newtheorem*{remark}{Remark}
\newtheorem*{remarks}{Remarks}

\newtheoremstyle{named}{}{}{\itshape}{}{\bfseries}{:}{.5em}{\thmnote{#3}}
\theoremstyle{named}
\newtheorem*{namedformula}{}
\newcommand{\slfrac}[2]{\left.#1\middle/#2\right.}
\newcommand{\pet}{\frac{\pi}{\sqrt {18}}}
\newcommand{\spet}{\slfrac{\pi}{\sqrt {18}}}
\newcommand{\sconf}{\mathcal{S}(\Sigma)}
\newcommand{\stfive}{St_5(\theta_1, \theta_2)}
\newcommand{\rhostar}{\tilde{\rho}(St(\cdot))}
\newcommand{\nm}[1]{|#1|}
\newcommand{\verteq}[0]{\begin{turn}{90} $=$\end{turn}}
\newcommand{\dotp}[2]{\mathbf{#1}\cdot\mathbf{#2}}
\newcommand{\crossp}[2]{\mathbf{#1}\times\mathbf{#2}}
\newcommand\sabc[0]{\sigma(ABC)}
\newcommand\tabc[0]{\Delta ABC}
\newcommand{\scr}[0]{r}

\newcommand{\sconfp}{\mathcal{S}^{\prime}(\Sigma)}
\newcommand{\rhobarred}{\bar{\rho}(\cdot)}
\newcommand{\stard}{St(\cdot)}
\newcommand{\sumf}{\sum_{i=1}^{5}}

\newcommand{\lof}{\mathcal{L}(\cdot)}
\newcommand{\pots}{\slfrac{\pi}{3}}

\title{A new local invariant and simpler proof of Kepler's conjecture
and the least action principle on the crystal formation of dense type}
\author[1,2]{Wu-Yi Hsiang}
\affil[1]{Department of Mathematics, University of California, Berkeley, CA 94720}
\affil[2]{Department of Mathematics, HK University of Science and Technology, Clear Water Bay, Kowloon, Hong Kong}

\begin{document}

\maketitle

\input{1-introduction/1-introduction.tex}

\input{2-major-theorems/2-major-theorems.tex}

\input{3-solid-geometry/3-solid-geometry.tex}

\input{4-type-I-configurations/4-type-I-configurations.tex}

\input{5-techniques/5-techniques.tex}

\input{6-proof-of/6-proof-of.tex}

\input{7-proof-case-2/7-proof-case-2.tex}

\input{8-conclusion/8-conclusion.tex}

\end{document}

%% file: 1-introduction/1-introduction.tex
\section{Introduction}

\subsection{Three kinds of sphere packings, various kinds of densities and problems of their optimalities}

Basically, there are three different kinds of packings of spheres of identical size, namely

\begin{enumerate}[(i)]
\item packings with containers; $\mathcal{P} \subset \Gamma$
\item finite packings without container, e.g. crystals;
\item infinite packings with the whole space as the container.
\end{enumerate}

For example, putting marbles into a jar, oranges into a box or soybeans into a silo are daily-life examples of the first kind; while pieces of crystals of gold, silver, lead etc. are Nature-created examples of the second kind.
On the other hand, those infinite packings such as the f.c.c. packing, hexagonal close packings \cite{barrow} and lattice packings are, in fact, just some mathematical models serving as the limiting situations of the first and the second kinds as their sizes tend to infinity.

In the study of sphere packings, the central problems are the \textit{problems of optimalities on various kinds of densities.}  Of course, it is necessary to first give a \textit{precise definition of the kind of density} before studying the \textit{problem of its optimality}.

Let $\mathcal{P} \subset \Gamma$ be a packing into a given container. It is quite obvious that the density $\rho(\mathcal{P} \subset \Gamma)$ of such a packing should be defined as follows, namely

\begin{equation}
  \rho(\mathcal{P} \subset \Gamma) \colonequals \vol \mathcal{P} / \vol \Gamma
\end{equation}
while the \textit{optimal density of packing $r$-spheres into $\Gamma$} is given by
\begin{equation}
\rho ( r, \Gamma ) \colonequals \lub \{ \rho ( \mathcal{P} , \Gamma ) \}
\end{equation}
where $\mathcal{P}$ runs through all packings of $r$-spheres into $\Gamma$.
Note that $\rho ( r, \Gamma )$ will certainly depend on the shape of $\Gamma$ and the relative size between $\Gamma$ and $r$-sphere. Anyhow, this motivates us to study
\begin{equation}
\displaystyle\hat{\rho}(\Gamma)=\limsup_{k\to\infty} \rho (1,k\Gamma)
\end{equation}
where $k\Gamma$ denotes the $k$-times magnification of $\Gamma$, and the following formulation of Kepler's conjecture on sphere packings:

\begin{Kep2}
For a large class of $\Gamma$, e.g. those with piecewise smooth $\partial\Gamma$, $\hat{\rho}(\Gamma)$ should always be equal to $\spet$.
\end{Kep2}

\subsubsection*{Problem of least action principle on the crystal formation of dense type (cf. \cite{hsiang})}

The physical shape of atoms can be regarded as microscopic spheres while a small piece of crystal of a monatomic element, such as gold and silver etc., often consists of billions of trillions of such microscopic spheres aggregated into a specific type of regular arrangement, exhibiting fascinating geometric regularity and remarkable precision.
For example, the crystal structures of forty-eight chemical elements are of hexagonal close packing type which have the highest known density of $\spet$.  Thus, it is quite natural to pose the following type of ``uniqueness'' problem, namely

``How to \textit{properly define} the density of packings of the second kind so that the above \textit{geometric regularity} is actually \textit{the consequence of density optimality}, which will be referred to as the least action principle of crystal formation of the dense type.''

The \textit{new local invariant} of \S1.2 will play the key-role of providing such a proper definition as well as a far-reaching localization for the proof of such a theorem (cf. Theorem III, \S2.1).

\subsection{A simple local invariant and the definition of global densities of packings of the second and third kinds}

In his booklet of 1611 \cite{kepler}, Kepler had already introduced the concept of \textit{local cell}, nowadays often referred to as Voronoi cell, which assigns a surrounding convex polyhedron to each given sphere $S_i$ in $\mathcal{P}$, consisting of those points that are as close to the center of $S_i$, say $O_i$, as to that of others, say $\{O_j\}$, namely
\begin{equation}
\displaystyle C(S_i,\mathcal{P}) = \bigcap_{j \ne i} H_{ij}
\end{equation}
where $H_{ij}$ is the $O_i$-side of the perpendicular bisector of $\overline{O_iO_j}$.
We shall always assume that $C(S_i,\mathcal{P})$ are bounded for every $S_i$ in $\mathcal{P}$, thus the above intersection can always be reduced to that of a finite, irreducible intersection, namely
\begin{equation}
C(S_i,\mathcal{P}) = H_{ij_1} \cap \ldots \cap H_{ij_a}
\end{equation}
Note that an index $j$ belongs to the above minimal set $\{j_1,j_2,\ldots,j_a\}$ when and only when $C(S_i,\mathcal{P})$ and $C(S_j,\mathcal{P})$ share a \textit{common face} and such a pair $\{S_i,S_j\}$ are \textit{defined to be neighbors of each other}.
Anyhow, \textit{local cell} and \textit{neighbor} are the most basic concepts on the geometry of sphere packings that all the other important ones are based upon.

\subsubsection{Local cell decomposition and its dual decomposition}

Let $\mathcal{P} = \{ S_i, i \in I \}$ be a given infinite packing with all of its local cells $C(S_i,\mathcal{P})$ of bounded type, each of them containing a single sphere of $\mathcal{P}$.
We shall denote them simply by $\{ C_i, i \in I \}$ and call it the \textit{local cell decomposition} of $\mathcal{P}$.
Moreover, the above decomposition has a natural dual decomposition into convex polyhedra with centers of spheres of $\mathcal{P}$ as their vertices, namely, the \textit{Delaunay decomposition}.
The duality between such a pair of fundamental decompositions associated to a given $\mathcal{P}$ can be summarized as follows:

\begin{enumerate}[(i)]
  \item The set of vertices of the D-decomposition are $\{ O_i \} \leftrightarrow \{ C_i \}$.
  \item The set of edges of the D-decomposition are those $\overline{O_iO_j}$ linking the centers of neighboring pairs $\{S_i,S_j\} \leftrightarrow$ those common faces of $\{C_i,C_j\}$ situated on the perpendicular bisector of $\overline{O_iO_j}$.
  \item The set of faces of the D-decomposition are those polygons spanned by the centers of those local cells with a \textit{common edge} $\leftrightarrow$ \{ their common edge, perpendicular to the face at its circumcenter \}.
  \item The set of convex polyhedra spanned by the centers of those local cells with a \textit{common vertex} $\leftrightarrow$ \{ their common vertex which is exactly the circumcenter of its corresponding D-cell \} .
\end{enumerate}

\subsubsection{A new kind of locally averaged density}

(cf. \S 1.4 of \cite{hsiang} for another kind.)

Set $I$ (resp. $J$) to be the indices sets of the L-cells (resp. D-cells) of the above dual pair of decompositions associated to a given $\mathcal{P}$, and set

\begin{equation}
w_i = \vol C_i, \quad
w^j = \vol \Omega_j, \quad
w_i^j = \vol C_i \cap \Omega_j
\end{equation}

\begin{definition}

  To each D-cell $\Omega_j$, set
  \begin{equation}
    \rho(\Omega_j) = \vol(\Omega_j \cap \mathcal{P}) / \vol \Omega_j
    \label{eqn:dcell}
  \end{equation}
  and call it the density of the D-cell $\Omega_j$.
  To each $S_i \in \mathcal{P}$, the {\it locally averaged density} of $\mathcal{P}$ at $S_i$ is defined to be
    \begin{equation}
    \displaystyle \bar{\rho}(S_i,\mathcal{P}) \colonequals \sum_{j \in J}w_i^j \rho(\Omega_j) / \sum_{j \in J} w_i^j
    \label{eqn:rhobar}
  \end{equation}
\end{definition}

Note that, for each given $i$, there are only a rather small number of $j$ with non-zero $w_i^j$ and $\sum_j w_i^j = w_i$.

\begin{remarks}
  \begin{enumerate}[(i)]
  \item In retrospect, the cluster of spheres centered at the vertices of a given $\Omega_j$ can be regarded as the subcluster of $\mathcal{P}$ of the \textit{most localized kind}.
    Thus, $\rho(\Omega_j)$ can be regarded as a kind of \textit{ultimate localization} of the concept of densities associated to a given $\mathcal{P}$.
  \item Note that $\bar{\rho}(S_i,\mathcal{P})$ is, itself, a weighted average of $\{ \rho(\Omega_j) \}$ that makes use of the dual pair of decompositions.
    Of course, its usefulness will only be determined by the ultimate test of whether it can provide a better result in the study of global optimalities of sphere packings, (cf. \S 2).
  \end{enumerate}
\end{remarks}

\subsubsection{Relative density and global densities of the second and third kind}\label{subsubsec:relativedensity}

Let $\mathcal{P}$ be an infinite packing and $\mathcal{P}^\prime = \{ S_i : i \in I^\prime\}$ be one of its finite subpackings.

\begin{RelDens}
  The \textit{relative density} of $\mathcal{P}^\prime$ in $\mathcal{P}$, denoted by $\bar{\rho}(\mathcal{P}^\prime,\mathcal{P})$, is defined to be
  \begin{equation}
    \displaystyle\bar{\rho}(\mathcal{P}^\prime,\mathcal{P}) \colonequals \frac
      {\sum_{i \in I^\prime}w_i \bar{\rho}(S_i,\mathcal{P})}
      {\sum_{i \in I^\prime}w_i}, \quad
      \mathcal{P}^\prime = \{ S_i, i \in I^\prime \}
      \label{eqn:reldens}
  \end{equation}
\end{RelDens}

\begin{IntDens}
  Let $\mathcal{P}^\prime$ be a given finite packing without container.
  The intrinsic density of $\mathcal{P}^\prime$, denoted by $\bar{\rho}(\mathcal{P}^\prime)$, is defined to be
  \begin{equation}
    \bar{\rho}(\mathcal{P}^\prime) \colonequals \lub \{ \bar{\rho} (\mathcal{P}^\prime, \bar{\mathcal{P}}) \}
  \end{equation}
  where $\{\bar{\mathcal{P}}\}$ runs through all possible extensions of $\mathcal{P}^\prime$.
\end{IntDens}

\begin{GlobDens}
  \begin{equation}
    \displaystyle \rho(\mathcal{P}) \colonequals \lub \{ \limsup_{n\to\infty}
    \bar{\rho}(\mathcal{P}_n,\mathcal{P}) \}
  \end{equation}
  where $\{\mathcal{P}_n\}$ runs through all possible exhaustion sequences of $\mathcal{P}$.
\end{GlobDens}

\begin{example}
  Let $\mathcal{P}$ be a hexagonal close packing.  Then
  \begin{equation}
    \bar{\rho}(S_i,\mathcal{P}) = \spet \quad \forall S_i \in \mathcal{P}
  \end{equation}
  and hence
  \begin{equation}
    \bar{\rho} (\mathcal{P}^\prime, \mathcal{P}) = \spet
  \end{equation}
  for all finite subpackings $\mathcal{P}^\prime$ in $\mathcal{P}$, and $\rho(\mathcal{P})$ is also equal to $\spet$.
  \begin{proof}
    The local cluster of D-cells that occur in (\ref{eqn:rhobar}) consists of octuple regular 2-tetrahedra and sextuple regular 2-octahedra.
    Therefore their volumes (resp. total solid angles) are given by
    \begin{equation}
     \Bigg\{ \begin{array}{l}
        \slfrac{\sqrt{8}}{3} \\
        \slfrac{8\sqrt{2}}{3}
      \end{array} 
      \left( \textrm{resp.}
      \Bigg\{ \begin{array}{c}
        4(3 \arccos \frac{1}{3}-\pi) \\
        6(4 \arccos (-\frac{1}{3})-2\pi)
      \end{array}
      \right)
    \end{equation}
and hence
    \begin{equation}
      \rho(\Omega_j) =
      \Bigg\{\begin{array}{ll}
        \sqrt{2}(3\alpha_0-\pi) & \sim 0.7796356 \\
        \frac{1}{\sqrt{8}}(3\pi-6\alpha_0) & \sim 0.7209029
      \end{array}
      \end{equation}
where $\alpha_0=\arccos \frac{1}{3}$. Thus
    \begin{equation}
      \begin{aligned}
        \bar{\rho}(S_i,\mathcal{P}) & = \frac{1}{4\sqrt{2}}\left \{ \frac{4}{3}\sqrt{2}\cdot\sqrt{2}(3\alpha_0-\pi) +
        \frac{8}{3}\sqrt{2}\cdot \frac{1}{\sqrt{8}}(3\pi-6\alpha_0) \right \} \\
        & = \spet
      \end{aligned}
    \end{equation}
\end{proof}
\end{example}

\subsection{Fundamental problem and fundamental theorem of sphere packings}

Note that a pair of spheres with their center distance less than $2\sqrt{2}r$ are automatically neighbors of each other (i.e. their local cells must share a common face whatever the arrangements of the others).
Thus, it is natural to introduce the following definition of \textit{clusters} of spheres:

\begin{definition}
  A finite packing of $r$-spheres is called a \textit{cluster} if any pair of them can always be linked by a chain with consecutive center distances less than $2\sqrt{2}r$.
\end{definition}

\begin{remark}
  A single sphere is, of course, regarded as a special case of cluster.
\end{remark}

\begin{FundProb}
  Set $\rho_N$ to be the {\it optimal intrinsic density} of all possible $N$-clusters, namely
  \begin{equation}
    \rho_N \colonequals \lub \{ \bar{\rho} (\mathcal{C}) \}
  \end{equation}
  where $\mathcal{C}$ runs through all possible clusters of $N$ spheres.
  What is $\rho_N$ equal to?
  and what are the geometric structures of those $N$-clusters together with their tightest surroundings with 
$\bar{\rho}(\mathcal{C},\mathcal{C}^*) = \rho_N$?
\end{FundProb}

In the beginning case of $N = 1$, $\rho_1$ is just the \textit{optimal locally averaged density}.
The above fundamental problem naturally leads to the proof of the following fundamental theorem, namely

\begin{theorem}
  The optimal locally averaged density is equal to $\spet$ and $\bar{\rho}(S_0,\mathcal{L}) = \spet$ when and only when the local packing $\mathcal{L}(S_0)$ is isometric to either that of the f.c.c. or the h.c.p. packing, which surrounds $S_0$ with twelve touching neighbors with their touching points as indicated in Figure~\ref{fig:fcchcp}.
\end{theorem}

\begin{figure}
  \begin{center}
    \includegraphics[width = 6in]{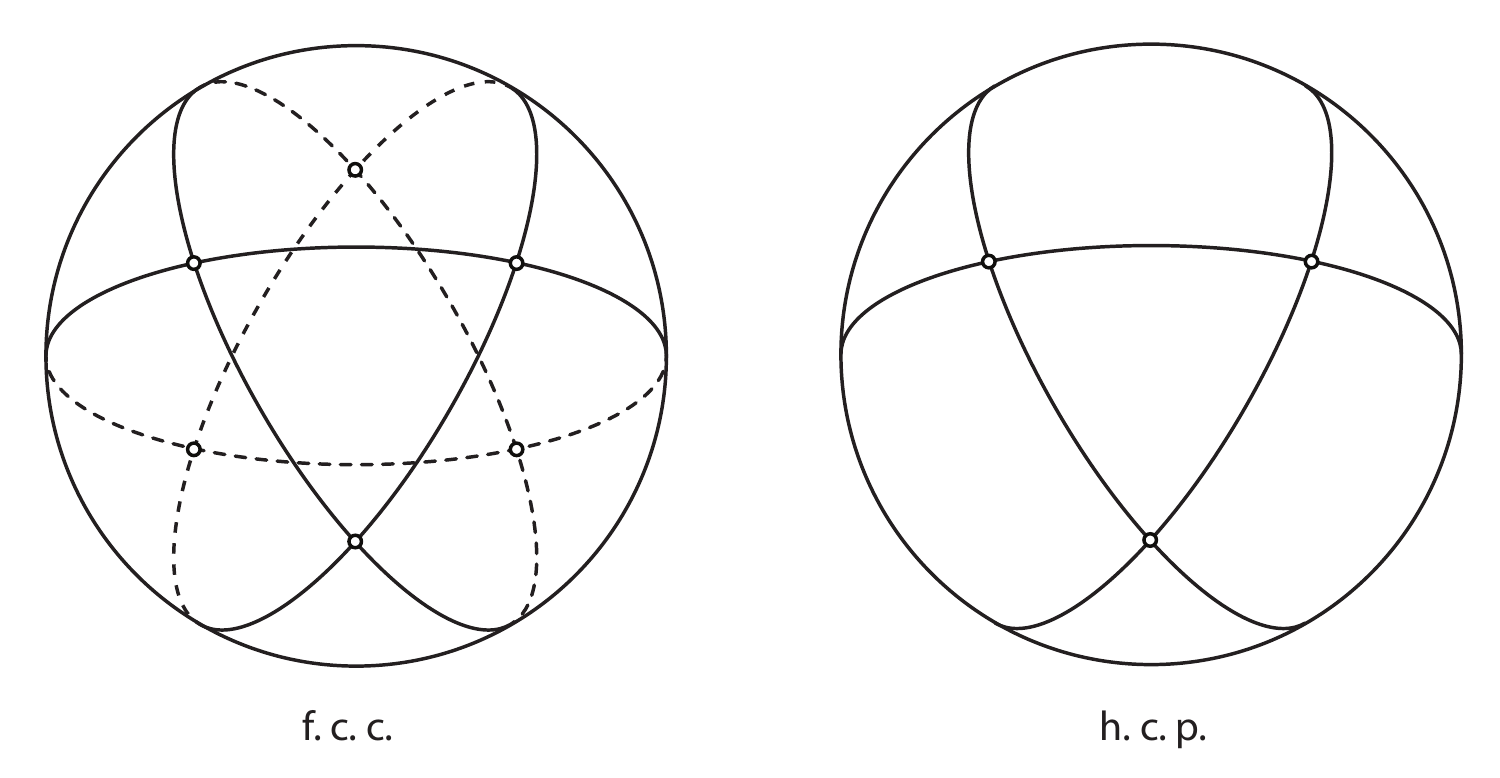}
    \caption{\label{fig:fcchcp} }
  \end{center}
\end{figure}

%% file: 2-major-theorems/2-major-theorems.tex
\section{Major theorems on global optimalities of sphere packings and their proofs via Theorem I}\label{sec:majortheorems}

In this section, we shall state and deduce the major results on the optimalities of global densities of sphere packings as consequences of Theorem I.

\subsection{Statements on the optimalities of the global densities of sphere packings (i.e. of the third, second, and first kinds)}

\begin{theorem}[Kepler's conjecture, 1st version]\label{keplerv1}
  The optimal global density of infinite sphere packings is equal to $\spet$, namely
  \begin{equation}
    \rho(\mathcal{P}) \le \spet = \rho(\textrm{\rm hexagonal close packings})
  \end{equation}
\end{theorem}

\begin{theorem}[Least action principle of crystal formation of dense type]\label{leastaction}
  \begin{equation}
    \rho_N = \spet \quad \forall N
  \end{equation}
  and $\bar{\rho}(\mathcal{C},\mathcal{C}^*) = \rho_N$ when and only when the $N$-cluster $\mathcal{C}$ together with its tightest surrounding $\mathcal{C}^*$ is a piecewise hexagonal close packing, namely, an assemblage of pieces of subclusters of hexagonal close packings.
\end{theorem}

\begin{theorem}[Kepler's conjecture, second version]\label{keplerv2}
  For all kinds of containers $\Gamma$ with piecewise smooth boundaries $\partial \Gamma$,
  \begin{equation}
    \hat{\rho}(\Gamma) = \spet
  \end{equation}
\end{theorem}

\subsection{Deductions of Theorems II, III, and IV via Theorem I: A far-reaching localization on global optimalities of sphere packings}

\begin{proposition}
  Theorem I implies Theorem II.
\end{proposition}
\begin{proof}
  Let $\mathcal{P}$ be an infinite packing and $\{ \mathcal{P}_n \}$ be one of its exhaustion sequences of finite subpackings.  Then, by Theorem I
  \begin{equation}
    \rho(\mathcal{P}_n,\mathcal{P}) \le \spet \quad \forall n
  \end{equation}
  Therefore
  \begin{equation}
    \displaystyle \limsup_{n \to \infty} \rho(\mathcal{P}_n,\mathcal{P}) \le \spet
  \end{equation}
  and hence
  \begin{equation}
    \displaystyle \rho(\mathcal{P}) \colonequals \lub \left \{ \limsup_{n \to \infty} \rho(\mathcal{P}_n,\mathcal{P}) \right \} \le \spet
  \end{equation}
\end{proof}

\begin{proposition}
  Theorem I implies Theorem III.
\end{proposition}
\begin{proof}
  Let $\mathcal{C}$ be a given $N$-cluster and $\mathcal{P}$ be one of those extensions of $\mathcal{C}$.  Then, by Theorem I, 
  \begin{equation}
    \bar{\rho}(S_i,\mathcal{P}) \le \spet \quad \forall S_i \in \mathcal{C}
  \end{equation}
  and equality holds when and only when $\mathcal{L}(S_i,\mathcal{P})$ (resp. ${C}(S_i,\mathcal{P})$) is the same as that of the f.c.c. or the h.c.p. Therefore
  \begin{equation}
    \displaystyle \bar{\rho}(\mathcal{C},\mathcal{P}) \colonequals
    \sum_{S_i \in \mathcal{C}} w_i \bar{\rho}(S_i,\mathcal{P}) /
    \sum_{S_i \in \mathcal{C}} w_i \le \pet
  \end{equation}
  and the above equality holds when and only when {\it all} of $\{ {C}(S_i,\mathcal{P}), S_i \in \mathcal{C} \}$ are either that of the f.c.c. or that of the h.c.p.
  Hence $\rho_N = \spet$ and $\bar{\rho}(\mathcal{C},\mathcal{C}^*) = \spet$ when and only when all of $\{ C(S_i,\mathcal{P}),S_i \in \mathcal{C} \}$ are either that of the f.c.c. or that of the h.c.p.;
  it follows from the cluster condition that such a collection of local cells are glued together along their common faces, and such a gluing is possible only when all the local pieces of
  \begin{equation}
    \mathcal{C}^* = \cup \{ \mathcal{L}(S_i,\mathcal{P}), S_i \in \mathcal{C} \}
  \end{equation}
  constitute subclusters of certain hexagonal close packings.
\end{proof}

\begin{proposition}
  Theorem I implies Theorem IV.
\end{proposition}
\begin{proof}
  Let $\Gamma$ be a given container with piecewise smooth boundary $\partial \Gamma$.
  Therefore, for sufficiently large $k$, $\partial(k \Gamma)$ is \textit{locally almost flat} everywhere except those corner or edge points.
  Let $\mathcal{P}$ be a packing of unit spheres into $k \Gamma$.
  The same kind of local cell decomposition can be generalized to such a setting of $\mathcal{P} \subset k \Gamma$.
  We shall call a sphere $S_i \in \mathcal{P}$ an \textit{interior} (resp. boundary) sphere if ${C}(S_i,\mathcal{P})$ has no face belonging to $\partial(k \Gamma)$ (resp. otherwise).
  Set $\mathcal{P}^\circ$ (resp. $\partial \mathcal{P}$) to be the subset of interior (resp. boundary)
  spheres of $\mathcal{P}$.
  However, the corresponding dual decomposition of $k \Gamma$ into D-cells certainly needs some kind of modification.
  Note that the same kind of D-cells construction still applies \textit{up until} those D-cells containing some faces solely spanned by centers of spheres in $\partial \mathcal{P}$.
  Set $R$ to be the union of such D-cells and $\Omega_0$ to be the \textit{complementary} region of $R$ in $k \Gamma$, namely
  \begin{equation}
    \Omega_0 = k \Gamma \setminus R, \quad 
    \partial \Omega_0 = \partial R + \partial (k \Gamma)
  \end{equation}
  which constitute a collar region lying between $\partial R$ and $\partial (k \Gamma)$.  For the sake of simplicity, we shall regard the whole $\Omega_0$ as a {\it single} D-cell and setting
  \begin{equation}
    \rho(\Omega_0) = \frac{\vol \mathcal{P} \cap \Omega_0}{\vol \Omega_0} \quad
    (\mathcal{P} \cap \Omega_0 = \partial \mathcal{P} \cap \Omega_0)
  \end{equation}
  thus completing the D-decomposition of $k\Gamma$ with respect to the given $\mathcal{P} \subset k\Gamma$.
  Using such a pair of L-decomposition and D-decomposition, we shall again define $\bar{\rho}(S_i,\mathcal{P})$, $\bar{\rho}(\mathcal{P}^\circ,\mathcal{P})$, and $\bar{\rho}(\partial \mathcal{P},\mathcal{P})$ by the same kind of weighted average as that of~(\ref{eqn:dcell}) and~(\ref{eqn:rhobar}).
  
  Now, let us proceed to analyze and then to estimate $\rho(\Omega_0)$ which is geometrically a ``total measurement'' of the ``boundary effect'' for $\mathcal{P} \subset k\Gamma$.
  It is quite natural to make use of the almost local flatness of $\partial(k \Gamma)$ to provide the following upper bound estimate on $\rho(\Omega_0)$ via the method of localization.
  
  For the purpose of such a localized upper bound estimate, one may assume without loss of generality that $\partial(k \Gamma)$ is, actually, \textit{locally flat} instead of just locally almost flat.
  Thus, the local geometry of arrangement of $\partial \mathcal{P}$ along $\partial(k \Gamma)$ can be represented by that of arranging spheres on top of a ``table'' (i.e. a piece of plane), and moreover, such local arrangements can also be regarded as the half of their corresponding \textit{reflectionally symmetric} local arrangements, thus enabling us to analyze the localized densities of those D-cells of the latter.
  
  \begin{example} \label{ex:avedense}
    The average density of a star cluster of such D-cells is at most equal to $\slfrac{\pi}{\sqrt {27}} \approx 0.6046$, and it is equal to $\slfrac{\pi}{\sqrt {27}}$ when and only when $S_i$ and its surrounding sextuple
    neighbors forms a close hexagonal cluster of spheres touching the table.
    
    \begin{proof}
      Each D-cell of such a star cluster is an upright triangular prism as indicated in Figure 2.
      
      \begin{figure}
        \begin{center}
          \includegraphics[width = 3.5in]{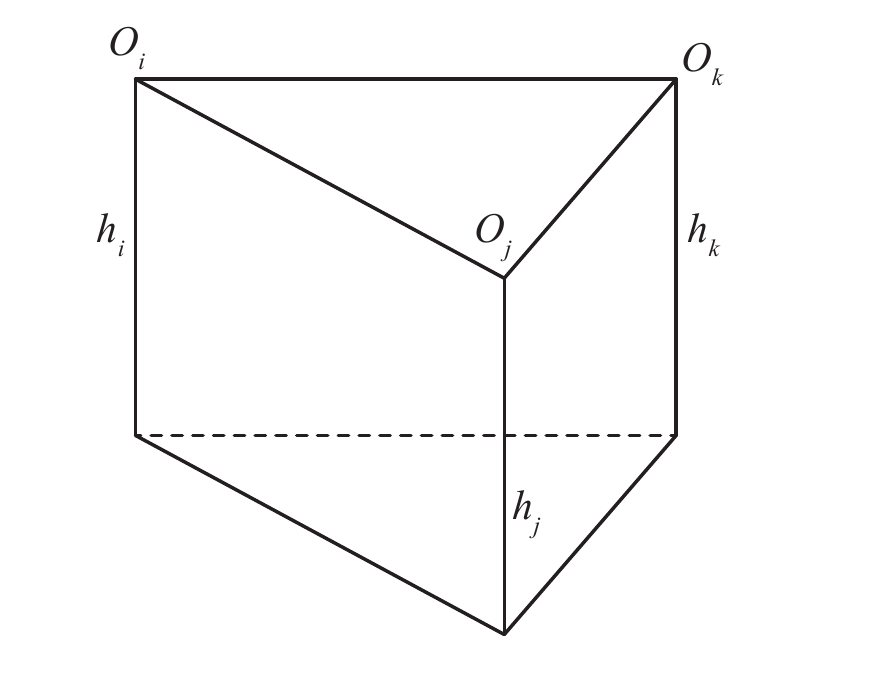}
          \caption{\label{fig:dcell}}
        \end{center}
      \end{figure}
      
      It is easy to show that the density of such a D-cell is at most equal to $\slfrac{\pi}{\sqrt {27}}$ and it is equal to $\slfrac{\pi}{\sqrt {27}}$ when and only when $\{ h_i,h_j,h_k \}$ are all equal to 1 and $\triangle O_i O_j O_k$ is a regular triangle of side length 2.
      
      Therefore, it follows from the above localized density estimate that $\rho(\Omega_0)$ is at most equal to $\slfrac{\pi}{\sqrt {27}}$ and hence, by Theorem I
      \begin{equation}
        \begin{split}
          \bar\rho(S_i,\mathcal{P}) \le \spet \quad \forall S_i \in \mathcal{P} \\
          \rho(\mathcal{P} \subset k \Gamma) \le \spet
        \end{split}
      \end{equation}
      On the other hand, to any $\epsilon > 0$, there exists sufficiently large $k$ such that the volume of the following subregion of $k\Gamma$, namely
      \begin{equation}
        k\Gamma^{(4)} \colonequals \{ x \in k\Gamma, d(x, \partial (k \Gamma)) > 4 \}
      \end{equation}
      already exceeds $(1-\frac{1}{2}\epsilon)$-times of $\vol k \Gamma$.
      Let $\mathcal{P}^\prime$ be the subpacking of hexagonal close packing consisting of those spheres with their centers lying inside of $k\Gamma^{(4)}$.
      Then, it is easy to show that
      \begin{equation}
        \rho(\mathcal{P}^\prime,k\Gamma) > \spet - \epsilon
      \end{equation}
      This proves that $\hat{\rho}(\Gamma)=\spet$.
    \end{proof}
  \end{example}
\end{proof}

%% file: 3-solid-geometry/3-solid-geometry.tex
\section{A concise summary on basics of solid geometry - the geometric invariant theory of the physical space}\label{sec:summary}

Solid geometry studies the properties of the physical space, the space that we and everything else of the universe are situated inside.
In its modern setting and the most effective and advantageous formulation, it is the \textit{geometric invariant theory of the physical space}.
In this section, we shall provide a concise summary on the basics of geometric invariant theory of the physical space (often referred to as the Euclidean 3-space in mathematical terminology), which will supply the fundamental geometric ideas as well as basic useful techniques along our journey of proving Theorem I.

\input{3-solid-geometry/3.1-vector-algebra.tex}

\input{3-solid-geometry/3.2-basic-formulas.tex}

\input{3-solid-geometry/3.3-area-preserving-deformations.tex}

%% file: 3-solid-geometry/3.1-vector-algebra.tex
\subsection{Vector algebra, the basic part of linearizable geometric invariant theory of the space}

The most basic symmetric property of the space is that it is reflectionally symmetric with respect to any given plane $\Pi \subset V$.
The totality of all such reflectional symmetries generates a fundamental transformation group on the space $V$, say denoted by $G(V)$, which is the group of isometries of $V$ (often referred to as the Euclidean group), while the solid geometry studies the invariant theory of this fundamental transformation group.

\begin{enumerate}[(1)]
\item \textbf{The translation subgroup of $G(V)$:} Let $\mathcal{R}_{\Pi}$ be the reflectional symmetry with respect to a given plane $\Pi$ in $V$.
  If $\Pi_1 \parallel \Pi_2$, then the composition $\mathcal{R}_{\Pi_2} \circ \mathcal{R}_{\Pi_1}$ leaves every \textit{common perpendicular} line $\ell$ \textit{invariant} (i.e. mapping onto itself) and pushes its points along $\ell$ by a distance twice of the distance between $\Pi_1$ and $\Pi_2$.
  Thus, it is called a translation.
  It is well-known that the subset of all translations form a commutative subgroup of $G(V)$, say denoted by $T$, which is algebraically isomorphic to $(\mathbb{R}^3, +)$, and moreover, it is an invariant (i.e. normal) subgroup of $G(V)$.
\item \textbf{The orthogonal subgroups of $G(V)$:} Let $p_0$ be a given point in $V$ (or rather, a chosen base point in $V$).
  Then, the subgroup of $G(V)$ generated by the collection of reflectional symmetries with respect to those planes containing $p_0$, namely, $\{\mathcal{R}_\Pi ; p_0 \in \Pi \}$, is exactly the isotropy (i.e. stability) subgroup of $G(V)$ fixing $p_0$, i.e.,
  \begin{equation}
    G_{p_0} \colonequals \{g \in G(V), gp_0=p_0\}
  \end{equation}
  which will be, henceforth, referred to as the \textit{orthogonal subgroup fixing} $p_0$.
\item From the viewpoint of the geometric transformation group of $G(V)$ acting on $V$ as isometries, one has the following generalities, namely:
  \begin{enumerate}[(i)]
  \item The translation group $T \subset G(V)$ acts simple transitively on $V$, while the group $G(V)$, itself, acts simple transitively on the set of all orthonormal frames, say denoted by $\mathcal{F}(V)$, in particular, the subgroup $G_{p_0}$ acts simple transitively on the subset $\mathcal{F}_{p_0}$ of orthonormal frames based at $p_0$. Therefore, $V$ is geometrically a flat homogeneous Riemannian space with $G(V)$ as the isometry group,
    \begin{equation}
      V = \slfrac{G(V)}{G_{p_0}} \quad G_{p_0} \cong O(3)
    \end{equation}
    Algebraically, one has the following diagram of homomorphisms,
    \begin{center}
      \begin{tikzpicture} [
          singlearrow/.style={->, black,fill=none},
          doublearrow/.style={<->,black,fill=none}]
        \node (C) at (0,0) {$G(V)$};
        \node (E) at (2.2,0) {$\slfrac{G(V)}{T}$};
        \node (S) at (0,-2) {$G_{p_{0}}$};
        \node (W) at (-2,0) {$T$};
        \node[font=\LARGE, rotate=90] (S_C) at (-0.1,-1) {$\subset$};
        \node[font=\LARGE] (SE) at (1.4,-1.4) {$\cong$};
        \draw[singlearrow] (W) -- (C);
        \draw[singlearrow] (S) -- (E);
        \draw[singlearrow] (C) -- (E);
        \node[font=\LARGE] at (3.5,0) {$\cong$};
        \node (EE) at (4.5,0) {$O(3)$};
      \end{tikzpicture}
    \end{center}
    and moreover, in terms of modern concept of principal bundles, $\mathcal{F}(V)$ is an important example of principal bundles, namely
    \begin{center}
      \begin{tikzpicture} [
          singlearrow/.style={->, black,fill=none},
          doublearrow/.style={<->,black,fill=none}]
        \node (C) at (0,0) {$G_{p_0}$};
        \node (E) at (2.2,0) {$G(V)$};
        \node (SE) at (2.2,-2) {$V$};
        \node (W) at (-2,0) {$\mathcal{F}_{p_0}$};
        \node[font=\LARGE] at (-1,0) {$\cong$};
        \draw[singlearrow] (C) -- (E);
        \draw[singlearrow] (E) -- (SE);
        \node[font=\LARGE] at (3.5,0) {$\cong$};
        \node (EE) at (4.5,0) {$\mathcal{F}(V)$};
      \end{tikzpicture}
    \end{center}
  \end{enumerate}
\item \textbf{The vector algebra:} Geometrically, a translation $\tau$ is a motion of the whole space $V$ in a given \textit{direction}
  (i.e. the common perpendicular direction of $\Pi_1 \parallel \Pi_2$) by a given \textit{distance} (i.e. $2d(\Pi_1, \Pi_2)$),
  which can be visualized by the equivalence class of directional intervals $\{\overrightarrow{A\tau(A)}\}$,
  thus will be called a vector and denoted by a bold-face lower case latin letter such as $\mathbf{a}, \mathbf{b}$ etc.
  or by exhibiting just one of its equivalence classes of directional intervals such as $\overrightarrow{AB}$.
  Such quantities with both directions and distances will be the most basic kind of geometric quantity which will be, henceforth,
  referred to as \textit{vectors}, and furthermore, we shall change the notations to denote the translation group by $V$, instead of $T$,
  and the commutative group operation as addition ``+''.
  Algebraically, $(V, +)$ is \textit{naturally endowed} with the following three kinds of multiplications, namely
  \begin{enumerate}[(i)]
  \item the scalar multiplication, $\lambda \cdot \mathbf{a}$, which is the algebraic representation of homothety magnification;
  \item the inner product, $\mathbf{a} \cdot \mathbf{b} = \frac{1}{2}\{|\mathbf{a}+\mathbf{b}|^2-|\mathbf{a}|^2-|\mathbf{b}|^2\}$,
    which is a remarkable synthesis of length-angle and the generalized Pythagoras Theorem;
  \item the outer product, $\mathbf{a} \times \mathbf{b}$, which encodes concisely the multi-linearity of the oriented area and the oriented volume.
  \end{enumerate}
  In summary, $(V, +)$ together with the above three kinds of multiplication constitutes a systematic, complete algebraization of the basic geometric structure of the space,
  in which the basic theorems and formulas of quantitative solid geometry have been transformed into powerful {\it distributive laws} of their respective multiplications (i.e. multi-linearity).
  In short, the vector algebra provides a simple, wholesome algebraic system which accomplishes the algebraization as well as linearization of the basic foundation of geometric invariant theory in excellence.
\end{enumerate}
\begin{remarks}
\begin{enumerate}[(i)]
\item  Of course, the geometric invariants of the space can \textit{not} be algebraized entirely.
  For example, the totality of solid angle (i.e. the area of the unit sphere) is equal to $4\pi$, the monumental contribution of Archimedes,
  is a transcendental invariant which can only be understood via integration.
  
\item $(T,+)$ is a commutative invariant subgroup of $G(V)$, thus having the adjoint $G(V)$-action reduces to an orthogonal action of $O(3)\cong \slfrac{G(V)}{T}$, while the above three kinds of products are $SO(3)$-invariant and bilinear.
\end{enumerate}
\end{remarks}

%% file: 3-solid-geometry/3.2-basic-formulas.tex
\subsection{Basic formulas of spherical trigonometry} \label{subsec:basictrig}

The study of spherical trigonometry has a long history, at least dating back to antiquity, due to its importance in quantitative astronomy and in solid geometry.
The proof of the fundamental area formula of spheres (cf. \S 3.2.1) by Archimedes is a monumental achievement of Greek geometry, a glorious milestone in the civilization of rational mind.

\subsubsection{Basic properties and basic theorems of spherical geometry}

A given spherical surface, $S^2(O, R)$, is reflectionally symmetric with respect to those planes, $\Pi \supset \{O\}$, containing its center $O$.
Thus, intrinsically speaking, it is reflectionally symmetric with respect to those great circles $\Pi \cap S^2(O, R)$.
Therefore, the geometry of such a spherical surface has the same kind of reflection symmetries as that of a Euclidean plane, and hence it also has the same kind of congruence conditions of triangles such as SAS, ASA, SSS etc. and also the same kind of ``isosceles triangle theorem'' together with many of its consequences.

Note that two spheres of equal radius are translationally congruent to each other, while two concentric spheres are just homothetically different by a magnification.
Thus, in the study of spherical geometry, it suffices to study the normalized model of the unit sphere $S^2(O,1)$.
First of all, the most important theorem of spherical geometry is the fundamental theorem of Archimedes, namely 

\begin{namedformula}[Archimedes Theorem]
The total area of the unit sphere is equal to $4\pi$.
\end{namedformula}

\begin{corollary}
The area of a spherical triangle $\sigma(ABC)$ is equal to the excess of angle, namely
\begin{equation}
\nm{\sigma(ABC)} = A + B + C - \pi
\end{equation}
which is a kind of localization of the above theorem.
\label{co:aaa}
\end{corollary}

As usual in spherical trigonometry, the angles (resp. side-lengths) of $\sigma(ABC)$ will simply be denoted by $\{A, B, C\}$ (resp. $\{a, b, c\}$).
Let $\mathbf{a} = \overrightarrow{OA}$, $\mathbf{b} = \overrightarrow{OB}$, $\mathbf{c} = \overrightarrow{OC}$. 
Then $\mathbf{a}\cdot\mathbf{b} = \cos c$, etc. Set
\begin{equation} 
\mathbf{b}^{\prime} = \mathbf{b} - (\mathbf{b}\cdot\mathbf{a})\mathbf{a},\ 
\mathbf{c}^{\prime} = \mathbf{c} - (\mathbf{c}\cdot\mathbf{a})\mathbf{a}
\end{equation}
Then, one has
\begin{equation} 
\begin{aligned}
\nm{\mathbf{b}^{\prime}} &= \sin c, \nm{\mathbf{c}^{\prime}} = \sin b \\
D &\equiv \mathbf{a}\cdot(\mathbf{b}\times\mathbf{c}) = \det (\mathbf{a}, \mathbf{b}, \mathbf{c}) = \det (\mathbf{a}, \mathbf{b}^{\prime}, \mathbf{c}^{\prime}) \\
&= \nm{\mathbf{b}^{\prime}} \nm{\mathbf{c}^{\prime}} \sin A 
= \sin c\sin b\sin A \\
\mathbf{b}^{\prime}\cdot\mathbf{c}^{\prime} &= \nm{\mathbf{b}^{\prime}} \nm{\mathbf{c}^{\prime}}\cos A
= \sin c\sin b\cos A \\
\verteq \\
(\mathbf{b} - (\mathbf{b}\cdot\mathbf{a})\mathbf{a}) \cdot (&\mathbf{c} - (\mathbf{c}\cdot\mathbf{a})\mathbf{a}) \\
&= \dotp{b}{c} - (\dotp{b}{a})(\dotp{a}{c}) - (\dotp{c}{a})(\dotp{b}{a}) + (\dotp{b}{a})(\dotp{c}{a})(\dotp{a}{a}) \\
&= \dotp{b}{c} - (\dotp{b}{a})(\dotp{a}{c}) = \cos a-\cos c\cos b
\end{aligned}
\end{equation}
 
Therefore, one has very simple proofs of both the spherical sine law and the spherical cosine law as a straightforward application of vector algebra, namely:

\begin{namedformula}[\textbf{Spherical sine law}] \label{eq:spheresinelaw}
\begin{equation}
\frac{\sin A}{\sin a} 
= \frac{\sin B}{\sin b} 
= \frac{\sin C}{\sin c}
= \frac{D}{\sin a\sin b\sin c}
\end{equation}
\end{namedformula}

\begin{namedformula}[\textbf{Spherical cosine law}] \label{eq:spherecosinelaw}
\begin{equation}
\sin b \sin c \cos A = \cos a - \cos b \cos c, \quad \mathrm{etc}.
\end{equation}
\end{namedformula}

Note that Corollary~\ref{co:aaa} is a rather beautiful area formula of AAA-type.
It would be useful to derive area formulas of SSS-type and SAS-type, similar to the Heron's formula and $\nm{\Delta} = \frac{1}{2}ab\sin C$ in the case of Euclidean triangles,
making use of the above two laws to express $\{\sin A, \cos A, \mathrm{etc}\}$ in terms of $\{\sin a, \cos a, \mathrm{etc}\}$.
This idea naturally leads to the following:

\begin{namedformula}[\textbf{Area formulas of SSS-type and SAS-type}]
\begin{equation}
\tan\frac{\nm{\sigma}}{2} = \frac{D}{u} = \frac{\sin C}{\cot \frac{a}{2} \cot \frac{b}{2} + \cos C}
\end{equation}
where $u=1+\cos a+\cos b+\cos c$.
\end{namedformula}

\begin{proof}
Direct substitution of 
\begin{equation}
\sin A = \frac{D}{\sin b \sin c}, \quad \cos A = \frac{\cos a - \cos b \cos c}{\sin b \sin c}, \quad \mathrm{etc}
\end{equation}
into the well-known formula of 
\begin{equation}
e^{i\nm{\sigma}} = -1 \cdot e^{iA} \cdot e^{iB} \cdot e^{iC}
\label{eqn:exponential}
\end{equation}
and algebraic simplification will show that
\begin{equation}
\sin \nm{\sigma} = \frac{2uD}{u^2 + D^2}, \quad \cos \nm{\sigma} = \frac{u^2 - D^2}{u^2 + D^2}
\label{eqn:sincossigma}
\end{equation}
Therefore, 
\begin{equation}
\begin{aligned}
\tan \frac{\nm{\sigma}}{2} &= \frac{\sin\nm{\sigma}}{1+\cos\nm{\sigma}} = \frac{D}{u} \\
&= \frac{\sin a \sin b \sin C}{(1+\cos a)(1+\cos b)+\sin a \sin b \cos C} \\
&= \frac{\sin C}{\cot\frac{a}{2}\cot\frac{b}{2}+\cos C}.
\end{aligned}
\label{eqn:tansigma}
\end{equation}
\end{proof}
\begin{remark}
  Area is the most fundamental geometric invariant of spherical triangles.
  Thus, the three types of area formulas will play their central roles in the entire geometric invariant theory of spherical trigonometry.
\end{remark}

\subsubsection{Basic geometric invariants of spherical triangles and basic formulas of spherical trigonometry}
First of all, spherical triangles are an important class of basic geometric objects which have quite a few basic geometric invariants, for example side-lengths, angles, area, circumradius, inradius, determinant, etc.
Moreover, in the study of various kinds of problems of solid geometry, such as the sphere packing problem that we treat in this paper, the key geometric invariants that will naturally emerge are often expressible in terms of those basic geometric invariants of spherical triangles, such as the locally averaged density of sphere packings that we discussed in \S 2.
In fact, the important and useful part of the geometric invariant theory of spherical triangles lies in the intricate system  of relations among their rich family of invariants.

\noindent
{\bf Notations and basic geometric invariants of spherical triangles:}

To a given spherical triangle $\sigma(ABC)$, the Euclidean triangle $\Delta ABC$, the 1-isosceles tetrahedron $\tau(\sigma,1)$ with $\{O, A, B, C\}$ as the vertices and the portion of the solid angle cone $\Gamma(\sigma)$ bounded by the tangent planes at $A, B, C$, namely
\begin{equation}
T(\sigma) := \{X \in \Gamma(\sigma), 
\overrightarrow{OX}\cdot\overrightarrow{OA} \leq 1,
\overrightarrow{OX}\cdot\overrightarrow{OB} \leq 1,
\overrightarrow{OX}\cdot\overrightarrow{OC} \leq 1\}
\end{equation}
are a triple of geometric objects in the space canonically associated with it.
Therefore, their geometric invariants should also be considered as that of $\sigma$ itself.

\begin{figure}
  \begin{center}
    \includegraphics[width=3.5in]{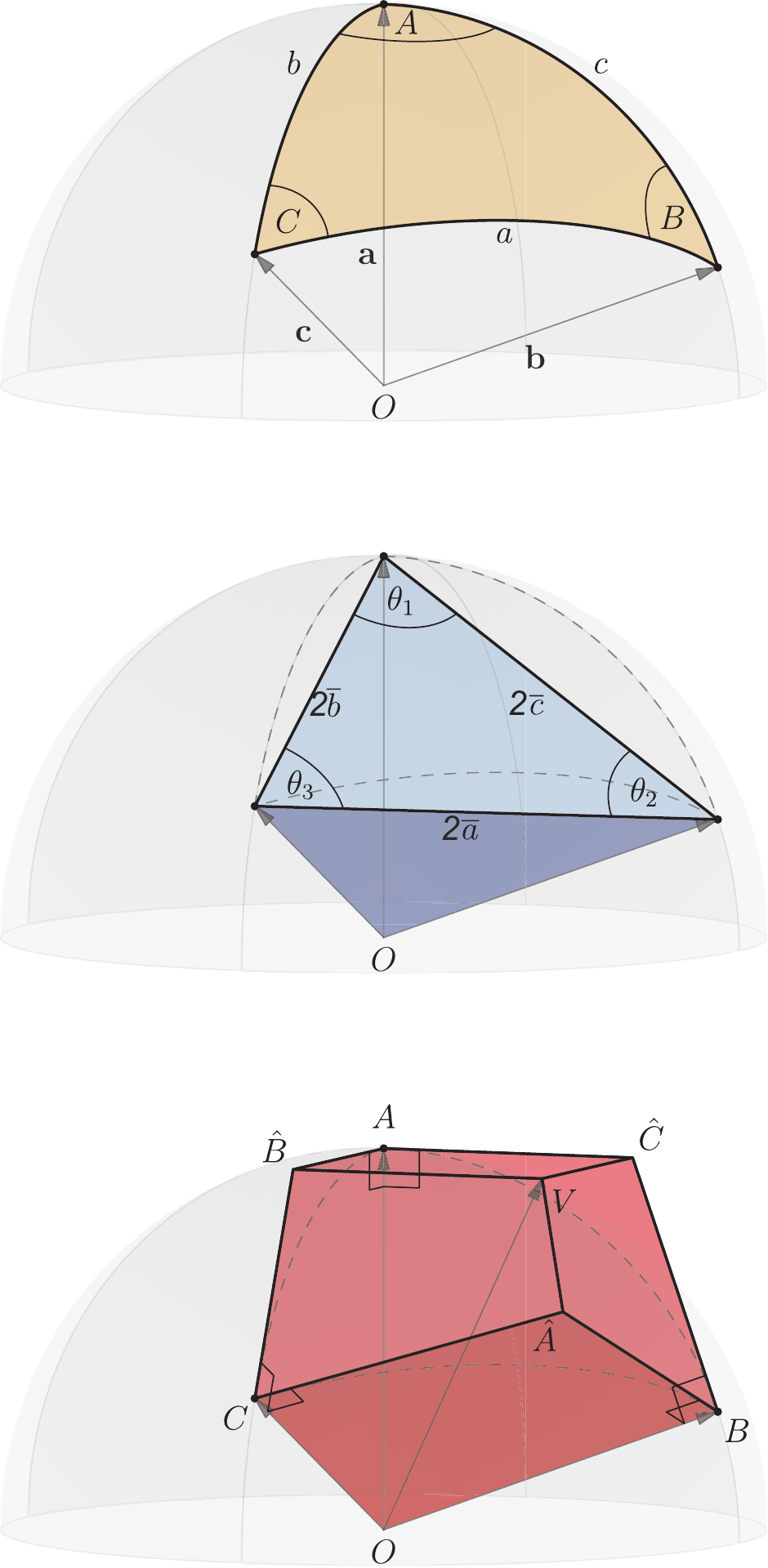}
    \caption{\textit{Top}: The spherical triangle $\sigma(ABC)$ formed from unit vectors $\mathbf{a}, \mathbf{b}, \mathbf{c}$ with spherical angles $A, B, C$ and edge lengths $a, b, c$ (also angles between unit vectors).
      \textit{Middle}: The Euclidean triangle $\triangle ABC$ with angles $\theta_1, \theta_2, \theta_3$ and edge lengths $2\overline{a}, 2\overline{b}, 2\overline{c}$ is a base of the 1-isosceles tetrahedron $\tau(\sigma, 1)$ with $O$ as the vertex.
      \textit{Bottom}: The solid $T(\sigma)$ formed by the solid angle cone bounded by the tangent planes at $A, B, C$.}
    \label{fig:canonical}
  \end{center}
\end{figure}

As indicated in Figure \ref{fig:canonical}, $V = T_A \cap T_B \cap T_C$ and 
\begin{equation}
\overrightarrow{OV} = \frac{1}{D}(\crossp{a}{b} + \crossp{b}{c} + \crossp{c}{a}), 
\quad D = \mathbf{a} \cdot (\crossp{b}{c}) > 0
\end{equation} 
and it passes through the circumcenters $M$ (resp. $\overline{M}$) of $\sigma(ABC)$ (resp. $\Delta ABC$). Thus the circumcentric subdivisions of $\sigma(ABC)$ (resp. $\Delta ABC$) corresponds to each other
under radial projection. 

\noindent
{\bf Notations}: We shall use the following system of notations:
\begin{enumerate}[(i)] 
\item The circumradius of $\sabc$ (resp. $\tabc$) will be denoted by $R$ (resp. $\overline{R}$),
while their inradius will be denoted by $\scr$ (resp. $\overline{\scr}$); the \textit{half}-sidelengths of $\tabc$
will be denoted by $\overline{a}, \overline{b}, \overline{c}$ and the central angles at $M$
(resp. $\overline{M}$) will be denoted by $\{\theta_1, \theta_2, \theta_3\}$ which is in fact equal
to the inner angles of $\tabc$.
\item Set $\{d_1, d_2, d_3\}$ to be the oriented distances between $M$ and its three sides, and
$\{\lambda_1, \lambda_2, \lambda_3\}$ to be the oriented side-angle of the circumcentric subdivision.
\item Set $\{\sigma_A, \sigma_B, \sigma_C\}$ to be the spherical triangles of $\tau(\sigma,1)$
(i.e. solid angles at $\{A, B, C\}$), $\nm{\sigma}$ (resp. $\nm{\Delta}$) to be the areas of $\sigma$ (resp. $\Delta$)
and 
\begin{equation}
\begin{aligned}
\nu(\sigma) &= \nm{\sigma} + \nm{\sigma_A} + \nm{\sigma_B} + \nm{\sigma_C} \\
\rho(\sigma) &= \frac{\frac{1}{3}\nu(\sigma)}{8\vol\tau(\sigma,1)} = \frac{1}{4D}\nu(\sigma)
\end{aligned}
\end{equation}
\end{enumerate}

\noindent
{\bf Spherical trigonometric formulas}:

The intricate system of formulas relating various kinds of geometric invariants of spherical triangles
constitutes a set of important, useful techniques of solid geometry. The following is just a concise
summary of those often useful ones. All of them can be deduced by means of the basic formulas of
\S (3.2.1) or sometimes by a direct application of vector algebra. Mostly, it is their clean-cut
statements and usefulness that's important, interesting and sometimes quite novel.

\begin{enumerate}[(1)]
\item Let us begin with the special case of right-angle spherical triangles (i.e. $C = \frac{\pi}{2}$).
Their trigonometric formulas become particularly simple and hence much easier to use. However, via the
canonical circumcentric subdivision of a general one, such extremely simple formulas can be used to
study that of the general spherical trigonometry.
\begin{enumerate}[(i)]
\item \begin{equation}
\begin{aligned}[t]
\sin C&=1 \quad \mathrm{and} \quad \cos C=0 \\
\Rightarrow \sin A &= \frac{\sin a}{\sin c}, \quad \sin B = \frac{\sin b}{\sin c}, \quad \cos c = \cos a\cdot\cos b
\end{aligned}
\label{eq:basici}
\end{equation}

and moreover,

\item \begin{equation}
\begin{aligned}[t]
\cos A &= \frac{\tan b}{\tan c}, \quad \cos B = \frac{\tan a}{\tan c}, \quad \tan A = \frac{\tan a}{\sin b}, \quad \tan B = \frac{\tan b}{\sin a} \\
\tan A \cdot \tan B &= \sec c
\end{aligned}
\label{eq:basicii}
\end{equation}

and the following special form of area formulas

\item \begin{equation}
\begin{aligned}[t]
\tan \frac{\nm{\sigma}}{2} &= \tan\frac{a}{2}\tan\frac{b}{2}, \quad \sin\nm{\sigma}=\frac{\sin a \sin b}{1 + \cos c} \\
\cos\nm{\sigma} &= \frac{\cos a + \cos b}{1 + \cos c}
\end{aligned}
\label{eq:basiciii}
\end{equation}
\end{enumerate}

\item Trigonometric formulas of geometric invariants of the circumcentric subdivision of $\sabc$ (resp. $\tabc$). 
\begin{enumerate}[(i)]

\item \begin{equation}\begin{aligned}[t]
D &= 6 \vol \tau(\sigma,1) = 2 \nm{\Delta}\cos R, \quad \overline{O\overline{M}} = \cos R \\
\overline{a} &= \sin \frac{a}{2}, \mathrm{etc}. \quad \overline{R} = \sin R, \quad \nm{\Delta} = \frac{2\overline{a}\overline{b}\overline{c}}{\overline{R}} \\
\Rightarrow D &= 4 \sin \frac{a}{2} \sin \frac{b}{2} \sin \frac{c}{2} \cot R, \quad \tan R = \frac{4\overline{a}\overline{b}\overline{c}}{D} \\
\tan^2 R &= \frac{2}{D^2} (1-\cos a)(1-\cos b)(1-\cos c) \\
\end{aligned}\label{eq:tanR}\end{equation}

\item The SSS (resp. AAA and SAS) data of $\sigma_A$, etc. are given as follows, namely
\begin{equation}\begin{aligned}[t]
& \{\frac{1}{2}(\pi - b), \frac{1}{2}(\pi - c), \theta_1\} & \mathrm{for}\ \sigma_A, \mathrm{etc}. \\
\mathrm{(resp.)} \quad & \{A, \frac{\pi}{2} - d_2, \frac{\pi}{2} - d_3\} & \mathrm{for}\ \sigma_A, \mathrm{etc}. \\
& \{\frac{1}{2}(\pi - b), \frac{1}{2}(\pi - c), A\} & \mathrm{for}\ \sigma_A, \mathrm{etc}.
\end{aligned}\end{equation}
and moreover, 
\begin{equation}\begin{aligned}[t]
    \tan d_i &= \cos \theta_i \tan R \\
    \cos \theta_1 &= \frac{\overline{b}^2 + \overline{c}^2 - \overline{a}^2}{2\overline{b}\overline{c}} \\
2 \lambda_1 &= B + C - A = \nm{\sigma} + \pi - 2A, \\
\tan \lambda_1 &= \frac{1}{D} (1 + \cos a - \cos b - \cos c), \mathrm{etc}.
\end{aligned}
\label{eq:tand}\end{equation}

\item \begin{equation}\begin{aligned}
\nu(\sigma) &= \pi + 2 \nm{\sigma} - 2(d_1 + d_2 + d_3), \\
\mathrm{and}\ w(\sigma) :&= \vol\ T(\sigma) = \frac{1}{6}\sin\nm{\sigma}\{\frac{8}{u} - \tan^2 R\} \\
\end{aligned}
\label{eq:volT}
\end{equation}
in the case that $\sigma$ contains its circumcenter. 
\end{enumerate}

Considering the sum of distances between $M$ and the three sides in the first equation of~(\ref{eq:volT}), one has:
\begin{align}
  & \mathrm{(i):} \quad & \tan\left(\sum d_i\right) &= \frac{\sum \tan d_i - \prod \tan d_i}{1 - \displaystyle\sum_{i<j} \tan d_i \tan d_j}
  = \frac{\tan R\sum \cos \theta_i - \tan^3 R\prod \cos \theta_i}{1 - \tan^2 R\displaystyle\sum_{i<j} \cos \theta_i \cos \theta_j} \nonumber \\
  & \mathrm{(ii):} \quad & \sum \cos \theta_i &= \frac{1}{2\overline{a}\overline{b}\overline{c}}\sum\overline{c}(\overline{a}^2 + \overline{b}^2 - \overline{c}^2)
  = \frac{1}{2\prod\overline{a}}\left(\sum\overline{a}^2\overline{b}-\sum\overline{a}^3\right) \nonumber \\
  & & \tan R \sum \cos \theta_i &= \frac{2}{D} \left\{ \sum\overline{a}^2\overline{b}-\sum\overline{a}^3 \right\} \nonumber \\
  & \mathrm{(iii):} \quad & \prod \cos \theta_i &= \frac{1}{8\prod\overline{a}^2} \prod (\overline{a}^2 + \overline{b}^2 - \overline{c}^2)  \\
  & & \tan^3 R \prod \cos \theta_i &= \frac{8}{D^3} \prod \overline{a} \prod \left( \overline{a}^2 + \overline{b}^2 - \overline{c}^2 \right) \nonumber \\
  & & & = \frac{8 \prod\overline{a}}{D^3} \left\{ \sum \overline{a}^4 \overline{b}^2 - \sum \overline{a}^6 - 2 \prod \overline{a}^2 \right\} \nonumber \\
  & \mathrm{(iv):} \quad & \displaystyle \sum_{i<j} \cos \theta_i \cos \theta_j &= \frac{1}{4 \prod \overline{a}^2} \sum \overline{a} \overline{c} \left( \overline{b}^4 - \overline{a}^4 - \overline{c}^4 + 2 \overline{a}^2 \overline{c}^2 \right) \nonumber \\
  & & \tan^2 R \displaystyle \sum_{i<j} \cos \theta_i \cos \theta_j &= \frac{4}{D^2} \left\{ 2 \sum \overline{a}^3 \overline{b}^3 + \sum \overline{a}^4 \overline{b} \overline{c} - \sum \overline{a}^5 \overline{b} \right\} \nonumber 
\end{align}

Substituting (ii), (iii), (iv) into the last equation of (i), one gets
\begin{equation}
  \tan\left(\sum d_i\right) = \frac{1}{D}\cdot \frac{2D^2 \left\{ \sum \overline{a}^2 \overline{b} - \sum \overline{a}^3 \right\} - 8 \prod\overline{a}\left\{ \sum \overline{a}^4 \overline{b}^2 - \sum \overline{a}^6 - 2 \prod \overline{a}^2 \right\}}{D^2 - 4 \left\{ 2 \sum \overline{a}^3 \overline{b}^3 + \sum \overline{a}^4 \overline{b} \overline{c} - \sum \overline{a}^5 \overline{b} \right\}}
\end{equation}

The second equation in~(\ref{eq:volT}) can be proved by the following vector algebra computations, namely:
As indicated in Figure \ref{fig:canonical}, $T(\sigma)$ is the non-overlapping union of a triple of cones
\footnote{In our terminology, a cone can have any flat base; it doesn't need to be circular.} 
with $V$ as their common vertex and with $\Box(OA\hat{C}B)$,
$\Box(OB\hat{A}C)$, $\Box(OC\hat{B}A)$ as their respective bases whose area together
with direction can be represented in terms of vector algebra by:
\begin{equation}
\frac{\mathbf{a}\times\mathbf{b}}{1+\mathbf{a}\cdot\mathbf{b}},\ 
\frac{\mathbf{b}\times\mathbf{c}}{1+\mathbf{b}\cdot\mathbf{c}},\ 
\mathrm{and}\ \frac{\mathbf{c}\times\mathbf{a}}{1+\mathbf{c}\cdot\mathbf{a}} 
\end{equation}
respectively, while $\overrightarrow{OV} = \frac{1}{D}(\mathbf{a}\times\mathbf{b} + \mathbf{b}\times\mathbf{c} + \mathbf{c}\times\mathbf{a})$. Therefore
\begin{equation}\begin{aligned}
\vol T(\sigma) &= \frac{1}{3}\overrightarrow{OV} \cdot \{
\frac{\mathbf{a}\times\mathbf{b}}{1+\mathbf{a}\cdot\mathbf{b}} + 
\frac{\mathbf{b}\times\mathbf{c}}{1+\mathbf{b}\cdot\mathbf{c}} + 
\frac{\mathbf{c}\times\mathbf{a}}{1+\mathbf{c}\cdot\mathbf{a}} \} \\
&= \frac{1}{6}\sin\nm{\sigma}\cdot\{\frac{8}{u} - \tan^2 R\}
\end{aligned}\end{equation}
while the last step is a matter of vector algebraic computations and simplifications.

\item \textit{Half angle formulas and incentric subdivision}:

Set $s = \frac{1}{2}(a + b + c)$. Then, by cosine law,
\begin{equation}
\begin{aligned}
\sin b \sin c \cos^2 \frac{A}{2} &= \frac{1}{2}(\cos a - \cos (b+c)), \\
\sin b \sin c \sin^2 \frac{A}{2} &= \frac{1}{2}(\cos (b-c) - \cos a), \\
\end{aligned}
\end{equation}
Therefore, one has
\begin{equation}
\left\{
\begin{aligned}
\cos\frac{A}{2} &=\sqrt{\frac{\sin s \sin (s-a)}{\sin b \sin c}} \\
\sin\frac{A}{2} &=\sqrt{\frac{\sin (s-b) \sin (s-c)}{\sin b \sin c}} \\
\tan\frac{A}{2} &=\sqrt{\frac{\sin (s-b) \sin (s-c)}{\sin s \sin (s-a)}} \\
\end{aligned}
\right.
\qquad \text{etc.}
\end{equation}
and via the geometry of incentric subdivision,
\begin{equation}\begin{aligned}
\tan \scr &= \sin (s-a) \tan \frac{A}{2}
&= \sqrt{\frac{\sin(s-a)\sin(s-b)\sin(s-c)}{\sin s}}
\end{aligned}\end{equation}
\end{enumerate}
\textit{Remark}: Roughly speaking, there are two types of invariants of spherical
triangles, namely, those partial individual ones such as lengths, angles, $\theta_i$,
and $\lambda_i$ etc. and those wholesome and symmetric ones such as area
$\nm{\sigma}$, $u$, $D$, $R$, $\scr$ etc. The simplest and also the most basic 
invariants of individual type are $\{a, b, c\}$ or 
$\{\overline{a}, \overline{b}, \overline{c}\}$ or 
$\{\cos a, \cos b, \cos c\}$, while $\nm{\sigma}$, $u$, $D$, and $\tan R$ naturally
emerge as those most important wholesome invariants.
\begin{example}
  Let $\sigma_\theta$ be the $\slfrac{\pi}{3}$-isosceles triangle with $\theta$ as its top angle, $\alpha_0 \le \theta \le \pi - \alpha_0$, (cf. Figure~\ref{fig:lune}).
  Then
  \begin{equation}
    \begin{aligned}
      D &= \frac{3}{4}\sin\theta, \quad u=2+\frac{3}{4}\cos \theta+\frac{1}{4} = \frac{3}{4}(3+\cos\theta) \\
      |\sigma_\theta| &= 2\arctan\frac{\sin\theta}{3+\cos\theta}, \quad \tan^2 R = \frac{1}{3}\sec^2\frac{\theta}{2}=\frac{2}{3(1+\cos\theta)}\\
      \nu(\sigma_\theta) &= \pi + 2|\sigma_\theta| - 4\arctan\left(\frac{1}{2}\tan\frac{\theta}{2}\right)-2\arctan\left(\frac{3\cos\theta+1}{4\sqrt{3}\cos\frac{\theta}{2}}\right)\\
      \rho(\tau(\sigma_\theta,2)) &= \frac{1}{3\sin\theta}\nu(\sigma_\theta)\\
      w(\sigma_\theta) &= \vol(T(\sigma_\theta))=\frac{1}{6}\sin|\sigma_\theta|\left\{\frac{8}{u}-\tan^2R\right\}=\frac{1}{9}\tan^2\frac{\theta}{2}\frac{13+15\cos\theta}{5+3\cos\theta}
    \end{aligned} \label{eq:sigisos}
  \end{equation}
\end{example}

\begin{figure}
  \begin{center}
    \begin{tikzpicture}
      \node[anchor=south west,inner sep=0] (image) at (0,0) {\includegraphics[width=3.5in]{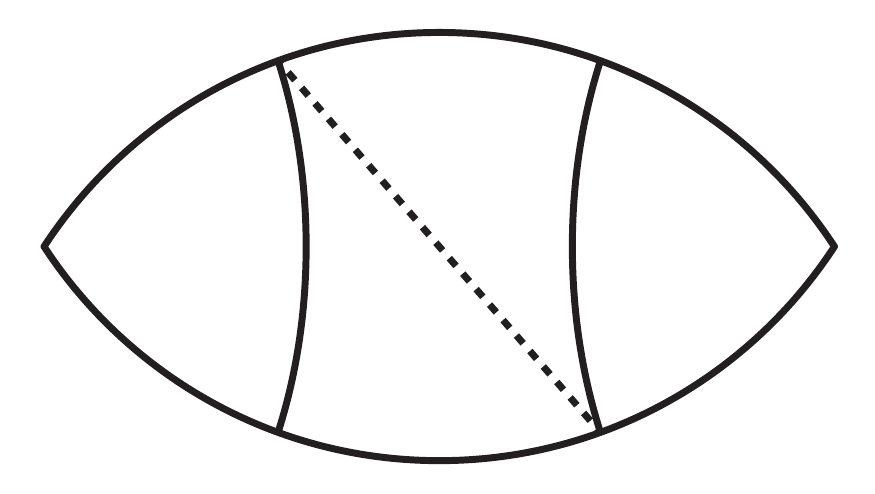}};
      \begin{scope}[x={(image.south east)},y={(image.north west)}]
        \node at (0,0.5) {$N$};
        \node at (0.1,0.5) {$\theta$};
        \node at (0.9,0.5) {$\theta$};
        \node at (1,0.5) {$S$};
        \node at (0.25,0.6) {$\sigma_\theta$};
        \node at (0.75,0.6) {$\sigma_\theta$};
        \node at (0.45,0.3) {$\tilde{\sigma}_\theta$};
        \node at (0.55,0.7) {$\tilde{\sigma}_\theta$};
      \end{scope}
    \end{tikzpicture}
    \caption{\label{fig:lune}}
  \end{center}
\end{figure}

\begin{example} \label{eg3.2}
  Let $\widetilde{\sigma_\theta}$, $\alpha_0 \le \theta \le \frac{\pi}{2}$, be half of the spherical rectangle as indicated in Figure~\ref{fig:lune} (i.e. $\Box(\frac{\pi}{3}, b),\ \cos b = \frac{3}{4}\cos\theta + \frac{1}{4}$). Then
  \begin{align}
      D &= \frac{3}{4}\sin \theta, \quad u = 2(\cos b + \frac{1}{2}) = \frac{3}{2}(1+\cos\theta) \nonumber \\
      |\widetilde{\sigma_\theta}| &= 2\arctan\frac{D}{u}=2\arctan\left(\frac{1}{2}\tan\frac{\theta}{2}\right) \nonumber \\
      d_1 &= \frac{\theta}{2}, \quad \cos R = \frac{\sqrt{3}}{2}\cos\frac{\theta}{2} = \cos d_2\cos\frac{b}{2} \nonumber \\
      \cos d_2 &=\sqrt{\frac{3(1+\cos\theta)}{5+3\cos\theta}}, \quad \tan d_2=\frac{1}{\sqrt{3}\cos\frac{\theta}{2}} \nonumber \\
      \tan^2 R &= \sec^2 R-1 = \frac{5-3\cos\theta}{3(1+\cos\theta)} \nonumber \\
      \rho(\tau(\widetilde{\sigma_\theta},2)) &= \frac{1}{3\sin\theta}\left\{\pi + 4 \arctan \left( \frac{1}{2}\tan\frac{\theta}{2}\right) - \theta - 2 \arccos \sqrt{\frac{3+3\cos\theta}{5+3\cos\theta}}\right\} \label{eq:sigtisos} \\ 
      \vol(T(\widetilde{\sigma_\theta})) &= \frac{1}{6}\sin|\widetilde{\sigma_\theta}|\left\{\frac{8}{u}-\tan^2R\right\}=\frac{2}{9}\tan\frac{\theta}{2}\frac{11+3\cos\theta}{5+3\cos\theta} \nonumber \\
      H&=\sec R - 2\cos R = \frac{2}{\sqrt{3}\cos\frac{\theta}{2}}-\sqrt{3}\cos\frac{\theta}{2}=\frac{1-3\cos\theta}{\sqrt{6(1+\cos\theta)}} \nonumber \\
      \vol(\mathrm{truncated\ tip}) &= \frac{1}{3}H^3\cot^2R\left\{\frac{1}{\sqrt{3}\sin\frac{\theta}{2}}+\sqrt{3}\sin\frac{\theta}{2}\right\} = \frac{1}{36}\frac{(1-3\cos\theta)^3}{\sin\theta} \nonumber \\
      w(\widetilde{\sigma_\theta})&=\vol(T(\widetilde{\sigma_\theta}))-\vol(\mathrm{truncated\ tip})=\frac{2}{9}\tan\frac{\theta}{2}\frac{11+3\cos\theta}{5+3\cos\theta}-\frac{1}{36}\frac{(1-3\cos\theta)^3}{\sin\theta} \nonumber 
  \end{align}
\end{example}

\begin{figure}
  \begin{center}
    \includegraphics[width=3.5in]{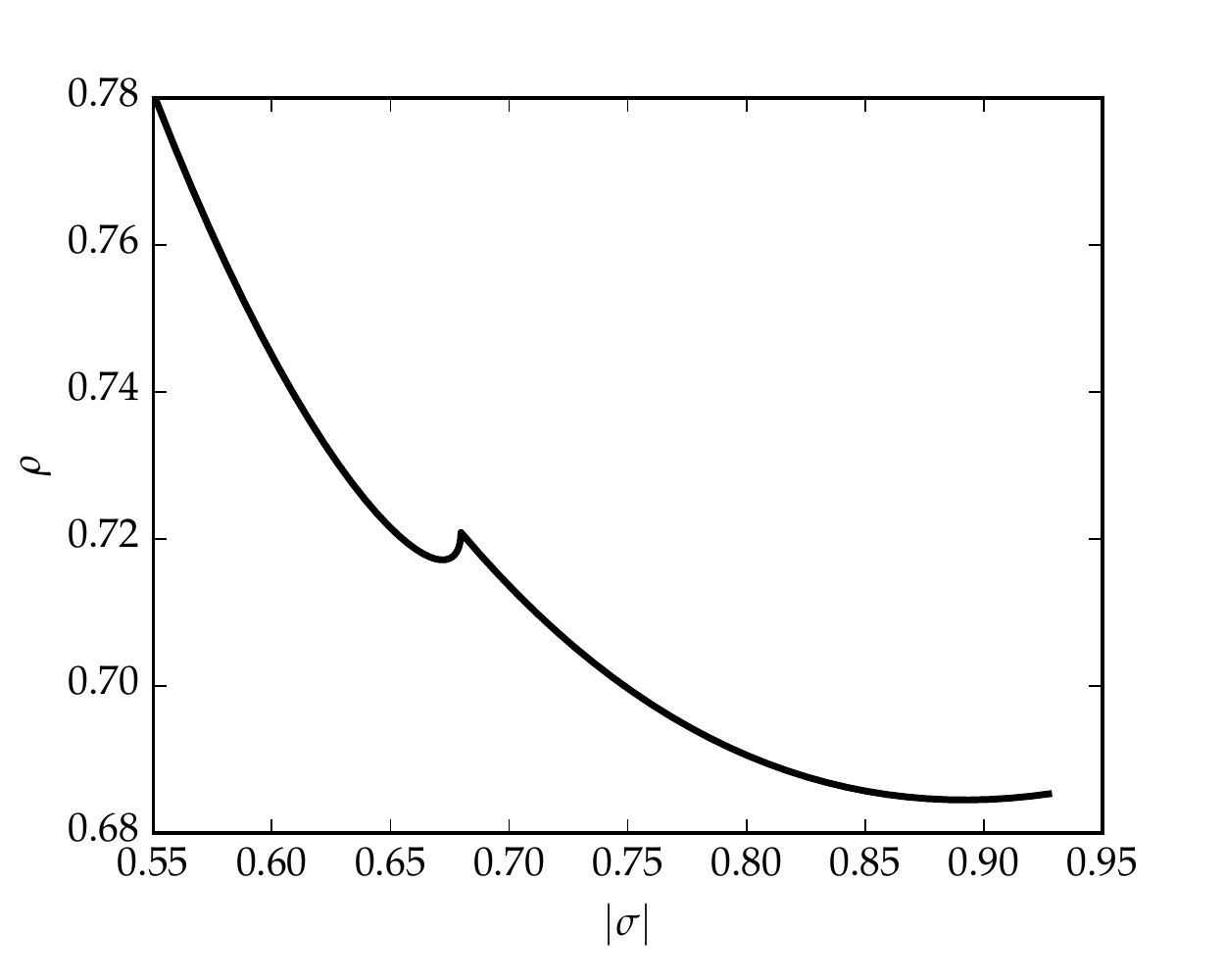}
    \caption{The combined graph of $\rho(\tau(\sigma_\theta, 2))$ (resp. $\rho(\tau(\widetilde{\sigma_\theta},2))$) as functions of $|\sigma_\theta|$ (resp. $|\widetilde{\sigma_\theta}|$)}
    \label{fig:rhosigma}
  \end{center}
\end{figure}

\subsubsection{Geometric invariants of spherical quadrilaterals}

Geometrically, a spherical quadrilateral $\sigma(ABCD)$ can be subdivided into a pair of spherical triangles by one of its diagonals. Therefore its geometric invariants can always be expressed in terms of that of its pair of triangles, thus expressible in terms of the basic invariants of such a pair of triangles with a common edge, in particular, the cosine of the other diagonal. Algebraically, let $\{\mathbf{a,b,c,d}\}$ be the quadruple of unit vectors of the vertices of a given spherical quadrilateral $\sigma(ABCD)$. The sextuple of cross inner products (i.e. the cosines of side-lengths and diagonal lengths) of course consists of a complete set of congruence invariants but with one functional relation, thus making any quintuple subset already consisting of a complete set of congruence invariants. In fact, this is exactly the fundamental result on the relations among $n$ unit vectors for $n \geq 4$. Anyhow, it will be a useful tool in the analysis of spherical configurations to have a kind of simple algebraic formula to express any one of the above sextuple of cross inner products in terms of the other quintuple.

\noindent
{\bf A simple method and an advantageous relation for basic invariants of spherical quadrilaterals}

Suppose $\mathbf{c} \cdot \mathbf{d}$ is the one that we would like to compute in terms of the other quintuple of cross inner product of $\{\mathbf{a,b,c,d}\}$. Then $\sigma(ABC)$ and $\sigma(ABD)$ are a pair of spherical triangles with $\overline{AB}$ as their common edge, while their orientations may be the same or opposite to each other. Set $D_1$ and $D_2$ to be their determinants. Then, obviously

\begin{equation} \label{eq:d1d2}
D_1 \cdot D_2 = \pm \sqrt{D_1^2 \cdot D_2^2}
\end{equation}

\noindent
while

\begin{align*}
D_1 \cdot D_2 &=
\begin{vmatrix}
1 & \mathbf{b}\cdot\mathbf{a} & \mathbf{c}\cdot\mathbf{a} \\ 
\mathbf{a}\cdot\mathbf{b} & 1 & \mathbf{c}\cdot \mathbf{b} \\ 
\mathbf{a}\cdot\mathbf{d} & \mathbf{b}\cdot\mathbf{d} & \mathbf{c}\cdot\mathbf{d} 
\end{vmatrix}
\end{align*}

\begin{align}\label{eq:dmatrix}
D_1^2 &=
\begin{vmatrix}
1 & \mathbf{b}\cdot\mathbf{a} & \mathbf{c}\cdot\mathbf{a} \\ 
\mathbf{a}\cdot\mathbf{b} & 1 & \mathbf{c}\cdot \mathbf{b} \\ 
\mathbf{a}\cdot\mathbf{c} & \mathbf{b}\cdot\mathbf{c} & 1
\end{vmatrix}
\end{align}

\begin{align*}
D_2^2 &=
\begin{vmatrix}
1 & \mathbf{b}\cdot\mathbf{a} & \mathbf{d}\cdot\mathbf{a} \\ 
\mathbf{a}\cdot\mathbf{b} & 1 & \mathbf{d}\cdot \mathbf{b} \\ 
\mathbf{a}\cdot\mathbf{d} & \mathbf{b}\cdot\mathbf{d} & 1 
\end{vmatrix}
\end{align*}

Therefore, it is quite simple to use the above equations to solve $\mathbf{c}\cdot\mathbf{d}$ in terms of the others.

%% file: 3-solid-geometry/3.3-area-preserving-deformations.tex
\subsection{Area estimates and area preserving deformations} \label{subsec:areadeform}

\subsubsection{Some corollaries of the area formulas}

In the study of spherical geometry, the area is the most important invariant, while the Archimedes Theorem and the area formulas
(i.e. A.A.A., S.S.S. and S.A.S.) of triangles are the fundamental theorems and powerful tools. In this subsection, we shall derive
some useful corollaries of the area formulas.
\begin{corollary}
For $\sigma(ABC)$ with given and fixed $\{a, b\}$,
\begin{equation}
\nm{\sigma} \leq 2 \arctan \left\{ \frac{(1-\cos a)(1-\cos b)}{2(\cos a + \cos b)} \right\}^{\frac{1}{2}}
\end{equation}
and equality holds when and only when $c=2R$.
\end{corollary}
\begin{proof}
Set $t = \tan \frac{C}{2}$ and $k = \cot\frac{a}{2}\cot\frac{b}{2}$. Then,
\begin{equation}
  (k-1)t^2 - 2 \cot \frac{\nm{\sigma}}{2} t + (k+1) = 0
  \label{eq:kquad}
\end{equation}
and the above inequality follows directly from the realness of $t$. Moreover, the pair of roots of the above quadratic equation actually
correspond to the pair of opposite angles $C_1, C_2$ of the centrosymmetric quadrilateral with side lengths $\{a, b, a, b\}$, namely,
the pair of roots $t_i = \tan\frac{C_i}{2}$. Thus, the equality holds when and only when $t_1 = t_2$ and the intersection point
of the pair of diagonals is actually the circumcenter.
\end{proof}
\begin{remark}
It follows from the S.A.S. formula of $\tan \frac{\nm{\sigma}}{2}$ that
\begin{equation}
\begin{aligned}
& \sin \frac{\nm{\sigma}}{2}(k+\cos C) = \cos\frac{\nm{\sigma}}{2} \sin C \\
\Rightarrow & \sin \frac{\nm{\sigma}}{2} \cdot k = \sin(C - \frac{\nm{\sigma}}{2}) \\
\Rightarrow & \ C = \arcsin(k \cdot \sin\frac{\nm{\sigma}}{2}) + \frac{\nm{\sigma}}{2}
\label{eq:cofk}
\end{aligned}
\end{equation}
\end{remark}
\begin{corollary} \label{cor2}
Consider $\nm{\sigma}$ as a function of $\{a, b, C\}$ (i.e. the S.A.S. data of $\sigma(ABC)$). (resp. $\{a, b, 1+\cos c\}$). Then
\begin{equation}\begin{aligned}
\frac{\partial \nm{\sigma}}{\partial C} &= \frac{1+\cos c-\cos a-\cos b}{1+\cos c} \\
\frac{\partial \nm{\sigma}}{\partial x} &= \frac{\cos a + \cos b - x}{xD}
\end{aligned}\end{equation}
where $x = 1 + \cos c$.
\end{corollary}
\begin{proof}
Set $x = 1 + \cos c$ and fixed $a, b$ (i.e. regarding $c_1 = \cos a$ and $c_2 = \cos b$ as constants). Then
\begin{equation}\begin{aligned}
u &= c_1 + c_2 + x, \quad D^2 = -(c_1 + c_2)^2 + 2(c_1 c_2 + 1) x - x^2 \\
\nm{\sigma} &= 2 \arctan \frac{D}{u}
\end{aligned}\end{equation}
Therefore, by differentiation w.r.t $x$, one has 
\begin{equation}
\frac{d\nm{\sigma}}{dx} = \frac{uD^{\prime}-D}{\frac{1}{2}(u^2+D^2)} = \frac{uDD^{\prime}-D^2}{(1+c_1)(1+c_2)xD} = \frac{c_1+c_2-x}{xD}
\end{equation}
while the differentiation of the cosine law gives that
\begin{equation}
-\sin a \sin b \sin C \frac{dC}{dx} = 1 \Rightarrow \frac{dx}{dC} = -D
\end{equation}
thus having
\begin{equation}
\frac{d\nm{\sigma}}{dC} = \frac{d\nm{\sigma}}{dx}\cdot\frac{dx}{dC} = \frac{x - c_1 - c_2}{x}.
\end{equation}
\end{proof}

\begin{corollary} \label{cor3}
The areas of spherical quadrilaterals with $\{\ell_i ; i=1,2,3,4\}$ as given side-lengths have a unique maximum at the cocircular one.
Set $c_i = \cos \ell_i$ and $x=1+\cos d, d=\overline{AC}$ being the cutting diagonal. Then that of the cocircular one is given by
\begin{equation}
x_0 = \frac{(c_1+c_2)\sqrt{(1-c_3)(1-c_4)}+(c_3+c_4)\sqrt{(1-c_1)(1-c_2)}}{\sqrt{(1-c_1)(1-c_2)} + \sqrt{(1-c_3)(1-c_4)}}
\end{equation}
\end{corollary}
\begin{proof}
Set $\sigma_1$ (resp. $\sigma_2$) to be the triangles with $\{\ell_1, \ell_2, d\}$ (resp. $\{\ell_3, \ell_4, d\}$) as their side-lengths
and $D_1, u_1$ (resp. $D_2, u_2$) to be that of $\sigma_1, \sigma_2$. 

Then the area, $A(x)$, of such quadrilaterals are given by:
\begin{equation}
A(x) = 2(\arctan\frac{D_1}{u_1} + \arctan\frac{D_2}{u_2})
\end{equation}
and by Corollary~\ref{cor2}
\begin{equation}
A^{\prime}(x) = \frac{c_1+c_2-x}{xD_1} + \frac{c_3+c_4-x}{xD_2}
\end{equation}
On the other hand, set $T_A, T_B, T_C$ and $T_D$ to be the tangent planes at vertices $\{A, B, C, D\}$, $L_1 = T_A \cap T_C$
and $V_1 = L_1 \cap T_B$, $V_2 = L_1 \cap T_D$. Then, it is not difficult to show that
\begin{equation}\begin{aligned}
\overrightarrow{OV_1} &= \frac{1}{1+\dotp{a}{c}}(\mathbf{a}+\mathbf{c}) + k_1(\crossp{a}{c}), \quad \overrightarrow{OV_1} \cdot \mathbf{b} = 1 \\
\overrightarrow{OV_2} &= \frac{1}{1+\dotp{a}{c}}(\mathbf{a}+\mathbf{c}) + k_2(\crossp{a}{c}), \quad \overrightarrow{OV_2} \cdot \mathbf{d} = 1
\end{aligned}\end{equation}
thus having
\begin{equation}\begin{aligned}
k_1 &= \frac{c_1+c_2-x}{xD_1}, \quad k_2=\frac{x-c_3-c_4}{xD_2} \\
\overrightarrow{V_1V_2} &= (k_1 - k_2)(\crossp{a}{c}) = A^{\prime}(x)(\crossp{a}{c})
\end{aligned}\end{equation}
Therefore $A^{\prime}(x) = 0$ when and only when $V_1 = V_2$, $\{A, B, C, D\}$ cocircular. Set $x_0$ to be the unique solution of $A^{\prime}(x_0) = 0$. Then
\begin{equation}\begin{aligned}
& \frac{c_1+c_2-x_0}{x_0D_1} = \frac{x_0-c_3-c_4}{x_0D_2} \\
& \rightarrow \frac{c_1+c_2-x_0}{x_0-c_3-c_4} = \frac{D_1}{D_2} (:= \lambda) \\
& \lambda^2 = \frac{(c_1+c_2-x_0)^2+D_1^2}{(c_3+c_4-x_0)^2+D_2^2} = \frac{2x_0(1-c_1)(1-c_2)}{2x_0(1-c_3)(1-c_4)} \\
& \lambda = \frac{\sqrt{(1-c_1)(1-c_2)}}{\sqrt{(1-c_3)(1-c_4)}}
\end{aligned}\end{equation}
\end{proof}
Recall that, in the case of plane geometry, the well known S.S.S. area formula of triangles has a beautiful generalization to that of cocircular
quadrilaterals, namely 
\begin{equation}
A(\ell_1 ... \ell_4) = \prod(\ess - \ell_i)^{\nicefrac{1}{2}}, \quad \ess = \frac{1}{2} \sum \ell_i
\end{equation}
Therefore, it is interesting to seek a version of the above formula in the realm of absolute geometry. Thus
\begin{corollary}
Let $A(\ell_1 ... \ell_4)$ be the cocircular spherical quadrilateral with $\{\ell_i \}$ as its side-lengths. Set
\begin{equation}
\begin{aligned}
\ess_i &= \sin \frac{\ell_i}{2} \\
S &= -4\prod \ess_i \sum \ess_i^2 - 4 \sum \ess_i^2 \ess_j^2 \ess_k^2 + 2 \sum \ess_i^2 \ess_j^2 + 8 \prod \ess_i
\end{aligned}
\end{equation}
Then
\begin{equation}
A(\ell_1 ...\ell_4) = 2 \arctan \left\{ \frac{\sqrt{S}}{2(1-\frac{1}{2}\sum \ess_i^2 - \prod \ess_i)} \right\}
\end{equation}
\end{corollary}
\begin{proof}
By Corollary~\ref{cor3},
\begin{equation}
A(\ell_1 ...\ell_4) = 2 \arctan \left( \frac{u_2D_1 + u_1D_2}{u_1u_2 - D_1D_2} \right)
\end{equation}
in which $x_0$ is given by the specific formula of Corollary~\ref{cor3}. It is an interesting computation of trigonometric algebra that
\begin{equation}\begin{aligned}
\frac{u_2D_1 + u_1D_2}{u_1u_2 - D_1D_2} &= \frac{(1+\lambda)D_2}{4(1-\frac{1}{2}\sum\ess_i^2-\prod\ess_i)} \\
(1+\lambda)D_2 &= D_1 + D_2 = 2\sqrt{S}
\end{aligned}\end{equation}
\end{proof}
\begin{corollary}
Let $\sigma$ (resp. $\rho$) be a spherical triangle (resp. cocircular polygon) containing its circumcenter with side-lengths at
least equal to $a$. Then its area $\sigma$ (resp. $\rho$) is at least equal to that of an equilateral one, namely
\begin{equation}
\nm{\sigma} \geq \nm{\sigma_a} \quad \left(\mathrm{resp.} \nm{\rho} \geq \nm{\rho_a} \right).
\end{equation}
\end{corollary}

\subsubsection[Hi1]{Area-preserving deformation and ($k, \delta$)-representation of the $|\sigma|$-level surface}
\label{subsubsec:areapreserve}
For fixed $\{C, k\}$ (or equivalently $\{|\sigma|, k\}$) the family of such triangles with side-lengths ordered as $a_1 \leq  a_2 \leq a_3$ and parametrized by $k \equiv \cot \frac{a_1}{2} \cot \frac{a_2}{2}, \delta \equiv \cot \frac{a_1}{2} - \cot \frac{a_2}{2} \ge 0$ constitutes a basic type of area-preserving deformations, characterized by the property of also fixing its largest angle $C$.
Geometrically, the congruence classes of spherical triangles with a given area $|\sigma|$ constitutes a 2-dimensional subset of the \textit{moduli-space of congruence classes}, which will be referred to as a \textit{$|\sigma|$-level surface}.
Note that the area $|\sigma|$ is naturally the most important, basic geometric invariant, while \textit{area-wise estimates} of various kinds of geometric invariants such as $\rho(\tau(\sigma, 2))$, $\vol(T(\sigma))$ etc., and the geometry of various kinds of area-preserving deformations naturally constitutes a useful system of basic techniques of solid geometry.
Moreover, it will be advantageous to provide a suitable organization of simple kinds of area-preserving deformations such as the above one \textit{fixing an angle} and the Lexell's deformations fixing a side length (cf. Example 2.1.3, p 59 \cite{hsiang}).

\noindent
\textbf{The $(k, \delta)$-representation of $|\sigma|$-level surface}

Note that $\{|\sigma|, k, \delta\}$ already constitutes a complete set of congruence invariants for the family of spherical triangles containing their circumcenters and with edge-lengths of at least $\slfrac{\pi}{3}$.
For the purpose of this paper, it suffices to consider the range of $|\sigma|$ up to 0.97.

For a given value of $|\sigma| \in [2 \arctan \slfrac{\sqrt{2}}{5}, 0.97 ]$, it is convenient to parametrize the $|\sigma|$-level surface by $(k, \delta)$, thus representing it as a domain in the $(k, \delta)$-plane.
It is natural to subdivide into two cases, namely:
\begin{itemize}
\item \textbf{Case 1}: $|\sigma| \leq 2 \arcsin \frac{1}{3} = \frac{1}{2}\Box_{\slfrac{\pi}{3}}$,
\item \textbf{Case 2}: $|\sigma| > 2 \arcsin \frac{1}{3}$ and at most equal to 0.97
\end{itemize}

\noindent
\textbf{Case 1:} $|\sigma| \in [2 \arctan \slfrac{\sqrt{2}}{5}, 2 \arcsin \frac{1}{3}]$

Note that the special case of $|\sigma| = 2\arctan \slfrac{\sqrt{2}}{5}$ only consists of a single point (i.e. $(3,0)$).
Thus, we shall assume that $|\sigma| > 2\arctan \slfrac{\sqrt{2}}{5}$ and at most equal to $2 \arcsin \frac{1}{3} = \frac{1}{2}\Box_{\slfrac{\pi}{3}}$.
For such a given $\nm{\sigma}$, there exists a unique $\slfrac{\pi}{3}$-isosceles (resp. equilateral) triangle with their areas equal to $\nm{\sigma}$, and moreover, a continuous family of isosceles triangles of area $|\sigma|$ linking them in between, namely, with $\{(k, 0), 3 \ge k \ge k_0\}$ as their $(k, \delta)$-coordinates where
\begin{equation}
  k_0 = \frac{1}{1-2\cos\frac{1}{3}(\pi+|\sigma|)}
\end{equation}

Now, beginning with such an isosceles triangle, say denoted by $\overline{\sigma}(a, C)$, with $a = 2 \arctan \slfrac{1}{\sqrt{k}}$ and $C$ given by~(\ref{eq:cofk}), one has the area-preserving deformation keeping the $(k, C)$ fixed, while increasing $\delta$ up until either its shortest side-length already reaches the lower-bound of $\slfrac{\pi}{3}$, or it becomes an isosceles triangle with $C$ as its base angles.
Therefore, the domain of $(k, \delta)$ representation of such a $|\sigma|$-level surface is as indicated in Figure 6-(i), where $(k_1, \delta_1)$ represents the unique $c_1$-isosceles triangle with the given area $|\sigma|$ and $\slfrac{\pi}{3}$-base, namely
\begin{equation}
  \begin{aligned}
    c_1 &= \arccos \left\{ \frac{\cos\frac{|\sigma|}{2} - \frac{\sqrt{3}}{2}}{\frac{2}{\sqrt{3}} - \cos\frac{|\sigma|}{2}} \right\}
    \quad (\hbox{i.e.} \ \cos\frac{|\sigma|}{2} = \frac{\frac{\sqrt{3}}{2} + \frac{2}{\sqrt{3}}\cos c_1}{1+\cos c_1}) \\
    k_1 &= \sqrt{3} \cot \frac{c_1}{2} = \sqrt{3}(7-4\sqrt{3}\cos\frac{|\sigma|}{2})^{-\frac{1}{2}}, \quad \delta_1=\sqrt{3}-\frac{k_1}{\sqrt{3}}
  \end{aligned}
\end{equation}

\noindent
\textbf{Case 2:} $2 \arcsin \frac{1}{3} < |\sigma| \le 0.97$

Set $\hat{a}$ to be the side-length of the spherical square with its half area $|\sigma| > \frac{1}{2}\Box_{\slfrac{\pi}{3}}$, namely
\begin{equation}
  \tan\frac{|\sigma|}{2} = \frac{1-\cos\hat{a}}{2\sqrt{\cos\hat{a}}}, \quad \hat{k} = \cot^2\frac{\hat{a}}{2} = \csc\frac{|\sigma|}{2}
\end{equation}

Therefore, the only difference between Case 2 and Case 1 is that $k$ is bounded above by $\hat{k} < 3$.
Thus, the domain of $(k, \delta)$ representation for Case 2 is as indicated in Figure 6-(ii).

\begin{figure}
  \begin{center}
    \begin{tikzpicture}
      \node[anchor=south west,inner sep=0] (image) at (0,0) {\includegraphics[width=6in]{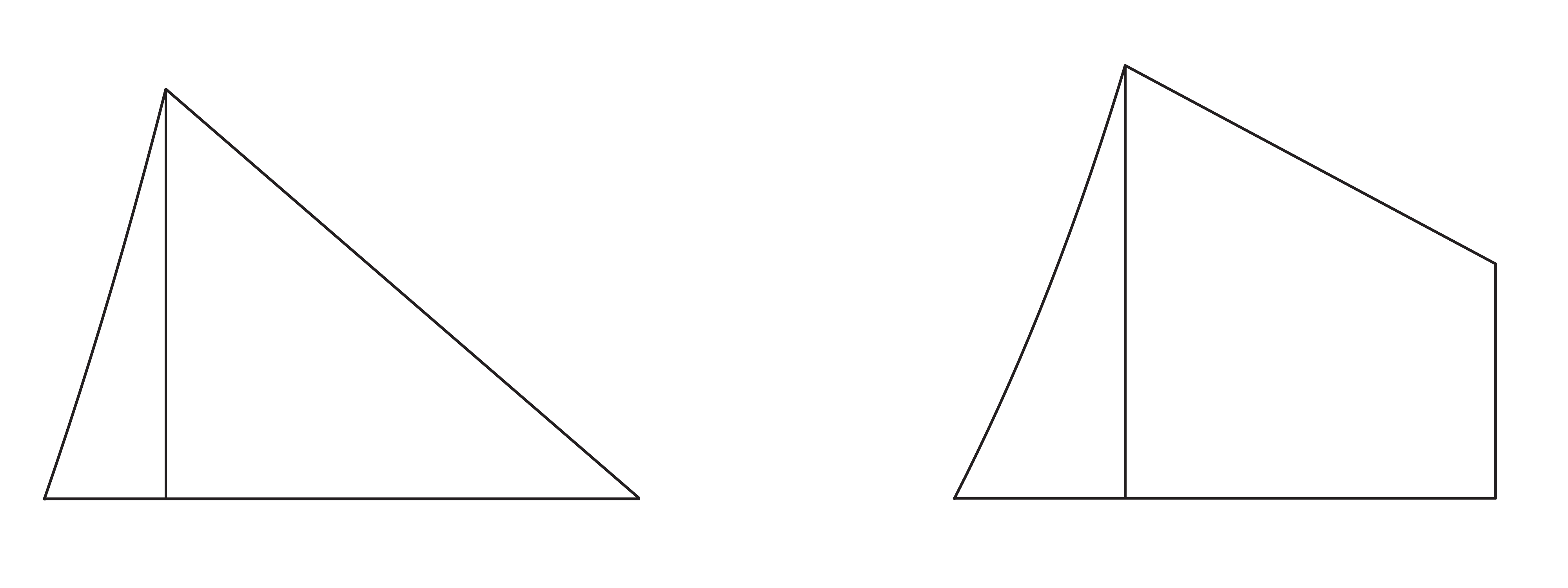}};
      \begin{scope}[x={(image.south east)},y={(image.north west)}]
        \node at (0.01,-0.05) {(i)};
        \node at (0.01,0.10) {$k_0 = 2.41$};
        \node at (0.11,0.10) {$k_1$};
        \node at (0.20,-0.05) {$|\sigma| = \frac{1}{2}\Box_{\slfrac{\pi}{3}} \approx 0.68$};
        \node at (0.41,0.10) {$3$};
        \node at (0.13,0.90) {$(k_1,\delta_1) = (2.53, 0.27)$};
        \node at (0.61,-0.05) {(ii)};
        \node at (0.61,0.10) {$k_0 = 2.03$};
        \node at (0.73,0.10) {$k_1$};
        \node at (0.80,-0.05) {$|\sigma| = 0.80$};
        \node at (0.95,0.10) {$\hat{k}$};
        \node at (1.00,0.70) {$(\hat{k},\hat{\delta}) = (2.57, 0.25)$};
        \node at (0.74,0.95) {$(k_1,\delta_1) = (2.20, 0.46)$};
      \end{scope}
    \end{tikzpicture}
    \caption{\label{fig:bounds}}
  \end{center}
\end{figure}

Note that the boundary of the above $(k,\delta)$-domain consists of the following segments, namely
\begin{enumerate}[(i)]
\item \textit{The horizontal segment:} $\left\{ (k,0), k_0 \le k \le \hat{k} \right\}$ representing those $a$-isosceles $\overline{\sigma}(a, C)$ with $a = 2 \arctan \frac{1}{\sqrt{k}}$ and $C$ given by (\ref{eq:cofk}),
\item \textit{The slant segment:} $\left\{ (k,\sqrt{3} - \frac{k}{\sqrt{3}}), k_1 \le k \le \hat{k} \right\}$ representing those $\sigma$ with $\slfrac{\pi}{3}$ as their shortest side-length, the deformation along it is exactly the Lexell's deformation fixing the $\slfrac{\pi}{3}$ side (cf.~\S\ref{lexell})
\item \textit{The curved segment:} Representing those isosceles with their base angles larger than their top angles.
\item \textit{The vertical segment:} In the case of $|\sigma| > \frac{1}{2}\Box_{\frac{\pi}{3}}$ and $\hat{k} < 3$, one has an additional vertical segment: $\left\{ (\hat{k}, \delta), 0 \le \delta \le \sqrt{3} - \slfrac{\hat{k}}{\sqrt{3}} \right\}$, representing those $\sigma$ with $C = \frac{\pi}{2} + \frac{|\sigma|}{2}$, each of them is the half of a spherical rectangle with area $2|\sigma|$ and side-lengths of at least $\slfrac{\pi}{3}$, while the corner point $(\hat{k},\hat{\delta})$ is exactly the $\tilde{\sigma}_{\theta}$ with $|\tilde{\sigma}_{\theta}| = |\sigma|$. (cf. Example~\ref{eg3.2})
\end{enumerate}

\subsubsection{Basic geometric invariants of isosceles spherical triangles}
Let $\sigma_\circ$ be an $a$-isosceles spherical triangle with given area $|\sigma_\circ|$.
Set $k = \cot^2\frac{a}{2}$, $c$ to be the length of its base, and $h$ to be its height. Then
\begin{equation}
  \begin{aligned}
    & k = \cot^2\frac{a}{2}=\frac{1+\cos a}{1-\cos a}, \quad \cos a = \frac{k-1}{k+1}, \quad \cos h \cdot \cos \frac{c}{2} = \cos a \\
    & \cos \frac{|\sigma_\circ|}{2} = \frac{\cos\frac{c}{2}+\cos h}{1+\cos a} = \frac{(k+1)\cos^2\frac{c}{2} + (k-1)}{2k \cos \frac{c}{2}} \\
    & (k+1)\cos^2\frac{c}{2} - 2k\cos\frac{|\sigma_\circ|}{2}\cos\frac{c}{2} + (k-1) = 0
  \end{aligned}
\end{equation}
thus enabling us to solve $\cos\frac{c}{2}$ as a root of the above quadratic equation, a simple function of $\cos \frac{|\sigma_\circ|}{2}$ and $k$.

Next let us compute those wholesome basic geometric invariants of such a $\sigma_\circ$, namely, $\{u_\circ, D_\circ, \tan R_\circ, \nu(\sigma_\circ)\ \mathrm{and}\ w(\sigma_\circ)\}$ as follows:

\begin{equation}
  \begin{aligned}
    u_\circ &= 1 + 2 \cos a + \cos c = 2 \cos a + 2 \cos^2 \frac{c}{2} \\
    &= 2 \left\{ \frac{k-1}{k+1} + \cos^2\frac{c}{2} \right\} = \frac{2}{k+1}\left\{ (k-1) + (k+1)\cos^2\frac{c}{2} \right\} \\
    &= \frac{4k}{k+1} \cos\frac{|\sigma_\circ|}{2} \cos\frac{c}{2} \\
    D_\circ &= \tan\frac{|\sigma_\circ|}{2} u_\circ = \frac{4k}{k+1} \sin \frac{|\sigma_\circ|}{2} \cos\frac{c}{2} \\
    \tan R_\circ &= \frac{4}{D_\circ}\sin^2\frac{a}{2}\sin\frac{c}{2} = \frac{4}{D_\circ}\frac{1}{k+1}\sin\frac{c}{2} \\
    &= \frac{1}{k}\csc\frac{|\sigma_\circ|}{2}\tan\frac{c}{2} \\
    w(\sigma_\circ) &= \frac{1}{6}\sin|\sigma_\circ|\left\{\frac{8}{u_\circ}-\tan^2R_\circ\right\} \\
    &= \frac{1}{6}\sin|\sigma_\circ|\left\{ \frac{2(k+1)}{k}\sec\frac{|\sigma_\circ|}{2}\sec\frac{c}{2} - \frac{1}{k^2} \csc^2 \frac{|\sigma_\circ|}{2} \tan^2\frac{c}{2} \right\} \\
    &= \frac{\sec\frac{c}{2}}{3k^2}\left\{ 2k(k+1)\sin\frac{|\sigma_\circ|}{2} - \cot\frac{|\sigma_\circ|}{2} \tan\frac{c}{2}\sin\frac{c}{2} \right\} \\
    \tan \left( \sum d_i \right) &= \frac{\tan R_\circ \displaystyle\sum \cos \theta_i - \tan^3 R_\circ \displaystyle\prod \cos \theta_i}{1 - \tan^2 R_\circ \displaystyle\sum_{i<j} \cos \theta_i \theta_j} \\
    &= \frac{1}{D_\circ} \left\{ \frac{2D_\circ^2(2\overline{a}^2\overline{c} + 2\overline{a}\overline{c}^2 - \overline{c}^3) - 8\overline{a}^2\overline{c}(2\overline{c}^4\overline{a}^2-\overline{c}^6)}{D_\circ^2-4(4\overline{a}^3\overline{c}^3+\overline{a}^2\overline{c}^4-2\overline{a}\overline{c}^5)} \right\}
  \end{aligned} 
\end{equation}
where $D_\circ = \frac{4k}{k+1}\sin\frac{|\sigma_\circ|}{2}\cos\frac{c}{2}$, $\overline{a} \equiv \sin\frac{a}{2} = \frac{1}{\sqrt{k+1}}$ and $\overline{c} = \sin\frac{c}{2}$.

\subsubsection[Hi2]{Geometric analysis of the $\delta$-deformation}
\label{sss:deltadeform}

For a given and fixed pair of $\{C, k\}$ (resp. $(|\sigma|, k)$), set $\sigma_\circ(a, C)$ to be the unique isosceles triangle, namely
\begin{equation}
  k = \cot^2\frac{a}{2} = \frac{1+\cos a}{1-\cos a}, \quad \cos a = \frac{k-1}{k+1}
\end{equation}
Let $\sigma(a_1, a_2; C)$ be the triangle with $\{a_1, a_2; C\}$ as its S.A.S. data, $\cot\frac{a_1}{2}\cot\frac{a_2}{2} = k$ and $\delta = \cot\frac{a_1}{2} - \cot\frac{a_2}{2} > 0$. We shall call it the $\delta$-\textit{deformation} of $\sigma_\circ(a, C)$. Set
\begin{equation}
  \eta = \frac{\delta^2}{(k+1)^2} = \tan^2\left(\frac{a_2-a_1}{2}\right).
\end{equation}


We shall proceed to compute those basic geometric invariants of $\sigma(a_1, a_2; C)$ in terms of $\eta$ and that of $\sigma_\circ(a, C)$ which we derived in the previous part.
\begin{equation}
  \begin{aligned}
    \sin^2\frac{a_1}{2} \sin^2\frac{a_2}{2} &= \frac{1}{(1+\cot^2\frac{a_1}{2})(1+\cot^2\frac{a_2}{2})} = \frac{1}{(k+1)^2+\delta^2}\\
    &= \frac{1}{(k+1)^2(1+\eta)} = \sin^4\frac{a}{2}(1+\eta)^{-1} \\
    \sin\frac{a_1}{2}\sin\frac{a_2}{2} &= \sin^2\frac{a}{2}(1+\eta)^{-\slfrac{1}{2}} \\
    D &= \sin a_1 \sin a_2 \sin C = 4k \sin C \sin^2 \frac{a_1}{2} \sin^2 \frac{a_2}{2} = D_\circ(1+\eta)^{-1} \\
    u &= \frac{D}{D_\circ}u_\circ = u_\circ(1+\eta)^{-1} \\
    \sin^2\frac{a_1}{2} + \sin^2\frac{a_2}{2} &= \frac{2(k+1)+\delta^2}{(k+1)^2+\delta^2} = \frac{2}{k+1} (1+\frac{k+1}{2}\eta)(1+\eta)^{-1} \\
    &= 2\sin^2\frac{a}{2}(1+\frac{k+1}{2}\eta)(1+\eta)^{-1}\\
    \cos a_1 + \cos a_2 &= 2-2(\sin^2\frac{a_1}{2} + \sin^2\frac{a_2}{2}) = 2\frac{k-1}{k+1}(1+\eta)^{-1}\\
    &= 2\cos a (1+\eta)^{-1} \\
    1+\cos a_3 &= u-(\cos a_1 + \cos a_2) = (u_\circ-2\cos a)(1+\eta)^{-1} \\
    &= (1+\cos c)(1+\eta)^{-1}\\
    \sin^2\frac{a_3}{2} &= 1-\cos^2\frac{a_3}{2} = 1-\frac{\cos^2\frac{c}{2}}{1+\eta} = \frac{\sin^2\frac{c}{2}+\eta}{1+\eta} \\
    &= \sin^2\frac{c}{2}(1+\csc^2\frac{c}{2}\eta)(1+\eta)^{-1} \\
    \sin\frac{a_1}{2} + \sin\frac{a_2}{2} &= \sqrt{2}\sin\frac{a}{2}(1+\eta)^{-\slfrac{1}{2}} \left[(1+\eta)^{\slfrac{1}{2}}+1+\frac{k+1}{2}\eta\right]^{\slfrac{1}{2}} \\
    \sin\frac{a_3}{2} &= \sin\frac{c}{2}\left[ (1+\eta)^{-1}(1+\csc^2\frac{c}{2}\eta)\right]^{\slfrac{1}{2}}
  \end{aligned}
  \label{eq:deltasin}
\end{equation}

\begin{remark}
  In our shorthand notation $\overline a_i = \sin \frac{a_i}{2},\ i \in \{1, 2, 3\}$, these relations are summarized as:
  \begin{equation}
    \begin{aligned}
      \overline a_1^2 \overline a_2^2 &= \overline a^4 (1+\eta)^{-1} \\
      \overline a_1 \overline a_2 &= \overline a^2 (1+\eta)^{-\slfrac{1}{2}} \\
      \overline a_1^2 + \overline a_2^2 &= 2 \overline a^2 (1+\frac{k+1}{2}\eta)(1+\eta)^{-1} \\
      \overline a_3^2 &= \overline c^2(1+(\overline c)^{-2})(1+\eta)^{-1} \\
      \overline a_1 + \overline a_2 &= \sqrt{2}\overline a(1+\eta)^{-\slfrac{1}{2}} \left[(1+\eta)^{\slfrac{1}{2}}+1+\frac{k+1}{2}\eta\right]^{\slfrac{1}{2}} \\
      \overline a_3 &= \overline c\left[ (1+\eta)^{-1}(1+(\overline c)^{-2}\eta)\right]^{\slfrac{1}{2}}
    \end{aligned}
  \end{equation}
\end{remark}



\begin{equation}
  \begin{aligned}
    \tan R &= \frac{4}{D}\overline a_1 \overline a_2 \overline a_3 = \frac{4}{D_\circ} \overline a^2 \overline c^2 (1+(\overline c)^{-2}\eta) \\
    &= \tan R_\circ(1+(\overline c)^{-2}\eta)(1+\eta)^{-\slfrac{1}{2}}\\
  \end{aligned}
\end{equation}

The expression for $\rho(\sigma)$ requires expressions for $d_i$.  We first derive an expression for $\tan d_3$:

\begin{equation}
  \begin{aligned}
    \tan \lambda_3 &= \frac{1+\cos a_3 - \cos a_1 - \cos a_2}{D} = \frac{1+\cos c-2\cos a}{D_\circ} \\
    &= \tan \lambda_3^\circ \\
    \tan d_3 &= \tan \lambda_3 \sin \frac{a_3}{2} = \tan  \lambda_3^\circ \sin \frac{c}{2} \left[ \frac{1+\csc^2\frac{c}{2}\eta}{1+\eta} \right]^{\slfrac{1}{2}} \\
    &= \tan d_3^\circ \left[ \frac{1+(\overline c)^{-2}\eta}{1+\eta} \right]^{\slfrac{1}{2}}
  \end{aligned}
\end{equation}



Finally, let us proceed to compute $\tan(d_1+d_2)$ as follows:
\begin{equation}
  \begin{aligned}
    \tan(d_1+d_2) &= \frac{\tan R(\cos \theta_1 + \cos \theta_2)}{1-\tan^2R\cos\theta_1\cos\theta_2} \\
    \tan R &= \frac{4}{D}\overline a_1 \overline a_2 \overline a_3, \quad \cos \theta_1 + \cos \theta_2 = \frac{\overline a_1 + \overline a_2}{2 \overline a_1 \overline a_2 \overline a_3} \left\{ \overline a_3^2 - (\overline a_1 - \overline a_2)^2 \right\} \\
    \cos \theta_1 \cdot \cos \theta_2 &= \frac{1}{4 \overline a_3 \cdot \prod \overline a_i} \left\{ \overline a_3^4 - (\overline a_1 + \overline a_2)^2(\overline a_1 - \overline a_2)^2 \right\}
  \end{aligned}
\end{equation}
Therefore,
\begin{equation}
  \begin{aligned}
    \tan R (\cos \theta_1 + \cos \theta_2) &= \frac{2(\overline a_1 + \overline a_2)}{D}(\overline a_3^2 - (\overline a_1 - \overline a_2)^2) \\
    1-\tan^2 R \cos \theta_1 \cdot \cos \theta_2 &= 1-\frac{4 \overline a_1 \overline a_2}{D^2}(\overline a_3^4 - (\overline a_1 + \overline a_2)^2(\overline a_1 - \overline a_2)^2)
  \end{aligned}
\end{equation}
where
\begin{equation}
  \begin{aligned}
    (\overline a_1 \pm \overline a_2)^2 &= 2 \overline a^2 \cdot (1+\eta)^{-1} \cdot \left[ 1 + \frac{k+1}{2}\eta \pm (1+\eta)^{\frac{1}{2}} \right] \\
    \overline a_3^2 &= \overline c^2(1+\eta)^{-1}\cdot(1+(\overline c)^{-2}\eta) \\
    \overline a_1 \overline a_2 &= \overline a^2(1+\eta)^{-\frac{1}{2}}, \quad D = D_\circ(1+\eta)^{-1}
  \end{aligned}
\end{equation}
Thus, straightforward substitutions will show that
\begin{equation}
  \begin{aligned}
    \tan R(\cos \theta_1 + \cos \theta_2) =& \frac{4\overline a}{\sqrt{2}D}(1+\eta)^{-\frac{1}{2}}\left[ 1+\frac{k+1}{2}\eta+(1+\eta)^{\frac{1}{2}} \right]^{\frac{1}{2}} \\
    & \cdot \left\{ \overline c^2(1+\eta(\overline c)^{-2})-2\overline a^2(1+\frac{k+1}{2}\eta - (1+\eta)^{\frac{1}{2}}) \right\} \\
    1-\tan^2R\cos \theta_1 \cos \theta_2 =& 1-\frac{4}{D_\circ^2} \overline a^2(1+\eta)^{-\frac{1}{2}}\\
    & \cdot \left\{ \overline c^4 (1+(\overline c)^{-2}) - 4 \overline a^4(k\eta+\frac{(k+1)^2}{4}\eta^2) \right\} \\
    =& 1 - \frac{4\overline a^2}{D_\circ^2} \cdot \overline c^4(1+\eta)^{-\frac{1}{2}}\\
    & \cdot \left\{ 1+(2(\overline c)^{-2} - 4k\overline a^4(\overline c)^{-4})\eta - (\overline c)^{-4}((k+1)^2\overline a^4-1)\eta^2 \right\} \\
  \end{aligned}
\end{equation}
Taking a power series expansion in $\eta$, one obtains to first order:
\begin{equation}
  \tan(d_1+d_2) = \tan(2d_\circ)\left\{1 + \left(
  \begin{gathered}
    \frac{D^2(8\overline a^3+8\overline a^2-8\overline a+\overline c^2-3\overline a^2\overline c^2)}{8\overline a^2\overline c^2(4\overline a^2\overline c^4-D^2)} \\
    - \frac{32\overline a^6 - 32\overline a^4 - 8\overline a^3\overline c^2 + 8 \overline a^2\overline c^2 + 8\overline a\overline c^2-\overline a^2\overline c^4-\overline c^4}{2(4\overline a^2\overline c^4-D^2)}
  \end{gathered}
  \right)
  \eta + O(\eta^2)\right\}
\end{equation}

\subsubsection[Hi4]{Lexell's deformation and some specific kinds of area-preserving deformations}
\label{lexell}

\begin{figure}
  \begin{center}
    \begin{tikzpicture}
      \node[anchor=south west,inner sep=0] (image) at (0,0) {\includegraphics[width=3.5in]{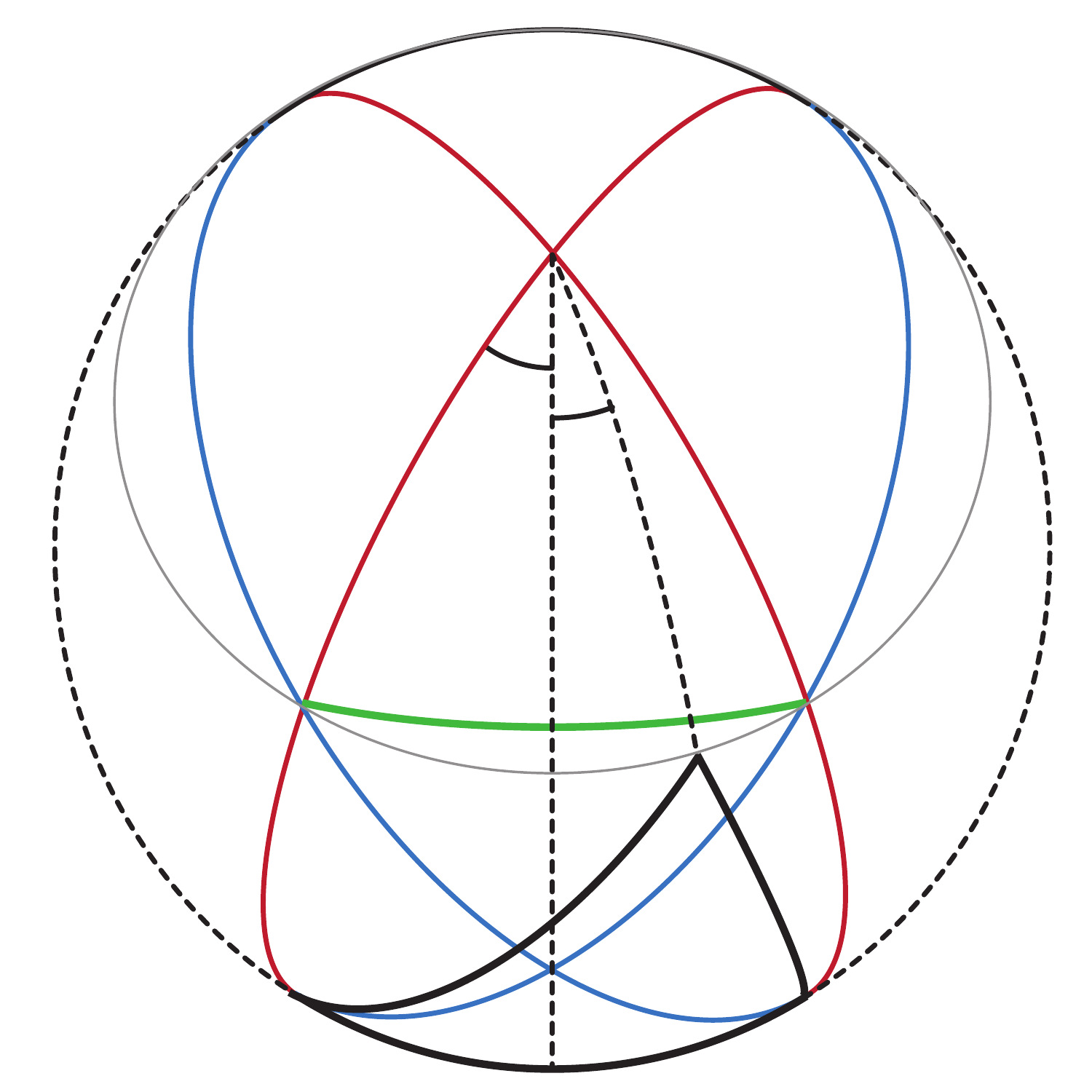}};
      \begin{scope}[x={(image.south east)},y={(image.north west)}]
        \node at (0.25,0.05) {$B$};
        \node at (0.75,0.05) {$C$};
        \node at (0.67,0.32) {$A$};
        \node at (0.23,0.36) {$A_2$};
        \node at (0.77,0.36) {$A_1$};
        \node at (0.55,0.77) {$O$};
        \node at (0.485,0.70) {$\theta$};
        \node at (0.525,0.65) {$\varphi$};
        \node at (0.22,0.21) {$b$};
        \node at (0.80,0.21) {$b$};
        \node at (0.47,0.00) {$a$};
        \node at (0.47,0.36) {$a$};
        \node at (0.42,0.21) {$d$};
        \node at (0.60,0.16) {$d$};
        \node at (0.25,0.62) {$\pi - d$};
        \node at (0.75,0.62) {$\pi - d$};
        \node at (0.50,0.90) {$\frac{1}{2}(\pi - b)$};
        \node at (0.50,0.50) [fill=white, fill opacity=0.7] {$\frac{1}{2}(\pi - b)$};
        \node at (0.50,0.50) {$\frac{1}{2}(\pi - b)$};
        \node at (0.77,0.93) {$B^\prime$};
        \node at (0.25,0.92) {$C^\prime$};
        \draw[arrows=->](0.59,0.90)--(0.63,0.90);
        \draw[arrows=->](0.41,0.90)--(0.37,0.90);
        \draw[arrows=->](0.59,0.50)--(0.66,0.50);
        \draw[arrows=->](0.41,0.50)--(0.35,0.50);
      \end{scope}
    \end{tikzpicture}
    \caption{Lexell's deformation of the spherical triangle $\sigma(ABC)$. The Lexell's circle is shown in dark gray. The top edge of the spherical rectangle $\Box A_1A_2BC$ is shown in green. The great circles passing through $BA_2OB^\prime$ and $CA_1OC^\prime$ form the sides of $\Box A_1A_2BC$ and the radii of the Lexell's circle, and are shown in red. The great circles passing through $BA_1B^\prime$ and $CA_2C^\prime$ form the diagonals of $\Box A_1A_2BC$ and are shown in blue.  $\varphi$ parameterizes the deformation starting from the isosceles case.}
    \label{fig:lexell}
  \end{center}
\end{figure}
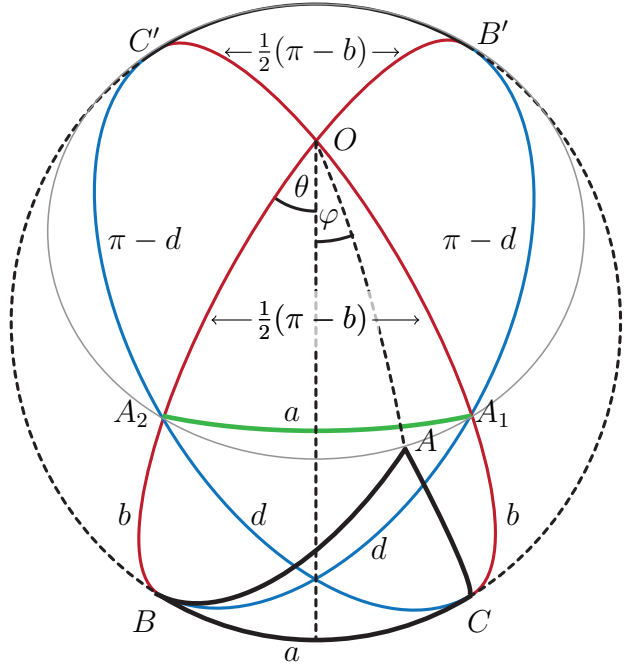

\begin{example}{(Lexell's deformations (cf. Example 2.13, p.~59 of \cite{hsiang})).}
  As indicated in Figure~\ref{fig:lexell}, $\Box A_1A_2BC$ is an $(a,b)$-``rectangle'' with 2$|\sigma|$ as its area, $O$ is the center of the Lexell's circle, i.e. passing $A_1$, $A_2$ and $B^{\prime}$, $C^{\prime}$ (the antipode of $B$, $C$). Set its radius to be $r(a, |\sigma|)$ and
  \begin{equation}
    \theta = \frac{1}{2}\angle A_2OA_1,\ \angle BOA=\theta+\varphi,\ \angle AOC=\theta-\varphi
  \end{equation}
  Then
  \begin{equation}
    \begin{aligned}
      \tan^2\frac{|\sigma|}{2} &= \frac{(1-\cos a)(1-\cos b)}{2(\cos a + \cos b)},\ \cos b = \frac{1-\cos a(2 \tan^2\frac{|\sigma|}{2} + 1)}{2\tan^2\frac{|\sigma|}{2} + 1 - \cos a} \\
      \sin r(a, |\sigma|) &= \cos \frac{b}{2},\ \sin \theta = \sin \frac{a}{2} \sec \frac{b}{2} \\
      \cos \overline{AB} &= \cos^2 \frac{b}{2} \cos(\theta + \varphi) - \sin^2\frac{b}{2} = 2 \cos^2 \frac{b}{2}\cos^2\frac{\theta+\varphi}{2}-1 \\
      \cos \overline{AC} &= \cos^2 \frac{b}{2} \cos(\theta - \varphi) - \sin^2\frac{b}{2} = 2 \cos^2 \frac{b}{2}\cos^2\frac{\theta-\varphi}{2}-1 \\
      \cos \frac{1}{2}\overline{AB} &= \cos\frac{b}{2}\cos\frac{\theta+\varphi}{2},\ \cos\frac{1}{2}\overline{AC} = \cos\frac{b}{2}\cos\frac{\theta-\varphi}{2}
    \end{aligned}
  \end{equation}
\end{example}

\begin{example}{(The computation of $\rho(|\sigma|, k, \sqrt{3}-\frac{k}{\sqrt{3}})$).}
  Note that $\rho(|\sigma|, k, \sqrt{3}-\frac{k}{\sqrt{3}}) = \rho(\tau(\sigma,2))$ where the S.A.S. data of such a $\sigma$ is given by $a_1 = \slfrac{\pi}{3}$, $a_2 = 2 \arctan \frac{\sqrt{3}}{k}$ and $C$ given by~(\ref{eq:cofk}).

  Therefore,
  \begin{equation}
    \begin{aligned}
      D &= \frac{\sqrt{3}}{2}\sin a_2 \sin C = \frac{3k}{k^2+3} \sin C \\
      \cos a_3 &= \frac{3k}{k^2+3} \cos C + \frac{1}{2}\frac{k^2-3}{k^2+3} \\
      \tan^2 R &= \frac{1}{D^2}(1-\cos a_2)(1-\cos a_3) = \frac{k^2 + 9 - 6k \cos C}{3k^2\sin^2C} \\
      d_1 &= \arccos{(\frac{2}{\sqrt{3}}\frac{1}{\sqrt{1+\tan^2R}})} \\
      d_2 &= \arccos{(\frac{1}{k}\sqrt{\frac{k^2+3}{1+\tan^2R}})} \\
      d_3 &= \arccos{\left[\frac{2}{(1+\cos a_3)(1+\tan^2 R)}\right]}^{\slfrac{1}{2}} \\
      \rho(|\sigma|, k, \sqrt{3}-\frac{k}{\sqrt{3}}) &= \frac{1}{4D}(\pi + 2 |\sigma| - 2(d_1 + d_2 + d_3))
    \end{aligned}
  \end{equation}
\end{example}

\begin{example}{(The computation of $\rho(|\sigma|, k, \delta^{*})$.}
  
  Note that $\rho(|\sigma|, k, \delta^*) = \rho(\tau(\sigma,2))$ where $\sigma$ are those isosceles triangles with given $|\sigma|$ and the $C$ given by~(\ref{eq:cofk}) as their base angles. Set $c$ (resp. $b,\ \alpha$) to be the equal side-length (resp. base-length and top-angle) of such a $\sigma = \sigma_\circ(c, \alpha)$. Then

  \begin{equation}
    \tan\frac{b}{2} = \tan c \cos C, \quad k = \cot \frac{c}{2} \cot \frac{b}{2} = \frac{\cos c}{1-\cos c}\sec C
  \end{equation}

  Therefore,

  \begin{equation}
    \cos c = \frac{k \cos C}{1 + k \cos C}, \quad \tan \frac{\alpha}{2} = \sec c \cot C = \frac{1+k \cos C}{k \sin C}
  \end{equation}

  and hence
  \begin{equation}
    \begin{aligned}
      D &= \sin^2c\sin\alpha = (1-\cos^2c)\frac{2\tan\frac{\alpha}{2}}{1+\tan^2\frac{\alpha}{2}} \\
      &= \frac{2k\sin C(1+2k\cos C)}{(1+k\cos C)(1+k^2+2k\cos C)} \\
      \tan d_1 &= \tan \frac{\alpha}{2}\sin\frac{c}{2}=\frac{\sqrt{1+k\cos C}}{\sqrt{2}k\sin C} \\
      \tan d &= \sin \frac{b}{2} \tan (C-\frac{\alpha}{2}) = \sin c \sin \frac{\alpha}{2} \frac{\tan C - \tan \frac{\alpha}{2}}{1+\tan C   \tan \frac{\alpha}{2}} 
    \end{aligned}
  \end{equation}

  Thus, for $k_0 \le k \le k_1$

  \begin{equation}
    \rho(|\sigma|, k, \delta^*) = \frac{1}{4D}\left\{ \pi + 2|\sigma| - 4 d_1 - 2 d \right\}
  \end{equation}

  where
  \begin{equation}
    \delta^* = \cot \frac{b}{2} - \cot \frac{c}{2} = \frac{k}{\sqrt{1+2k\cos C}} - \sqrt{1+2k\cos C}
  \end{equation}

\end{example}

\begin{figure}
  \begin{center}
    \includegraphics[width=3.5in]{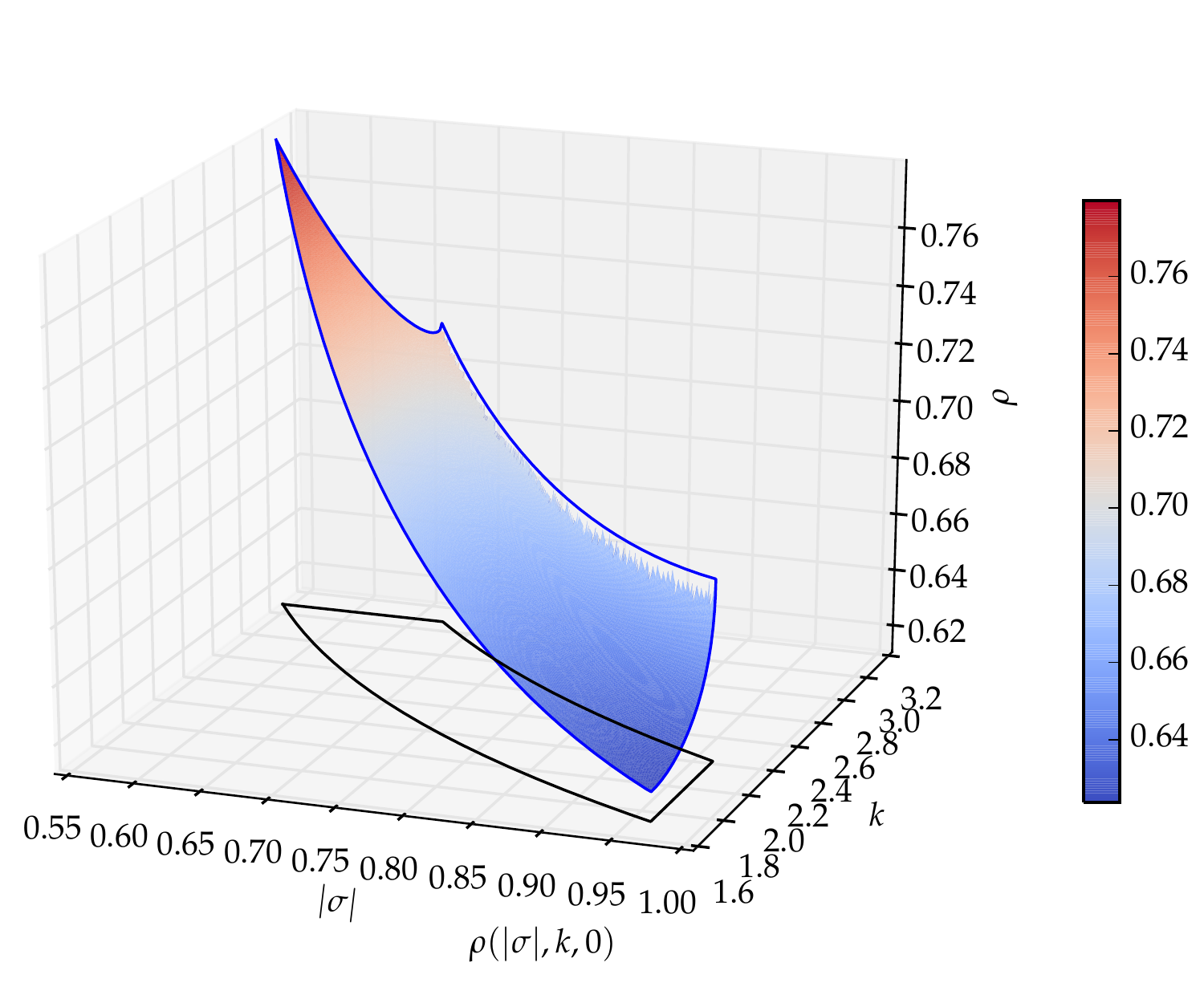}
    \caption{}
    \label{fig:sigmak}
  \end{center}
\end{figure}

\subsubsection[Hi5]{Remarks on the behavior of $\rho(|\sigma|, k, \delta)$}
\begin{enumerate}[(i)]
\item For each given pair of $(|\sigma|, k)$, it follows from the analysis of \S\ref{sss:deltadeform} that $\rho(|\sigma|, k, \delta)$ is an increasing function of $\delta$ with $\rho(|\sigma|, k, 0)$ as its minimum, while the above computation provides an estimate of its maximum (i.e. $\rho(|\sigma|, k, \delta^*)$ with only a rather small increment of $\eta^*$-order. 
\item For each given area $|\sigma|$, $\rho(|\sigma|, k_0, 0)$ is the unique minimum of $\rho(|\sigma|, k, \delta)$. Thus the values of $\rho(|\sigma|, k, \delta)$ with small $(k-k_0)$ only have very small second order increments above the minimum value of $\rho(|\sigma|, k_0, 0)$.
\item See Figure~\ref{fig:kdelta} for the graphs of $\rho(|\sigma|, k, \delta)$ as functions of $(k, \delta)$ with given values of $|\sigma|$.

\begin{figure}
  \begin{center}
    \includegraphics[width=6.5in]{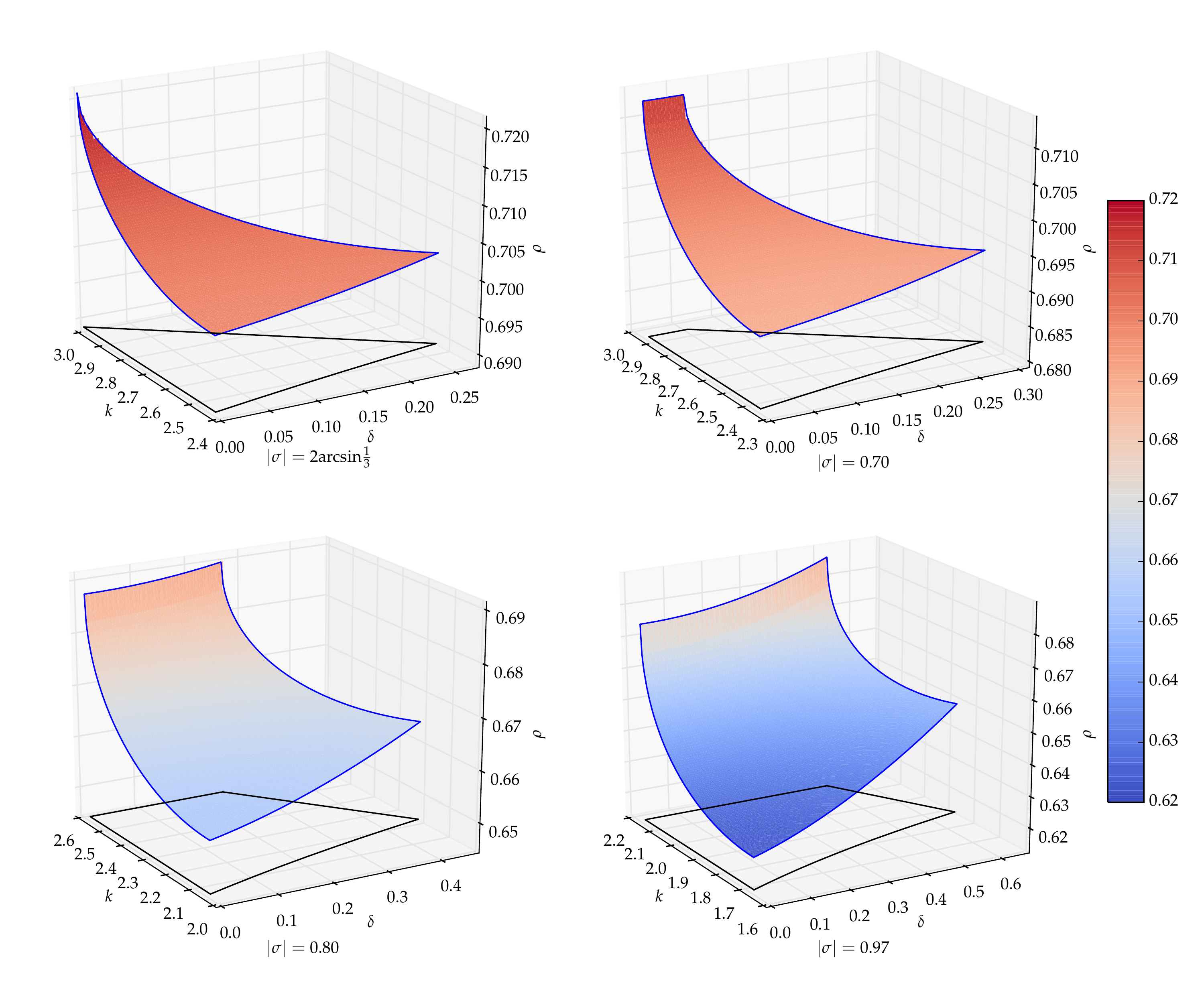}
    \caption{}
    \label{fig:kdelta}
  \end{center}
\end{figure}

\begin{figure}
  \begin{center}
    \includegraphics[width=3.5in]{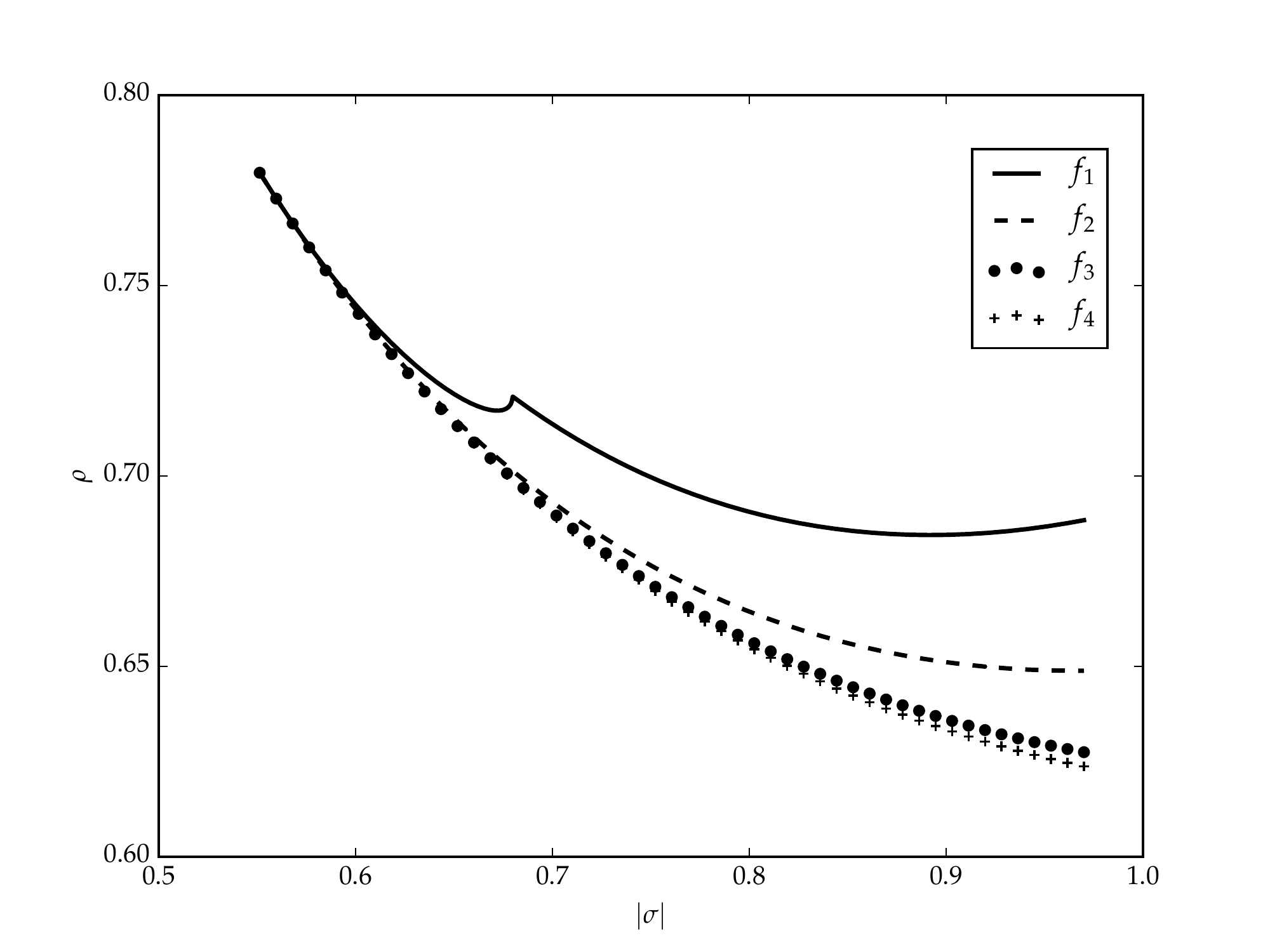}
    \caption{}
    \label{fig:fourcurves}
  \end{center}
\end{figure}

\item In order to present a simplified, over-all picture of the result of this section on the geometric analysis of $\rho(|\sigma|, k, \delta)$, we plot the graph of the following four critical functions of $|\sigma|$ in Figure~\ref{fig:fourcurves}, namely
  \begin{equation}
    \begin{aligned}
      f_1(|\sigma|) \colonequals \rho(|\sigma|, \hat k, \hat \delta),\ & f_2(|\sigma|) \colonequals \rho(|\sigma|, k_1, \delta_1) \\
      f_3(|\sigma|) \colonequals \rho(|\sigma|, k_1, 0),\ & f_4(|\sigma|) \colonequals \rho(|\sigma|, k_0, 0)
    \end{aligned}
  \end{equation}

\end{enumerate}

%% file: 4-type-I-configurations/4-type-I-configurations.tex
\section{A concise review of some highlights on the geometry of Type-I spherical configurations}
\label{sec:review}


A spherical configuration with twelve vertices and edge lengths of at least $\slfrac{\pi}{3}$ will be, henceforth, referred to as a Type-I (spherical) configuration.
The moduli space of congruence classes of Type-I configurations constitute a real semi-algebraic set of twenty-one dimension which has quite a few interesting properties such as those theorems of \S 7.1 in \cite{hsiang}.
In this section, we shall review some highlights of such special results on the geometry of Type-I configurations which will play a useful role in the proof of Theorem I.

Geometrically, to each given Type-I configuration $\sconf$, there exists a unique Type-I local packing $\mathcal{L}(\mathcal{S}_0)$ containing twelve touching neighbors with $\Sigma$ as the touching points together with the tightest extension of additional neighbors, if any, thus achieving the highest locally averaged density which shall be defined to be the associated locally averaged density of such a Type-I configuration.
Anyhow, one has a function of locally averaged density $\overline{\rho}(\cdot)$ defined on the moduli space of Type-I configurations, namely
\begin{equation}
\overline{\rho} \quad : \quad \mathcal{M}_I \rightarrow \mathbb{R}_+
\end{equation}
while the proof of Theorem I for the the major and most critical case of Type-I local packings amounts to prove that the above function has the f.c.c. and the h.c.p. configurations as the unique two maximal points with $\overline{\rho}(\cdot) = \spet$.
Thus, it is quite natural that our review should begin with the following geometric characterization of these two outstanding Type-I configurations, namely, the f.c.c. and the h.c.p.

\subsection{On some special features and the geometric characterizations of the f.c.c. and the h.c.p.}

Let $\sconf$ be a given Type-I configuration. Then $\Sigma$ is already $\slfrac{\pi}{3}$-{\it saturated}, or equivalently, the circumradii of those faces of $\sconf$ are all less than $\slfrac{\pi}{3}$.
Therefore the areas of a triangular (resp. quadrilateral) faces of a Type-I configuration are at least equal to
\begin{equation}
  \triangle_{\slfrac{\pi}{3}} = 3 \cos^{-1} \frac{1}{3} - \pi \quad \textrm{(resp. } \square_{\slfrac{\pi}{3}} = 4 \cos^{-1}(-\frac{1}{3})-2\pi\textrm{)}
\end{equation}
while
\begin{equation}
  8 \triangle_{\slfrac{\pi}{3}} + 6 \square_{\slfrac{\pi}{3}} = 4\pi
\end{equation}

Hence, a type-I configuration can have at most six quadrilaterals and the f.c.c. and the h.c.p as indicated in Figure 1 are the \textit{only} two such configurations with six quadrilaterals, say \textit{Type-I configurations of} 6$\square$-\textit{type}.

\subsubsection{}
Note that the local star configurations of the f.c.c. are all of the type of $\left\{ \triangle, \square, \triangle, \square \right\}$ at every point,
while that of the h.c.p. has six local stars of the $\left\{ \triangle, \square, \triangle, \square \right\}$-type and another sextuple of the $\left\{ \triangle, \triangle, \square, \square \right\}$ type.
It is a remarkable fact that the mere occurrence of a local star of $\left\{ \triangle, \square, \triangle, \square \right\}$-type in a Type-I configuration, in fact, already characterizes these two outstanding 6$\square$-type ones, namely

\begin{lemma}
  Suppose that $\sconf$ is a Type-I configuration with a star of $\left\{ \triangle, \square, \triangle, \square \right\}$-type. Then $\sconf$ is either the f.c.c. or the h.c.p.
  \label{lem:sixstar1}
\end{lemma}

\begin{proof}
  Let $N$ be such a point of $\Sigma$ with the above kind of star structure.  Set $\left\{ E, A, B, W, C, D \right\}$ to be its boundary vertices in cyclic order, while $\left\{ E, W \right\}$ are antipodal.
  Then, the complementary region of $St(N)$, namely
  \begin{equation}
    \Omega(St(N)) \colonequals S^2(1) \setminus \cup \left\{ D^\circ(\cdot, \slfrac{\pi}{3})\right\}
  \end{equation}
  can be represented via the stereographic projection with $N$ as the north pole, as indicated in Figure~\ref{fig:complement},

\begin{figure}[H]
  \begin{center}
    \includegraphics[width=3.5in]{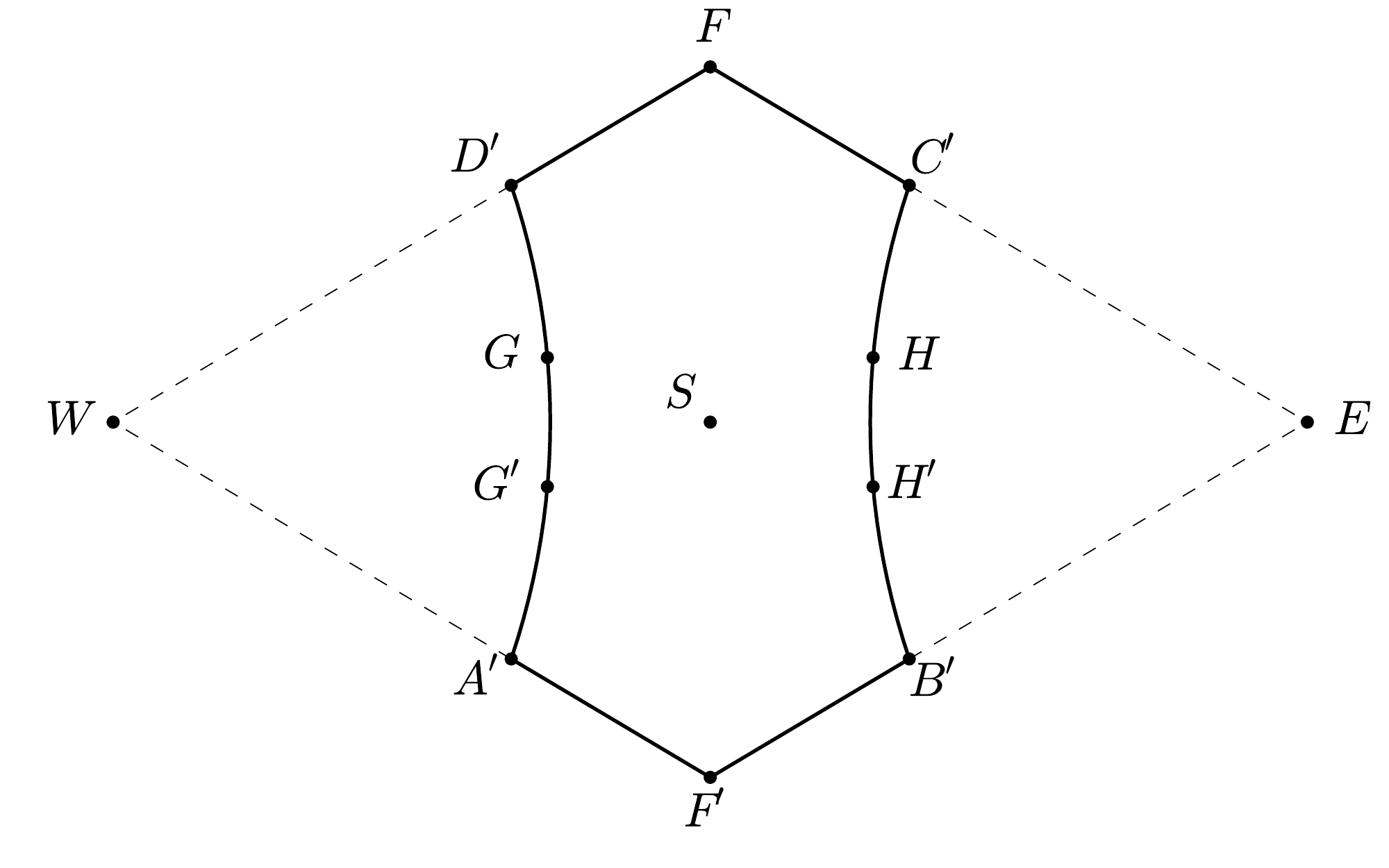}
    \caption{\label{fig:complement}}
  \end{center}
\end{figure}

\noindent
where the symmetric center of $\Omega$ is exactly the south pole and $\left\{ A^{\prime}, B^{\prime}, C^{\prime}, D^{\prime} \right\}$ are respectively the antipodal points of $\left\{ A, B, C, D \right\}$.

We shall use spherical trigonometry to analyze the possibilities of placing a quintuple of $\slfrac{\pi}{3}$-separated points inside of $\Omega$.
First of all, it is easy to check that
\begin{enumerate}[(i)]
\item $\left\{ S, A^{\prime}, B^{\prime}, C^{\prime}, D^{\prime} \right\}$,
\item $\left\{ F, G, H, A^{\prime}, B^{\prime} \right\}$ or $\left\{ F^{\prime}, G^{\prime}, H^{\prime}, D^{\prime}, C^{\prime} \right\}$
\end{enumerate}
are such possibilities, while the addition of (i) (resp. (ii)) extends the star configuration to the f.c.c. (resp. h.c.p.) configuration.
Thus, the proof of Lemma~\ref{lem:sixstar1} amounts to showing that they are the only possibilities.

Let $R$ be a point on the $\slfrac{\pi}{3}$-circular arc centered at $E$ between $H$ and $H^{\prime}$. As indicated in Figure~\ref{fig:crab}, $\left\{ P, Q; P^{\prime}, Q^{\prime} \right\}$ are a quadruple of points along $\partial\Omega$ such that
\begin{equation*}
  \overline{RP}, \overline{PQ} \textrm{ and } \overline{RP^{\prime}}, \overline{P^{\prime} Q^{\prime}}
\end{equation*}
are equal to $\slfrac{\pi}{3}$, namely, the spherical triangles of
\begin{equation*}
  \left\{ E, P, R \right\}, \left\{ P, W, Q \right\} ;\ \left\{ E, P^{\prime}, R \right\}, \left\{ P^{\prime}, W, Q^{\prime} \right\}
\end{equation*}
are all $\slfrac{\pi}{3}$-isosceles.
Set $\left\{ \theta_1, \theta_2;\ \theta_1^{\prime}, \theta_2^{\prime} \right\}$ to be their base angles and 
$\left\{ 2b_1, 2b_2;\ 2b_1^{\prime}, 2b_2^{\prime} \right\}$ to be their base lengths, respectively.
\begin{figure}[H]
  \begin{center}
    \includegraphics[width=3.5in]{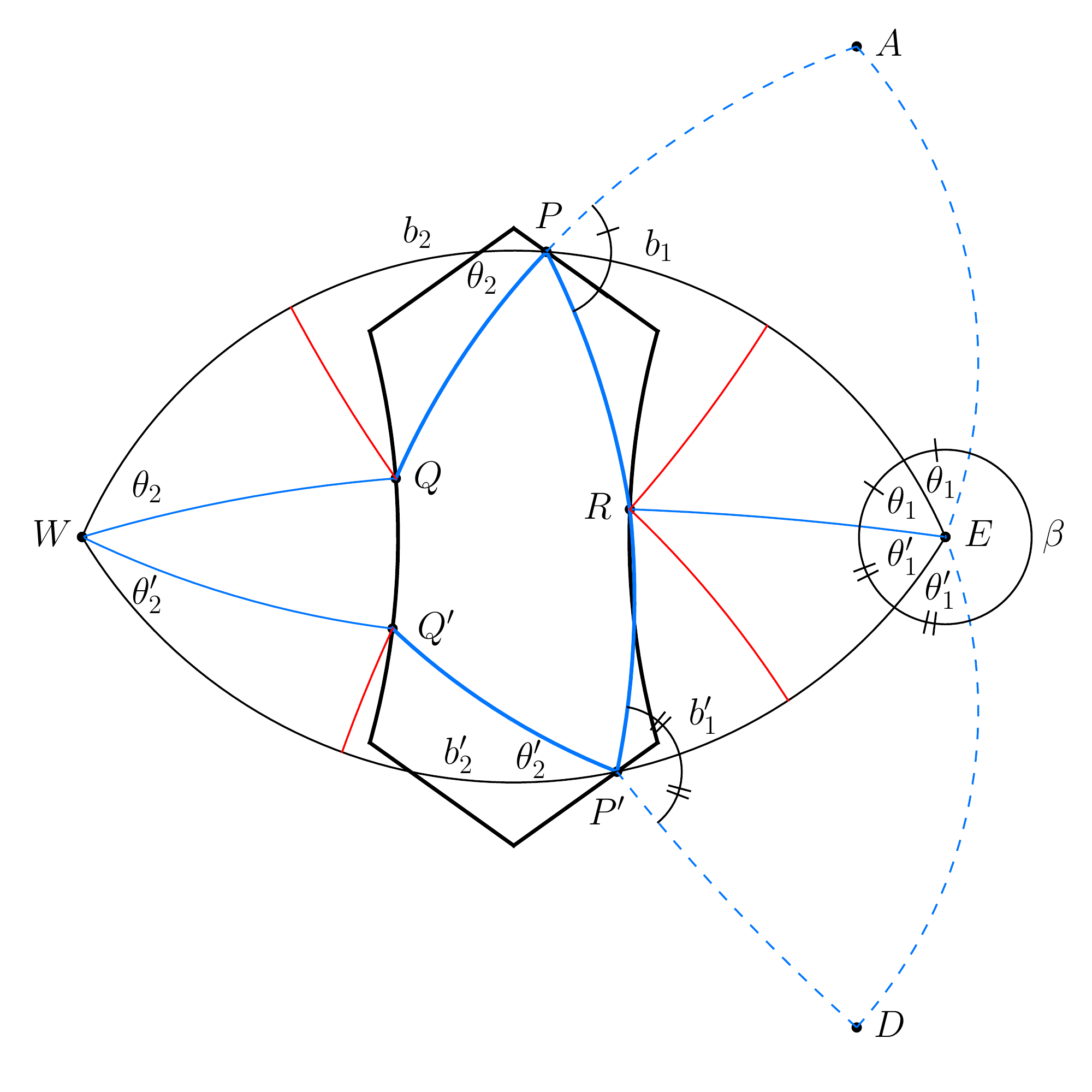}
    \caption{\label{fig:crab}}
  \end{center}
\end{figure}
Then, one has
\begin{equation}
  \begin{split}
    b_1 + b_2 &= b_1^{\prime} + b_2^{\prime} = \slfrac{\pi}{2} \\
    \tan b_1 &= \sqrt{3} \cos \theta_1,\; \tan b_2 = \sqrt{3} \cos \theta_2; \quad \cos \theta_1 \cdot \cos \theta_2 = \frac{1}{3} \\
    \tan b_1^{\prime} &= \sqrt{3} \cos \theta_1^{\prime},\; \tan b_2^{\prime} = \sqrt{3} \cos \theta_2^{\prime}; \quad\cos \theta_1^{\prime} \cdot \cos \theta_2^{\prime} = \frac{1}{3}
  \end{split}
\end{equation}
and moreover,
\begin{equation}
  \sigma(EPA) \cong \sigma(EPR), \quad \sigma(EP^{\prime}D) \cong \sigma(EP^{\prime}R)
\end{equation}
and hence
\begin{equation}
  2 \theta_1 + 2 \theta_1^{\prime} + \cos^{-1}(-\frac{1}{3}) = 2\pi, \quad \theta_1 + \theta_1^{\prime} = \cos^{-1}(-\frac{1}{\sqrt{3}})
\end{equation}

Set $\lambda$ to be the top angle of the isosceles $\sigma(QQ^{\prime}W)$, namely
\begin{equation}
  \lambda = \theta_1 + \theta_1^{\prime} - \theta_2 - \theta_2{^\prime}
\end{equation}
It is quite straightforward to use the above set of equations to compute $\lambda$ as a function of $\theta_1$ and to show that $\lambda = \cos^{-1}(\frac{1}{3})$ when and only when $R = H \textrm{ or } H^{\prime}$, while it is strictly smaller otherwise.
With the above critical analysis at hand, it is quite simple to prove that there are no other possibilities of placing a quintuple of $\slfrac{\pi}{3}$-separated points inside of $\Omega$.
\end{proof}

\begin{corollary}
  Suppose that a Type-I configuration $\sconf$ contains a 6$\triangle$-star whose pair of small triangles are separated by a pair of buckled quadrilaterals.
  Then $\sconf$ must be just such a triangulation of the f.c.c. or the h.c.p.
\end{corollary}

\begin{proof}
  If the quadruple radial edges of the pair of small triangles are all equal to the minimal length of $\slfrac{\pi}{3}$, then such a 6$\triangle$-star must be a triangulation of $\left\{ \triangle, \square, \triangle, \square \right\}$ of Lemma~\ref{lem:sixstar1} and $\sconf$ must be either the f.c.c. or the h.c.p.
  On the other hand, any small amounts of edge-excesses in the four radial edges will make the complementary region of such a 6$\triangle$-star carve away critical strips from that of Lemma~\ref{lem:sixstar1};
  even just a single tiny such edge-excess makes its complementary region incapable of accommodating a quintuple of $\slfrac{\pi}{3}$-separated points, contradicting the assumption of being contained in the Type-I configuration.
\end{proof}

\subsubsection[Hi]{On the deformations of 6$\square$-Type-I configurations (i.e. 6$\boxslash$-Type-I configurations)}
Let $\sconf$ be a small deformation of either the f.c.c. or the h.c.p. within $\mathcal{M}_I$, and $\left\{ \boxslash_i, 1 \le i \le 6 \right\}$ be the small deformation of the sextuple
$\slfrac{\pi}{3}$-squares of the f.c.c. (or h.c.p.).
Set $\epsilon_i$ to be the difference between $\arccos (-\frac{1}{3})$ and the smallest angle of $\boxslash_i$ and $\overline{\epsilon} = \frac{1}{6}\sum_i\epsilon_i$, which will be regarded as a kind of measurement of the size of such a deformation.
Note that the area of $\boxslash_i$ is at least equal to
\begin{equation}
  \square_{\slfrac{\pi}{3}} - 0.7046\epsilon_i^2, \quad (\hbox{mod } \epsilon_i^4)
\end{equation}
Then, it follows from the area estimate that the edge-lengths of its octuple small triangles must be all equal to $\slfrac{\pi}{3}$ modulo $\overline{\epsilon}^2$.
Therefore, up to congruence, the deformation of each small triangle is just a small rotation modulo $\overline{\epsilon}^2$, and moreover, it is easy to see that the sizes of those rotations of the octuple small triangles are, in fact, equal to each other modulo $\overline{\epsilon}^2$.
Note that such a deformation would be impossible for the h.c.p. because it has adjacent pairs of small triangles.
On the other hand, all small deformations of the f.c.c. are of the combinatorial type of icosahedra (cf. Lemma 7.1 of \cite{hsiang} for more precise results on such a $\sconf$).

\subsection[Hi1]{Geometry of non-icosahedron Type-I configurations and $5\boxslash$-Type-I configurations}\label{subsec:geometry}

Let $S^{\prime}(\Sigma)$ be a triangulation of a Type-I configuration $S(\Sigma)$, namely, by adding one of the pair of diagonals of its quadrilateral faces, if any.
Then, by the Euler's formula, $S^{\prime}(\Sigma)$ always has twenty triangles, thirty edges and twelve vertices of degrees 4, 5, or 6.
We shall, henceforth, call those $S^{\prime}(\Sigma)$ with uniform degree 5 at each vertex \textit{Type-I icosahedra}, while the other will be referred to as \textit{non-icosahedral} Type-I configurations.
For example, the f.c.c., the h.c.p. and the 5$\Box$-Type-I configurations of Example~\ref{ex:5sq} all have \textit{both} icosahedral and non-icosahedral triangulations, depending on the way of choosing diagonal cuttings of the $\Box$-faces.
In fact, if $S(\Sigma)$ has at least one quadrilateral face, then one (or both) of its diagonal cuttings will cause $S^{\prime}(\Sigma)$ to become non-icosahedral.
Anyhow, we would like to study the geometric structure of those genuine non-icosahedral Type-I configurations.

Suppose that $S^{\prime}(\Sigma)$ is a given non-icosahedral Type-I configuration which is \textit{not} such a triangulation of the f.c.c. or the h.c.p. Then it must contain a 6$\triangle$-star not of the kind of Corollary of Lemma 1, say of the other kind, namely with the pair of small triangles adjacent to each other.

\subsubsection[Hi2]{On the geometry of complementary regions of those 6$\triangle$-stars of the other kind that are capable of accommodating a quintuple of $\slfrac{\pi}{3}$-separated points}

Let us first take a look at some examples of such 6$\triangle$-stars and their complementary regions:

\begin{example}{6$\triangle$-star of $\{\triangle,\triangle,\Box,\Box\}$:}
  \label{ex:6tri}

  First of all, it is the only 6$\triangle$-star of such kind with a quadruple of $\slfrac{\pi}{3}$-edges (or with a pair of $\slfrac{\pi}{3}$-equilateral triangles).
  Secondly, it occurs as the local stars at a sextuple of points in the h.c.p. configuration, as well as that of some vertices of those 5$\Box$-Type-I configurations with at least a pair of adjacent $\{\alpha_0,\alpha_0\}$ in their angular distributions at the poles, such as $(\alpha_0, \alpha_0, \alpha_0, \alpha_0, 2\pi-4\alpha_0)$.
  Again, set the central point of such a $\{\triangle,\triangle,\Box,\Box\}$-star to be the north pole $N$ and using the stereographic projection to represent the complementary region of $St(N)$, as indicated in Figure~\ref{fig:6star1}.
\end{example}

\begin{figure}[H]
  \begin{center}
    \includegraphics[width=2.5in]{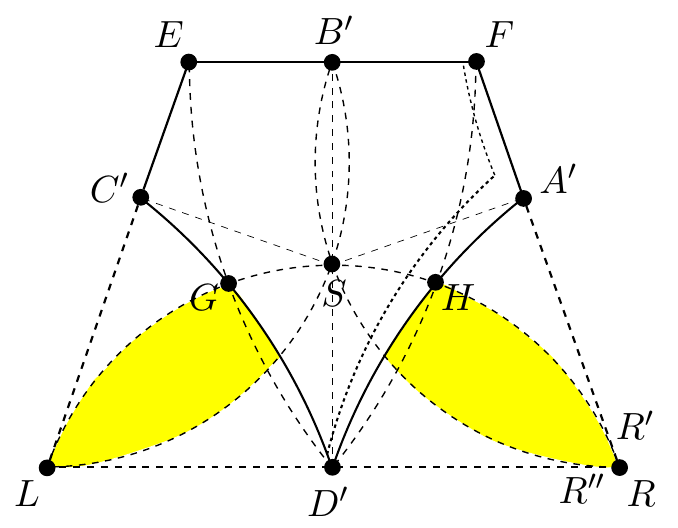}
    \caption{}
\label{fig:6star1}
\end{center}
\end{figure}

\begin{example} \label{ex:5sq}
  Let $S(\Sigma)$ be a 5$\Box$-Type-I with $(\theta_1, \theta_2, ..., \theta_5)$ as the angular distribution at the poles which is the assemblage of quintuple lune clusters $\{L_{\theta_i}\}$ as indicated in Figure~\ref{fig:lune}.
  Then the star at $A_2$ is a 6$\triangle$-star as indicated in Figure~\ref{fig:6star2}, which will be, henceforth, denoted by $St_6(\theta_1, \theta_2)\ (\theta_1, \theta_2 \ge \alpha_0 \textrm{ and } \theta_1 + \theta_2 \le 2\pi-3\alpha_0).$
\end{example}

\begin{figure}[H]
  \begin{center}
    \includegraphics[width=2.5in]{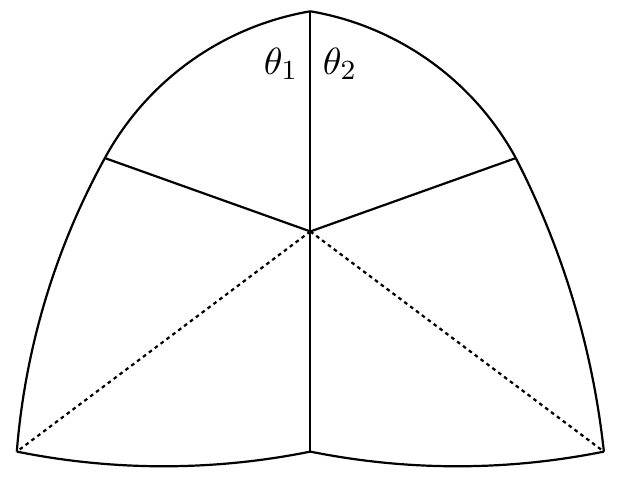}
    \caption{}
\label{fig:6star2}
\end{center}
\end{figure}

In the starting case of $\theta_1 = \theta_2 = \alpha_0$, $St_6(\alpha_0, \alpha_0) = \{ \triangle, \triangle, \Box, \Box \}$, (i.e. the same as that of Example~\ref{ex:6tri} whose complementary region is just the same as $\Omega_{4.1}$ (cf. Figure~\ref{fig:6star1}).
As $\theta_1$ (resp. $\theta_2$) gets slightly larger than $\alpha_0$, it is not difficult to check that the complementary region of $St_6(\theta_1, \theta_2)$ will be a smaller subset of $\Omega_{4.1}$ that has carved away a corresponding narrow strip along the left (resp. right) side of $\partial\Omega_{4.1}$, as indicated by the dotted lines in Figure~\ref{fig:6star1}.
Therefore, the geometric possibilities of arranging a quintuple of $\slfrac{\pi}{3}$-separated points in their complementary regions are just subsets of that of $\Omega_{4.1}$, while such arrangements exactly correspond to the possibilities of extending $St_6(\theta_1, \theta_2)$ to Type-I configurations.
For example, in the case of $\Omega_{4.1}$, the arrangements of $\{A^\prime, B^\prime, C^\prime, S, D'\}$ (resp. $\{E, F, G, H, D'\}$) correspond to the extension of the h.c.p. (resp. the 5$\Box$-Type-I with $(\alpha_0, \alpha_0, \alpha_0, \gamma, \alpha_0)$ as the angular distributions at the poles).
On the other hand, in the case of smallest complementary region of $St_6(\theta_1, \theta_2)$ with $\theta_1 + \theta_2 = 2\pi-3\alpha_0$, it is not difficult to check that such a complementary region only accommodates a unique such a quintuple which corresponds to the extension of 5$\Box$-Type-I with the angular distribution of $(\theta_1, \theta_2, \alpha_0, \alpha_0, \alpha_0)$.
This kind of geometric analysis naturally leads to the proof of the following lemma, namely

\subsubsection{Structure of non-icosahedra Type-I configurations}

\begin{lemma}
  \label{lem:nonico}
  A non-icosahedral Type-I configuration is either just such a triangulation of the f.c.c. (resp. the h.c.p.), or it is a small deformation of the 5$\Box$-Type-I configuration, namely, of the kind of 5$\boxslash$-Type-I.
\end{lemma}
\begin{proof}
  Let $S^{\prime}(\Sigma)$ be a non-icosahedral Type-I which is \textit{not} a triangulation of either the f.c.c. or the h.c.p.
  By the corollary of Lemma~\ref{lem:sixstar1}, it contains a 6$\triangle$-star which is a small deformation of $\{\triangle, \triangle, \Box, \Box\}$ such as $St_6(\theta_1, \theta_2)$ and their small deformations.

  (1) Among all possibilities of such 6$\triangle$-stars, $St_6(\alpha_0, \alpha_0)$ is the unique one with quadruple radial edges of the minimal length of $\slfrac{\pi}{3}$ which has a larger complementary region (i.e. $\Omega_{4.1}$) than that of the others.
  Therefore, one naturally expects that $\Omega_{4.1}$ will have more room of variations of arrangements of such quintuples.
  Thus, let us first analyze the range of variations of such accommodations inside of $\Omega_{4.1}$.
  Note that $D^\prime$ is a sharp corner point of $\Omega_{4.1}$.
  It is quite simple to show that any such arrangements must include a point in a very small vicinity of $D^\prime$.
  Anyhow, it is advantageous to use the polar coordinates with $D^\prime$ as the pole to analyze the geometry of $\Omega_{4.1}$.
  Then the triples $\{E, G, D^\prime\}$ (resp. $\{F, H, D^\prime\}$) lie on two longitudes with angular separation of $(2\pi-4\alpha_0)=\alpha_0 + 0.12838822$, while the triples $\{G, S, H\}$ (resp. $\{E, B^\prime, F\}$ are situated on the latitudes of distance $\slfrac{\pi}{3}$ (resp. $\slfrac{2\pi}{3}$) to the pole.
  First of all, it is easy to see that the arrangement of $\{D^\prime, E, G, F, H\}$ corresponds to the extension to a 5$\Box$-Type-I with angular distribution of $(\alpha_0, \alpha_0, \alpha_0, \gamma, \alpha_0)$.
  Moreover, $\Omega_{4.1}$ can also accommodate similar kinds of such quintuples $\{D^\prime, E^\prime, G^\prime, F^\prime, H^\prime\}$ which will produce the extensions to 5$\Box$-Type-I with angular distribution of $(\alpha_0, \alpha_0, \theta_3, \theta_4, \theta_5)$.
  In particular, we shall denote the the quintuple corresponding to that of $((\alpha_0, \alpha_0, \theta_0, \alpha_0, \theta_0))$, $\theta_0 = \pi - \frac{3}{2}\alpha_0$, by $\{D^\prime, E_0, G_0, F_0, H_0\}$.
  Just for the sake of simplicity of presentation, we shall regard the others as deformations of such a specific 5$\Box$-Type-I configuration, and then proceed to estimate the limitations on the sizes of such deformations.
  Set $\lambda_0 = 2\pi-5\alpha_0=0.12838822$, and $b_0$ to be the base length of the $\slfrac{\pi}{3}$-isosceles with $\lambda_0$ as its base angle, i.e. $b_0 = \slfrac{2\pi}{3}-0.007172022$.
  Let $\{D^*, E^*, G^*, F^*, H^*\}$ be any quintuple of $\slfrac{\pi}{3}$-separated points inside of $\Omega_{4.1}$.
  Then the geometric confinement of $\Omega_{4.1}$ implies that the angles of $\angle E^*D^*G^*$ (resp. $\angle F^* D^* H^*$) are at most equal to 0.1286, and hence the distances between $\{D^*, E^*\}$ (resp, $\{D^*, F^*\}$) are at least equal to $\slfrac{2\pi}{3}-0.0073$.
  Therefore, one has the following limitations on the sizes of variations at $\{D^\prime, E_0, G_0, F_0, H_0\}$, namely, the latitudinal differences at each point are at most equal to 0.0073 and the longitudinal differences at $\{E_0, G_0, F_0, H_0\}$ are at most equal to 0.065.

  (2) Now, let us proceed to study the geometric structures of those 6$\triangle$-stars whose complementary regions can still accommodate a quintuple of $\slfrac{\pi}{3}$-separated points.
  It is not difficult to see that such a 6$\triangle$-star will be a small deformation of Example~\ref{ex:5sq}, namely, having a pair of small adjacent triangles and a pair of $\boxslash$ sharing a short edge, such as those small deformations of Example~\ref{ex:5sq}.
  As indicated in Figure~\ref{fig:6star2}, we shall denote the boundary vertices of such as 6$\triangle$-star by $\{A_i, 1 \le i \le 6\}$, in cyclic order, and their corresponding edge lengths by $\{r_i, 1 \le i \le 6\}$, while $\{r_1, r_2, r_3, r_5\}$ are those short edge-lengths (i.e. at most only slightly longer than $\slfrac{\pi}{3}$).
  Let us first study the crucial case that the sextuple of boundary edges are all of the minimal length of $\slfrac{\pi}{3}$.
  For example, in case that $r_2 = r_5 = \slfrac{\pi}{3}$, such 6$\triangle$-stars are actually small deformations of $St_6(\theta_1, \theta_2)$ of Example~\ref{ex:5sq}. Set 
$$ \epsilon_i = r_i - \slfrac{\pi}{3}, i = 1, 2, 3, \textrm{ and } 5 $$
to be the small edge-excesses.
  Then, the same kind of geometric analysis of the complementary region as the above will show that
  \begin{enumerate}[(i)]
  \item in case $\epsilon_1 = \epsilon_3 = 0$, $\epsilon_2 + \epsilon_5 \le 0.0073$
  \item in case $\epsilon_2 = \epsilon_5 = 0$, $\epsilon_1 + \epsilon_3 \le 0.112$
  \item in general, $\epsilon_1 + \epsilon_3 + (\epsilon_2 + \epsilon_5) \le 0.112$
  \end{enumerate}
  Therefore, it follows readily that $\mathcal{S}^\prime(\Sigma)$ is necessarily a rather small deformation of a 5$\Box$-Type-I configuration.
\end{proof}

\subsubsection[Hi1]{A remarkable characterization of 5$\boxslash$-Type-I}

Just for the sake of simplicity of terminology in this paper, those Type-I configurations having quintuple of $\boxslash$'s and a pair of small 5$\triangle$-stars will be, henceforth, simply referred to as 5$\boxslash$-Type-I's. Geometrically, they are those \textit{small deformations} or \textit{close deformations} of the 5$\Box$-Type-I, meaning that those somewhat further deformations but still quite close to those 5$\Box$-Type-I. It is easy to check that such 5$\boxslash$-Type-I always have some 5$\triangle$-stars which are small or close deformations of $St_5(\theta,\theta^\prime)$.

The following lemma provides a remarkable characterization of 5$\boxslash$-Type-I, namely

\begin{lemma} \label{lem:fivestarchar}
Let $\sconfp$ be a Type-I configuration which contains a small deformation of $St_5(\theta,\theta^\prime)$. Then $\sconfp$ must be a 5$\boxslash$-Type-I.
\end{lemma}

\begin{proof}
We may assume without loss of generality of the proof that $\sconfp$ is an icosahedron (cf. Lemma \ref{lem:nonico}). Conceptually, this can also be regarded as the other case of Lemma \ref{lem:nonico}; and technically, the basic method of the proof here is also quite similar to that of Lemma \ref{lem:nonico}, i.e. the analysis of the geometry of accommodating sextuple of $\pots$-separated points inside the complementary region of such a small deformation of $St_5(\theta,\theta^\prime)$. 
%
%

(1) Therefore, one naturally begins the proof by analyzing the simplest but also most crucial case of $St_5(\alpha_0,\alpha_0)$. Again, set the center of such a star to be the north pole, and use the stereographic projection to represent the complementary region of such a star. It is, as indicated in Figure~\ref{fig:6star1}, that of $\Omega_{4,1}$ with an additional $\pots$-sector $\triangle D^\prime R A^\prime$. Now, in the geometric setting of Type-I icosahedra extensions, we are studying the problem of analyzing the possibilities of its \textit{opposite} 5$\triangle$-star (i.e. those with the sextuple vertices lying inside of $\Omega(S(\cdot))$. In this particular case of $St_5(\alpha_0,\alpha_0)$, the \textit{tightest extension}, which makes the opposite star the largest, turns out to be $St_5(\alpha_0,\gamma)$, namely, with $H$ as its center and $\{D^\prime,R,F,E,G\}$ as its boundary vertices, while the others are just its small deformations. The key-step to prove the second assertion is to show that any opposite star must also include a pair of vertices in the vicinity of $D^\prime$ and $R$ respectively, using the following simple technique of \textit{replacement subset of R in $\Omega(St(\cdot))$}, namely set it to be

\begin{equation}
S(R,\Omega )= \Omega \cap D^o ( \pots , D^{\prime}) \cap D^o( \pots , A^{\prime}) \cap D^o ( \pots ,R)
\end{equation}

\noindent
where $D^o(\pots , D^{\prime})$ and $D^o(\pots , A^{\prime})$ are the open $\pots$-disc centered at $D^{\prime}$ and $A^{\prime}$ respectively. Then we may assume that the $\pots$-separated sextuple in $\Omega$ includes $R$ itself or includes a point outside of it, say as indicated in Figure~\ref{fig:6star1} by $R^{\prime}$ or $R^{\prime \prime}$.

Next, it is straightforward to check that, in the case of $St_5(\theta,\theta^{\prime})$, $\theta + \theta^{\prime} \leq 2\pi - 3\alpha_0$, the opposite star of one of its tight extensions is $St_5(\alpha_0, 2\pi-\alpha_0-\theta - \theta^{\prime})$, while the others are just its small deformations.

(2) In summary, the tight extensions of $St_5(\theta,\theta^{\prime})$ is a 5$\Box$-Type-I with the angular distribution of $\{\theta, \theta^{\prime}, \alpha_0, \tilde{\theta}, \alpha_0\}$ at the poles, in particular, it has an opposite pair of 5$\triangle$-stars with no radial edge-excesses and with antipodal centers, and furthermore, with perfect longitudinal alignment. Therefore, in the case of a small deformation of $St_5(\theta,\theta^{\prime})$, its \textit{tight extension} is a small deformation of the 5$\Box$-Type-I with $\{\theta, \theta^{\prime}, \alpha_0, \tilde{\theta}, \alpha_0\}$ as the angular distribution, and hence, it has a pair of opposite 5$\triangle$-star with very small amount of radial edge-excesses, almost antipodal centers and small amount of angular \textit{non-alignments}, in short it is a 5$\boxslash$-Type-I.  This proves that ${\cal S}'(\Sigma)$ is also a $5\boxslash$-Type-I.  
\end{proof}

\subsection{On the geometry of Type-I icosahedra}
Recall that a Type-I configuration $\mathcal{S}^{\prime}(\Sigma)$ whose twelve stars are all of 5$\triangle$-type will be called a Type-I icosahedron.
Let us first review some generalities on the geometry of Type-I icosahedra.
\subsubsection{Some generalities} \label{subsubsec:somegeneralities}
(i) Lemma~\ref{lem:nonico} shows that Type-I non-icosahedra are, necessarily, rather small deformations of 5$\Box$-Type-I. Therefore, the great majority of Type-I configurations are icosahedra, which also include the f.c.c., the h.c.p. and those 5$\Box$-Type-I (i.e. with icosahedral triangulations.)

(ii) Let $\mathcal{S}^{\prime}(\Sigma)$ be a given Type-I icosahedron and $\{St(A_i), 1 \le i \le 12\}$ be its twelve 5$\triangle$-stars, $|St(\cdot)|$ be the area of $St(\cdot)$.
Then $$ \displaystyle\sum_{i=1}^{12}|St(A_i)| = 12\pi$$ because every triangle $\sigma_j$ of $\mathcal{S}^{\prime}(\Sigma)$ belongs to the triple of $St(\cdot)$ at its vertices, thus having $$ \displaystyle\sum_{i=1}^{12}|St(A_i)| = 3 \cdot \displaystyle\sum_{j=1}^{20} |\sigma_j| = 3 \cdot 4 \pi = 12 \pi$$ namely, the averaged value of the twelve areas of the stars of a Type-I icosahedron is equal to $\pi$.
However, the area distribution of Type-I icosahedra can be quite different for different kinds of icosahedra.
For example, both the f.c.c. and the h.c.p. icosahedra have uniform area distributions of $\pi$ for their twelve stars;
those Type-I icosahedra in the close vicinity of the f.c.c. also have almost uniform area distributions with areas close to $\pi$ for their twelve stars.
However, those 5$\Box$-Type-I icosahedra always have a pair of stars with close to the minimal area of 5$\triangle$-star (i.e. at the poles) and ten stars of areas at least equal to $\pi$.

(iii) Complementary regions of 5$\triangle$-stars and the opposite 5$\triangle$-star of a given 5$\triangle$-star in $\mathcal{S}^{\prime}(\Sigma)$:
Let $\mathcal{S}^{\prime}(\Sigma)$ be a given Type-I icosahedron and $St(A_i)$ be one of its 5$\triangle$-stars.
Then, the complementary region of $St(A_i)$ can be defined just the same way as in the case of 6$\triangle$-stars and will again be denoted by $\Omega(St(A_i))$.
For a given Type-I icosahedron $\mathcal{S}^{\prime}(\Sigma)$, to every 5$\triangle$-star, say $St(A_i)$, there is a unique opposite star, say $St(A'_{i})$, whose vertices are situated inside of $\Omega(St(A_i))$ which will be, henceforth, referred to as the \textit{opposite star} of $St(A_i)$.
For example, the pair of small 5$\triangle$-stars of a 5$\Box$-Type-I (resp. 5$\boxslash$-Type-I) icosahedra are opposite stars of each other.
\begin{example}
  Set the label of vertices of a given 5$\Box$-Type-I icosahedron to be $\{A_i, A_i^{\prime}; 0 \le i \le 5\}$ with $\{A_0, A_0^{\prime}\}$ at the poles and $\{A_i, A_i^{\prime}; 1 \le i \le 5\}$ having the same longitudes.
  Then $St(A_i)$ and $St(A^{\prime}_{i+2})$ are opposite stars of each other.
  Set $\{\theta_i\}$ to be the angular distribution of $St(A_0)$.
  Then $St(A_i)$ consists of a pair of small triangles and a triple of half rectangles, namely $\slfrac{\pi}{3}$-isosceles with top angles of $\theta_{i-1}$ and $\theta_i$, 2$\tilde{\sigma}_{\theta_{i-1}}$ and another $\tilde{\sigma}_{\theta_{i}}$.
  Therefore, their areas are at least equal to $\pi$ and equal to $\pi$ when and only when $\theta_{i-1} = \theta_i = \alpha_0$.
  We shall denote such a 5$\triangle$-star by $St_5(\theta_1, \theta_2)$.
  Then, the complementary region of $St_5(\theta_1, \theta_2)$ is given by that of $St_6(\theta_1, \theta_2)$ with an additional circular sector.
  For example, the complementary region of $St_5(\alpha_0, \alpha_0)$ is as indicated in Figure~\ref{fig:6star1}.
\end{example}
(iv) Small 5$\triangle$-stars and their complementary regions:
Let $\epsilon_i = (r_i - \slfrac{\pi}{3})$ be the radial edge lengths in excess of $\slfrac{\pi}{3}$ of a given 5$\triangle$-star, the total amount of such excesses will be, henceforth, referred to as the \textit{total radial edge-excess} of the given 5$\triangle$-star and will be denoted by $\epsilon(St(\cdot))$.
Just for the sake of definitiveness in presentation, we introduce the following quantitative definition of small 5$\triangle$-stars, namely
\begin{definition}
  A 5$\triangle$-star is defined to be a small 5$\triangle$-star if its total radial edge-excess is at most equal to 0.104.
\end{definition}
\begin{example}
  Among all 5$\triangle$-stars of a given total radial edge-excess, say $\overline{\epsilon} = \slfrac{1}{5}\epsilon \le 0.0208$, the following specific one will have the minimal area, as well as the maximum averaged density, namely, its radial edge-excess is concentrated in one and its boundary edge-excess is also concentrated in one of its $\slfrac{\pi}{3}$-isosceles.
\end{example}

\subsection{Geometry of Type-I icosahedra with rather lopsided area distributions of stars}

Let $\{ |St(A_i)|, A_i \in \Sigma \}$ be the area distribution in a given Type-I icosahedron $\mathcal{S}(\Sigma)$. Note that the averaged value of such an area distribution is always equal to $\pi$, while its individual values can be as small as ($\pi - 0.345$) and as large as ($\pi$ + 0.822). Geometrically, 5$\triangle$-stars with rather small total amount of radial edge-excesses, say denoted by $\sum \epsilon_i$, are those stars with areas quite close to ($\pi - 0.345$). In this subsection, we shall call those 5$\triangle$-stars with $\sum \epsilon_i$ at most equal to 0.104 \textit{small stars}, and on the other hand, call those with their areas exceeding ($\pi$ + 0.6) stars with rather large areas; while Type-I icosahedra with both small stars and stars with rather large areas will be referred to as Type-I icosahedra with rather lopsided area distributions. Let us begin with some specific examples of such icosahedra.

\subsubsection{Tight Extensions of small stars}

Let $St(N)$ be a small star centered at the north pole $N$. Set $\{ A_i, 1 \leq i \leq 5 \}$ (resp. $\{r_i$, $\{\theta_i\}$ and $\{b_i\}$) to  be its boundary vertices (resp. radial edge-lengths, central angles and boundary edge-lengths) and $\{V_i\}$ to be the vertices of its \textit{complementary region} $\Omega\left(St(N)\right)$. Set $\{\epsilon_i = r_i -\slfrac{\pi}{3}\}$ to be the \textit{radial edge-excesses}. Note that

\begin{equation}
\sum \epsilon_i \leq 2 \sin^{-1} \left( \slfrac{\sqrt{3}}{2}\sin(2\alpha_0)\right) - \slfrac{\pi}{3} = 0.10398539 \sim 0.104
\end{equation}

\noindent
will ensure that $\{V_i\}$ are always $\slfrac{\pi}{3}$\textit{-separated}, and hence, one may extend such a $St(N)$ to Type-I icosahedra simply by adding an opposite star with $\{V_i\}$ as the boundary vertices, while its center can vary in the vicinity of the south pole. Anyhow, we shall call such extensions \textit{tight extensions} of $St(N)$. It is easy to see that such extensions achieve the \textit{maximal area} of the \textit{opposite star} of $St(N)$ which will be henceforth referred as the tight extensions of the small star $St(N)$ whose geometric structures are uniquely determined, modulo the positions of their centers, and in particular, the geometric structure of the subconfiguration of the union of the sextuplet of stars centered at $\{N; \> A_i, \> 1 \leq i \leq 5\}$.

Geometrically, the set of $\{r_i, \theta_i{;} \> 1 \leq i \leq 5 \}$(resp. $\{r_i, b_i{;} \> 1 \leq i \leq 5 \}$) already constitutes a complete set of geometric invariants of $St(N)$, namely, that of the SAS (resp. SSS) type, while following quintuples of lengths, namely
\begin{equation*}
d_j = \overline{NV_i} \quad \hbox{and} \quad \overline{V_{i-1}V_i} = \tilde{b_i}
\end{equation*}
are the collection of lengths besides those $\slfrac{\pi}{3}$ edges of such subconfigurations of its tight extensions. For example, the additional ten spherical triangles of such a subconfiguration consists of quintuples of $\slfrac{\pi}{3}$-isosceles with $\{ b_i\}$(resp. $\{ \tilde{b}_i\}$) as their base-lengths, while the weighted average of $\{ \tilde{\rho} \left( St(N) \right) , \> \tilde{\rho} \left( St(A_i) \right) , \> 1 \leq i \leq 5 \}$ is also equal to the further weighted average of that of the above quintuples of $\slfrac{\pi}{3}$-isosceles (i.e. with $\{b_i\}$ and $\{ \tilde{b}_i \}$ as their base-lengths) and $\tilde{\rho} \left( St(N)\right)$ with weights $\{2,1,3\}$. The first step of analyzing geometric invariants of such a tight extension is, of course, to express $\{ \cos b_i, \> \cos \tilde{b}_i, \cos d_i\}$ in terms of the SAS data of $\{r_i, \theta_i\}$, which has the additional advantage of $\Sigma \theta_i = 2\pi$. Anyhow, it is quite straightforward to first compute $\cos b_i$ by the cosine law and then use the determinant formula for quadrilateral relations (cf section 3.2.3) to compute $\cos d_i$ and $\cos \tilde{b}_i$. We only include here the following rather specific simple examples.

\begin{example}
In the special ``starting" case of 5$\triangle$-stars with uniform $\slfrac{\pi}{3}$-radial edges and with $S$ as the center of its opposite star, the above computations become much simpler than otherwise, namely
\end{example}

\begin{equation}
\left\{
\begin{array}{ll}
\cos b_ i & = \frac{3}{4}\cos\theta_i + \frac{1}{4},\\
\tan\frac{d_i}{2} & = \sqrt{3}\cos\frac{\theta_i}{2}, \> \cos d_i  = \frac{1-\tan^2\frac{d_i}{2}}{1+\tan^2\frac{d_i}{2}},\\
\cos\tilde{b}_i & = \sin d_i \sin d_{i+1}\cos\frac{\theta_i + \theta_{i+1}}{2} + \cos d_i \cos d_{i+1}, \> \sin d_i  = \frac{2\tan\frac{d_i}{2}}{1+\tan^2\frac{d_i}{2}}
\end{array}
\right.
\end{equation}

Based upon the above basic invariants of SSS type for triangles of $\mathcal{S}^{\prime}(\Sigma)$ of the tight extension of such a special 5$\triangle$-star, it is straightforward to compute other geometric invariants of such an $\mathcal{S}^{\prime}(\Sigma)$, in particular the $\bar{\rho}$($\cdot$) as an explicit function of the angular distribution $\{ \theta_i \}$. Furthermore, it is not difficult to show that such a function of $\{ \theta_i \}$ will be minimal in the case of uniform angular distribution and maximal in the case of most lopsided angular distributions.

\begin{example} Let $St(N)$ be a small star whose radial (resp. boundary) edge-excesses are concentrated in a single radial edge(resp. in the base of an $\slfrac{\pi}{3}$-isosceles) such as the St$_5^o$($\theta$) indicated in Figure~\ref{fig:5star2}-(ii) with $\alpha_0 \leq \theta \leq \gamma$ or the other assemblage of the same quintuple of triangles.

\begin{figure}[H]
  \begin{center}
    \begin{tikzpicture}
      \node[anchor=south west,inner sep=0] (image) at (0,0) {\includegraphics[width=6in]{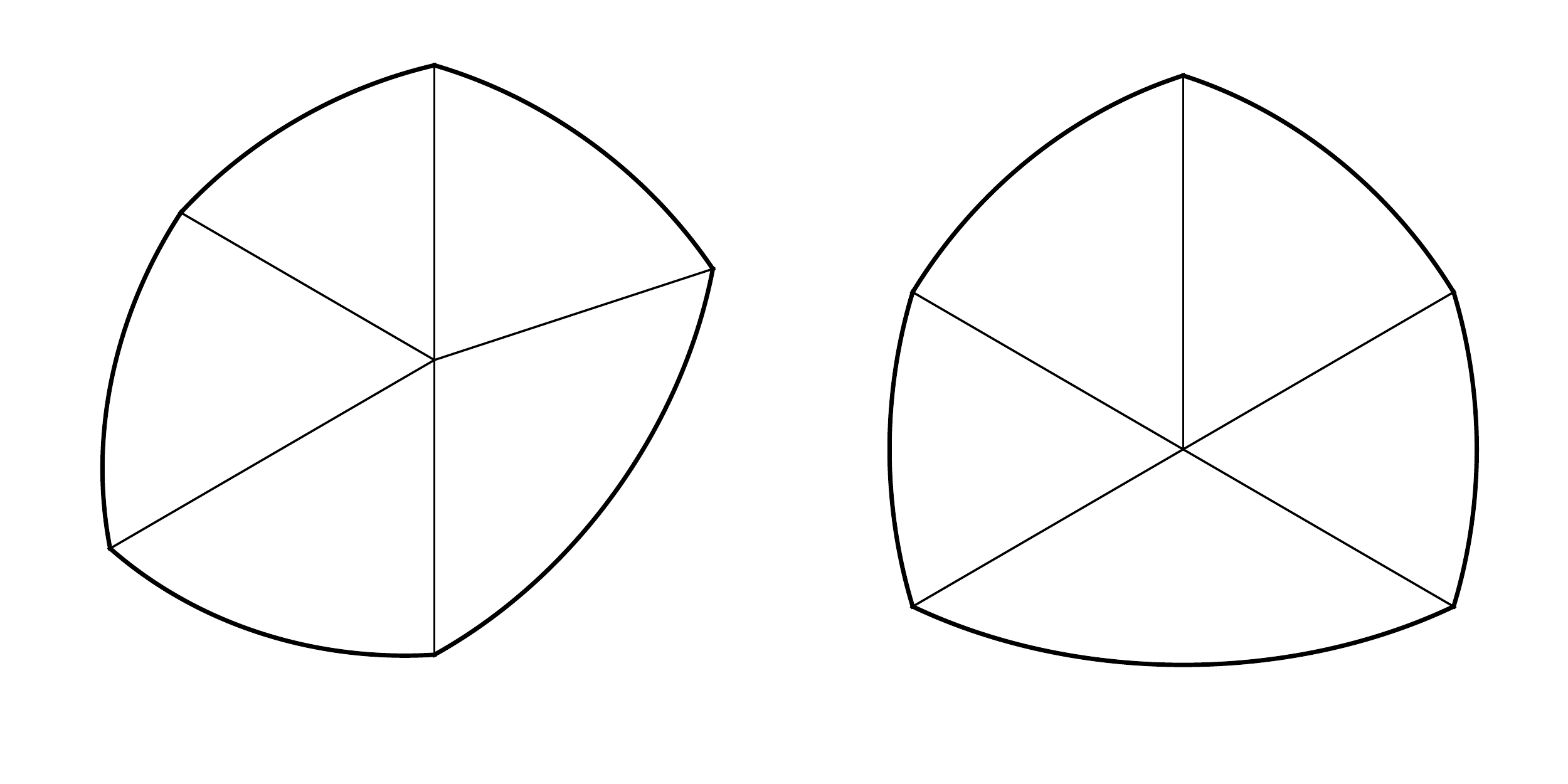}};
      \begin{scope}[x={(image.south east)},y={(image.north west)}]
        \node at (0.28,0.00) {(i)};
        \node at (0.00,0.75) {$St_5(\theta_1,\theta_2):$};
        \node at (0.25,0.65) {$\sigma_{\theta_2}$};
        \node at (0.32,0.65) {$\sigma_{\theta_1}$};
        \node at (0.32,0.48) {$\tilde{\sigma}_{\theta_1}$};
        \node at (0.24,0.44) {$\tilde{\sigma}_{\theta_2}$};
        \node at (0.21,0.54) {$\tilde{\sigma}_{\theta_2}$};
        \node at (0.76,0.00) {(ii)};
        \node at (0.55,0.75) {$St_5^\circ(\theta):$};
        \node at (0.72,0.54) {$\varphi$};
        \node at (0.79,0.54) {$\varphi$};
        \node at (0.70,0.72) {$\sigma_{\theta}$};
        \node at (0.81,0.72) {$\sigma_{\theta}$};
        \node at (0.88,0.63) {$\theta$};
        \node at (0.70,0.42) {$\alpha_0$};
        \node at (0.81,0.42) {$\alpha_0$};
        \node at (0.63,0.40) {$\sigma_{\alpha_0}$};
        \node at (0.88,0.40) {$\sigma_{\alpha_0}$};
        \node at (0.75,0.35) {$\overline{\theta}$};
        \node at (0.75,0.22) {$\sigma_{\overline{\theta}}$};
      \end{scope}
    \end{tikzpicture}
    \caption{\label{fig:5star2}}
  \end{center}
\end{figure}

Note that there are only small differences between the tight extension of $St_5^\circ(\theta)$, $\alpha_0 \leq \theta \leq \gamma$ and that of the other, while their $\bar{\rho}(\cdot)$ are almost the same. Therefore, we shall only exhibit that of the former and compute its $\bar{\rho}(\cdot)$ as a function of $\theta$ in the following. Set $r_i$ to be the only radial edge with excess. Then, one has $\{ r_i, \theta_i \}$ of $St(N)$ given as follows

\begin{equation}
\begin{array}{l}
\cos r_1 = \frac{1}{4} (3 \cos\theta + 1)\\
\theta_1 = \theta_5= \tan^{-1} 2 \cot \frac{\theta}{2}\\
\theta_2 = \theta_4=\alpha_0, \> \theta_3=2\pi - 2\alpha_0 - 2 \tan^{-1} 2 \cot \frac{\theta}{2}\\
\end{array}
\end{equation}
and moreover,

\begin{equation}
\begin{array}{l}
\cos b_3 = \frac{1}{4}(3\cos \theta_3 +1), \> b_i = \frac{\pi}{3} \> \hbox{for} \> i \neq 3\\
\cos d_1 = \cos d_5 = \frac{1}{4} \left(3 \cos(\theta + \alpha_0) +1 \right)\\
\cos d_2 = \cos d_4 = - \frac{1}{3}\\
(1 + \cos d_3)(1 + \cos b_3) = 1, \> \hbox{or} \> \cos d_3 = \frac{-\cos b_3}{1+ \cos b_3}
\end{array}
\end{equation}

Set $\tilde{\theta}_i$ to be the top angle of the $\slfrac{\pi}{3}$-isosceles with $\tilde{b}_i$ as the base-length. Then,

\begin{equation}
\begin{array}{l}
\tilde{\theta}_1 = 2\pi - 2\alpha_0 - 2\theta_1\\
\tilde{\theta}_2 = \tilde{\theta}_5 = 2\pi - 3\alpha_0 -\theta\\
\tilde{\theta}_3 = \tilde{\theta}_4 = 2\pi - 2\alpha_0 -2\tan^{-1} 2 \cot \frac{\theta_3}{2}
\end{array}
\end{equation}

while

\begin{equation}
\cos \tilde{b}_i = \frac{1}{4}(3 \cos \tilde{\theta}_i +1), \> 1 \leq i \leq 5  
\end{equation}

Now, with the above set of basic geometric invariants of the tight extension of $St_5^\circ(\theta)$, $\alpha_0 \leq \theta \leq \gamma$, at hand, and assuming $S$ to be the center of its opposite star, it is straightforward to compute its $\bar{\rho}(\cdot)$ as a function of $\theta$, thus having computer graphic to such a function as indicated in Figure~\ref{fig:5starplot}. 

\label{ex:tightextension}
\end{example}

\begin{figure}[H]
  \begin{center}
    \includegraphics[width=6in]{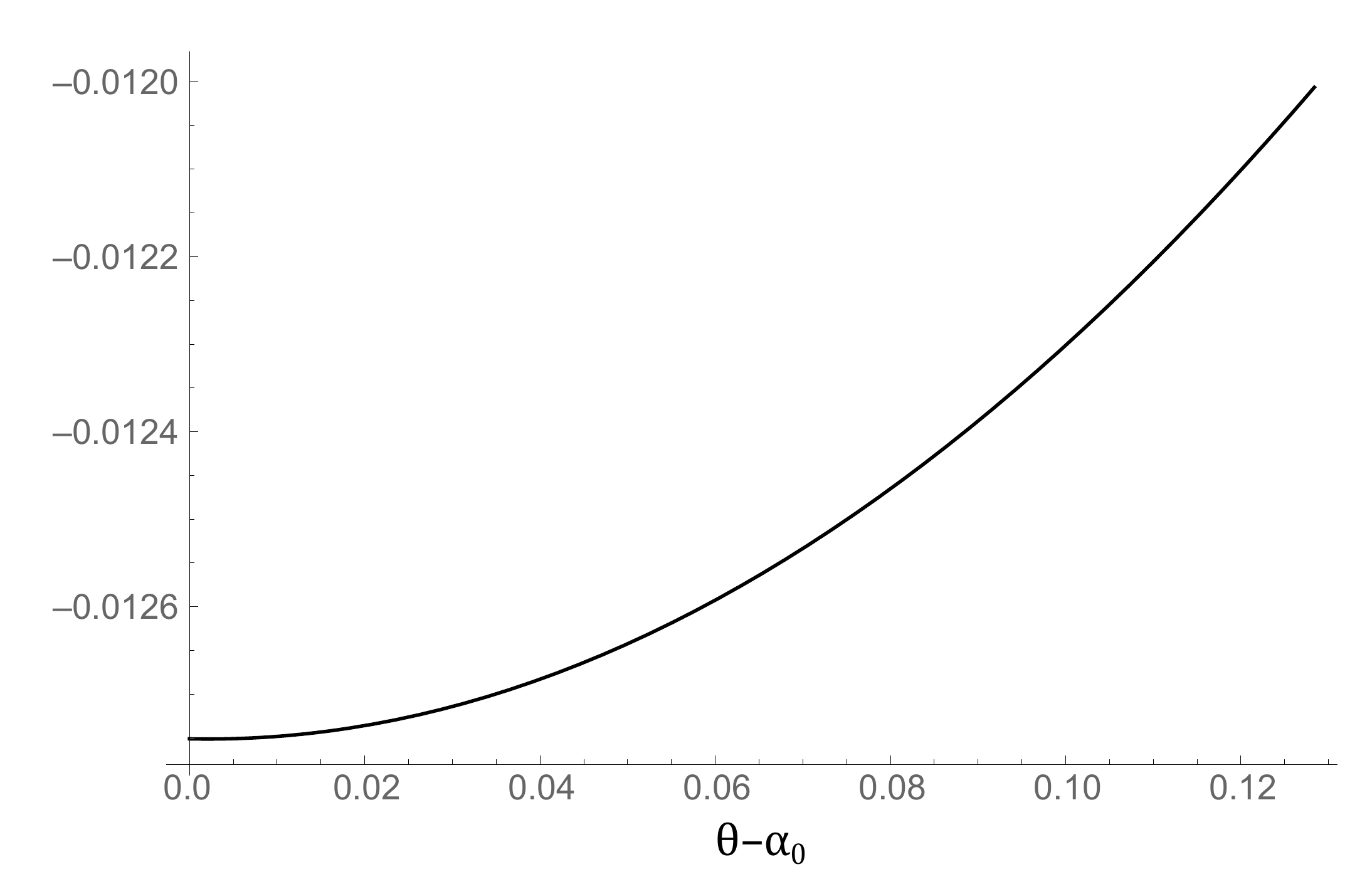}
    \caption{\label{fig:5starplot} Graph of $\bar{\rho}(\cdot)-\frac{\pi}{\sqrt{18}}$ of Example~\ref{ex:tightextension}.}
  \end{center}
\end{figure}

\begin{example}{\it Another specific type of Type-I icosahedra with an opposite pair of small stars and accommodating a large star with higher $\tilde{\rho}(\cdot)$}
  \label{ex:smallstar}

Let $\sconf$ be an icosahedron containing a star with minimal area, say $St(N)$, (i.e. with uniform $\slfrac{\pi}{3}$-radial edges and the central angular distribution of quadruple $\alpha_0$ and $\gamma$), while its opposite star, say $St(N')$, is a small star consisting a pair of $\slfrac{\pi}{3}$-equilateral quadrilaterals and another $\slfrac{\pi}{3}$-isosceles and moreover, the triple of its vertices ``opposite'' to the base edge of its $\slfrac{\pi}{3}$-isosceles are situated on the boundary of the complementary region of $St(N)$. As indicated in Figure~\ref{fig:smallstar}, the geometry of such a Type-I icosahedron can be advantageously analyzed in the setting of the kind of spherical cartesian representation with the equator corresponding to $N$ as the $x$-axis, better exhibiting the geometry of the collar region connecting the pair of small stars, i.e. the complementary of $St(N)\cup St(N')$: 
\begin{figure}[H]
  \begin{center}
    \includegraphics[width=6in]{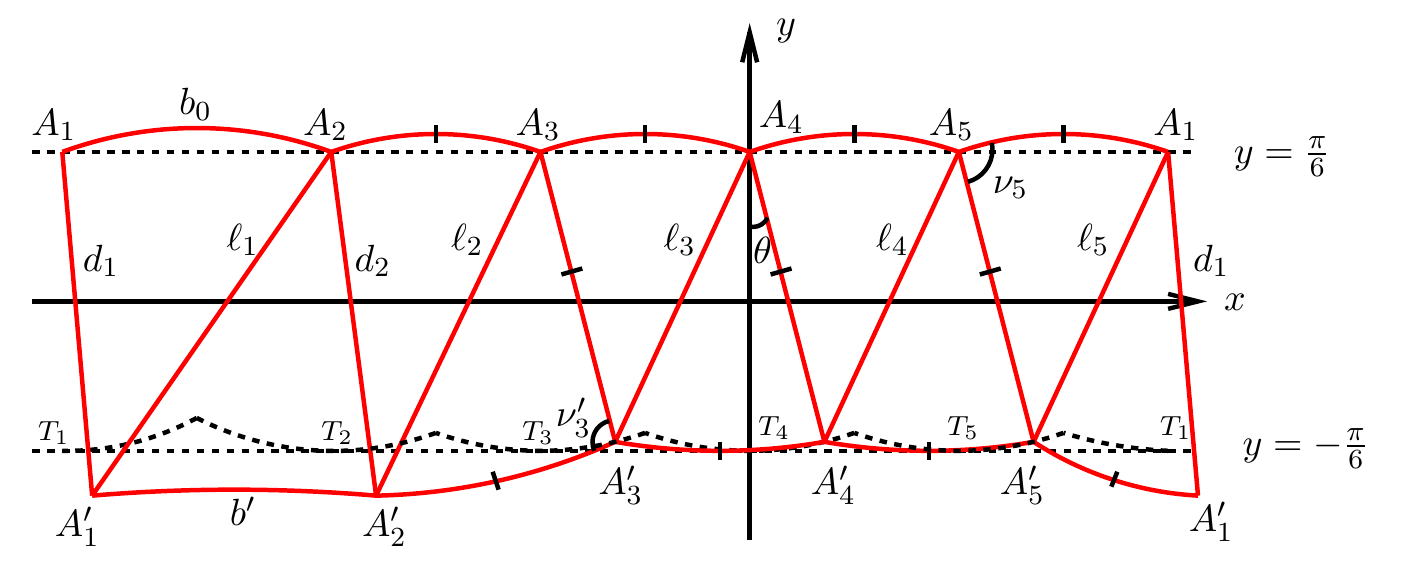}
    \caption{\label{fig:smallstar} ($\partial\Omega$ is represented by those $\frac{\pi}{3}$-circular arcs with $T_i$ as the touching points with $y=-\frac{\pi}{6}$.
$\partial St(N)$ (resp. $\partial St(N')$) are represented by those solid segments above (resp. below) the $x$-axis.)}
  \end{center}
\end{figure}
As indicated in Figure~\ref{fig:smallstar}, all those unmarked segments of the collar region of $\sconf$ are of length $\slfrac{\pi}{3}$, which consists of quintuple of $\{\boxslash_i\}$ with $\{\ell_i,\>1\leq i\leq 5\}$ as their cutting diagonals respectively. In particular $\boxslash_3$ and $\boxslash_4$ are $\slfrac{\pi}{3}$-equilaterals, whose geometry is determined by just one of their angles. Note that the outer angles of the collar region at $\{A_3,A_4,A_5\}$ are all equal to $2\alpha_0$ while that of $A_1$ and $A_2$ are given by
\begin{equation}
\alpha_0+2\arctan(2\cot\frac{\gamma}{2}).
\end{equation}
Set $\angle T_4A_4A_4'=\theta$. Then the angles of $\boxslash_3$ (resp. $\boxslash_4$) at $A_4$ are given by $\beta_0+\theta$ (resp. $\beta_0-\theta$), $\beta_0=\pi-\alpha_0$. Therefore the angles of $\boxslash_3$ (resp. $\boxslash_4$) at $A_4'$ are given by
\begin{equation}
\left\{\begin{array}{l}  
2\arctan(2\cot\frac{\pi+\theta-\alpha_0}{2}),\\ 
2\arctan(2\cot\frac{\pi-\theta-\alpha_0}{2})\end{array}\right.
\end{equation}
Therefore, the outer angle at $A_4'$ is given by
\begin{equation}
f(\theta)=2\left\{\pi-\arctan(2\cot\frac{\pi+\theta-\alpha_0}{2})-\arctan(2\cot\frac{\pi-\theta-\alpha_0}{2})\right\}
\end{equation}
for the range $0\leq \theta\leq 2\alpha_0-\beta_0$. It is easy to check that $f(\theta)$ is a monotonically increasing function with $f(0)=2\alpha_0$, which is just the $(\lambda_1+\lambda_2)$ of the pair of angles of $St(N')$ at $A_4'$.

Note that the geometry of such an $\sconf$ is completely determined by the chosen of $\{\theta,\lambda_1,\lambda_2\}$, while $\lambda_1+\lambda_2=f(\theta)$, and hence, by the chosen parameter of $0\leq\theta\leq 2\alpha_0-\beta_0$ and $\alpha_0\leq \lambda_1\leq f(\theta)-\alpha_0$. Therefore, from now on, we shall specify such an Type-I icosahedron as ${\cal S}(\Sigma(\theta,\lambda_1))$ and proceed to compute those edge lengths of the collar region other than the $\slfrac{\pi}{3}$-ones.
\begin{itemize}
\item[(i)] First of all, one has
\begin{equation}
\left\{\begin{array}{rcl}
b_0&=& 2\arcsin(\frac{\sqrt{3}}{2}\sin\frac{\gamma}{2}),\\
b'&=& 2\arcsin\left\{\frac{\sqrt{3}}{2}\sin\left(\pi-\arctan(2\cot\frac{\lambda_1}{2})-\arctan(2\cot\frac{\lambda_2}{2})\right)\right\},\\
\ell_4&=& 2\arcsin\frac{\sqrt{3}}{2}\sin\frac{1}{2}(\pi-\alpha_0-\theta),\\
\ell_3&=&2 \arctan\sqrt{3}\cos\frac{1}{2}(\pi-\alpha_0+\theta)
\end{array}\right.
\end{equation}
And the outer angles of the collar region at $\{A_1',A_2',A_3',A_5'\}$, say denoted by $\mu_i'$ ($i\neq 4$), are simply given as follows:
\begin{equation}
\begin{array}{rcl}
\mu_1'&=& \lambda_1+\arctan(2\cot\frac{\alpha_5'}{2})\\
\mu_2'&=& \lambda_2+\arctan(2\cot\frac{\alpha_5'}{2})
\end{array}
\end{equation}
where
\begin{equation}
\frac{\alpha_5'}{2}=\pi-\arctan(2\cot\frac{\lambda_1}{2})-\arctan(2\cot\frac{\lambda_2}{2})
\end{equation}
and
\begin{equation}
\mu_3'=2\arctan(2\cot\frac{\lambda_2}{2}),\quad \mu_5'=2\arctan(2\cot\frac{\lambda_1}{2}).
\end{equation}

\item[(ii)] Next let us compute $\ell_2$ (resp. $\ell_5$) which are respectively the base lengths of the $\slfrac{\pi}{3}$-isosceles $\triangle A_3A_2'A_3'$ and $\triangle A_5'A_1A_5$ whose top angles are given by
\begin{equation}
\begin{array}{ll} & \nu_3'=2\pi-\mu_3'-(\pi-\alpha_0+\theta)=\pi+\alpha_0-\mu_3'-\theta \\
\hbox{\big(resp. }& \nu_5=2\pi-2\alpha_0-2\arctan(2\cot\frac{\pi-\alpha_0-\theta}{2}) \big)
\end{array}
\end{equation}
\item[(iii)] Finally, it is straightforward to compute the remaining edge lengths of $\boxslash_1$, namely, $\{d_2,d_1\hbox{ and }\ell_1\}$ by applying cosine law to $\triangle A_2A_3A_2'$, $\triangle A_1A_5'A_1'$ and $\triangle A_1'A_1A_2$.
\end{itemize}
\end{example}

With such explicit formulas of the geometric structure of ${\cal S}(\Sigma(\theta,\lambda_1))$ at hand, it is straightforward to use computer to provide numerical computations of the following two averaged densities, namely
\begin{equation}
\begin{array}{rcl}
\overline{\rho}(\theta,\lambda_1)&=& \overline{\rho}({\cal S}(\Sigma(\theta,\lambda_1))) \\
\tilde{\rho}(\theta,\lambda_1)&=& \tilde{\rho}(St(A_2)\subset {\cal S}(\Sigma(\theta,\lambda_1)))
\end{array}
\end{equation}
as well as $A(\theta,\lambda_1)=|St(A_2)|$, thus obtaining their graphs and that of the weighted ratio between $(\slfrac{\pi}{\sqrt{18}}-\tilde{\rho}(\theta,\lambda_1))$ and $(A(\theta,\lambda_1)-\pi)$ with weights given by $\tilde{w}(\cdot)$, exhibited in Figure~\ref{fig:smallstarplot}.

\begin{figure}[H]
  \begin{center}
    \includegraphics[width=6.5in]{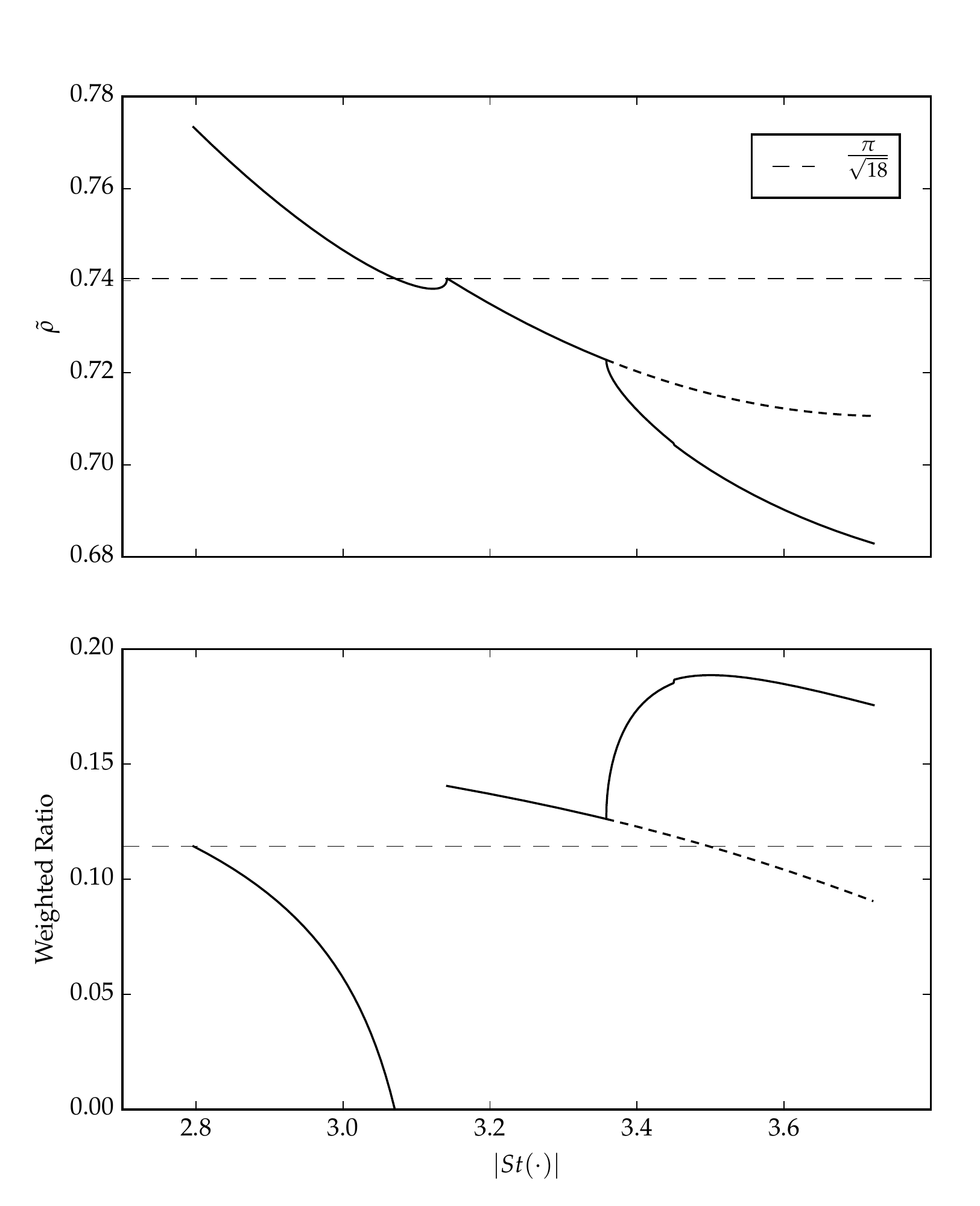}
    \caption{\label{fig:smallstarplot} Graph of $\tilde{\rho}(\cdot)$ and weighted ratio.}
  \end{center}
\end{figure}

%% file: 5-techniques/5-techniques.tex
\section{Techniques of collective area-wise estimates for clusters of triangles}
\label{sec:techniques}

In this section, we shall develop some pertinent techniques of \textit{collective areawise estimates} for the weighted averages of $\{ \rho (\sigma_j), \> \sigma_j \> \epsilon \> \sconfp\}$ over suitable clusters of triangles, such as $\boxslash$-pairs, lune clusters or star clusters, which will provide the kind of powerful analytical tools to achieve a clean cut proof of Theorem I for the major case of Type-I local packings jointly with the geometric insights of  \S \ref{sec:review}. 

Let $\sconf$ be a Type-I configuration and $\mathcal{L}$ be the tightest Type-I local packing with $\Sigma$ as the touching points of the twelve touching neighbors. Then, the locally averaged density $\bar{\rho}(\mathcal{L})$ is defined to be the specific weighted average of the triangular invariants of $\{\rho(\sigma_j), \> w(\sigma_j) {;} \sigma_j  \in \sconfp \}$, while the proof of Theorem I for the most important case of Type-I local packings amount to proof that $\bar{\rho}(\cdot){:} \> \mathcal{M}_I \rightarrow \mathbb{R}^+$ has the f.c.c. and the h.c.p. as the unique two maximum points with $\bar{\rho}(\cdot)=\slfrac{\pi}{\sqrt{18}}$. 

Geometrically, the f.c.c. and the h.c.p. are the two outstanding \textit{singular points} and those 5$\Box$-Type-I's constitutes an outstanding \textit{singular subvariety} in the moduli space $\mathcal{M}_I$, while the h.c.p. is an isolated point, the f.c.c. is a cusp point with 1-dimensional tangent cone and with a substantial neighborhood consisting of 6$\boxslash$-Type-I's; and moreover the 5$\Box$-Type-I's have a quite extensive neighborhood consisting of 5$\boxslash$-Type-I's. From the point of view of optimal estimation of $\bar{\rho}(\cdot)$, the 6$\boxslash$-Type-I's and the 5$\boxslash$-Type-I's are outstanding because, as it turns out, their $\bar{\rho}(\cdot)$ are in fact higher than that of the others. 

\subsection[Hi4]{Area-wise estimates for the proof of Theorem I in the subsets of 6$\boxslash$-Type-I's and 5$\boxslash$-Type-I's}
\label{subsec:proofthrm1}

(1) {\it The Estimate of $\bar{\rho}$($\cdot$) for 6$\boxslash$-Type-I's} 
\medskip

A 6$\boxslash$-Type-I configuration consists of octuple of small triangles and sextuple of $\boxslash$, it follows readily from the ($k, \delta$)-analysis of \S \ref{subsubsec:areapreserve} that the weighted average of $\rho$($\sigma_j$) for the two halves of a $\boxslash$ is always substantially less than that of $\rho$($\Box$) (by a linear factor of the angular decrement of $C$), while that of the small triangles are, of course, less than $\rho$($\sigma_{\alpha_0}$). Therefore, $\bar{\rho}$($\cdot$) of a 6$\boxslash$-Type-I's is always less than $\slfrac{\pi}{\sqrt{18}}$. in fact, by a decrement of linear proportion to the deformation. Thus, the f.c.c. is a {\it local maximum of cusp-type} of the $\bar{\rho}$($\cdot$)-function.  

\medskip

\noindent(2) {\it The Estimate of $\bar{\rho}$($\cdot$) for 5$\boxslash$-Type-I's}
\medskip

First of all, let us provide an optimal estimate of the restriction of $\bar{\rho}$($\cdot$) to the subvariety of 5$\Box$-Type-I's , which is a function of the angular distributions $\{\theta_i\}$  at the poles. A 5$\Box$-Type-I $\sconf$ is the assemblage of quintuple of lune clusters $\{ L_{\theta_i}\}$, $L_{\theta_i}=\{2\sigma_{\theta_i}, 2\tilde{\sigma}_{\theta_i}\}$. Set

\begin{equation} \label{eq:5.1}
\tilde{\rho}(L_\theta):=\frac{w(\sigma_\theta)\rho(\sigma_\theta) + w(\tilde{\sigma}_\theta)\rho(\tilde{\sigma}_\theta)}{w(\sigma_\theta) + w(\tilde{\sigma}_\theta)}, \quad \alpha_0 \leq \theta \leq \gamma
\end{equation}

Then it is straightforward to check that $\tilde{\rho}(L_\theta), \> \alpha_0 \leq \theta \leq \gamma$, is a convex function of $\theta = \slfrac{1}{2}$ Area($L_\theta$). Therefore $\bar{\rho}\left(\sconf\right)$ is \textit{minimal} in the case of uniform angular distribution (i.e. $\theta_i = \slfrac{2\pi}{5}$) on the one hand and on the other hand it is \textit{maximal} in the case of the most lopsided angular distribution, namely, with quadruple of $\alpha_0$ and $\gamma$, where $\hat{\rho}(\cdot)$ is equal to $\slfrac{\pi}{\sqrt{18}}- 0.000578464$.

Next, let us consider the upper bound estimate of $\rhobarred$ for those small or even just close deformations of 5$\Box$-Type-I's, namely, those 5$\boxslash$-Type-I's, which are, again, assemblages of \textit{lune clusters}, each of them consisting of a $\boxslash$ and a pair of small triangles, say denoted by $\{ \tilde{L}_i \}$. Set 

\begin{equation} \label{eq:5.2}
\tilde{\theta}_i = \slfrac{1}{2} \> \mathrm{area} \> \mathrm{of} \>  \tilde{L}_i, \> \sum \tilde{\theta}_i = 2\pi
\end{equation}

Then, a direct application of the ($k, \delta$)-estimates to the pairs of small triangles and $\boxslash$ will again provide the following upper bound estimate $\tilde{\rho}(\tilde{L}_i)$, namely 

\begin{equation}\label{eq:5.3} 
\left\{
\begin{array}{lr}
\tilde{\rho}(\tilde{L}_i) \leq \tilde{\rho}({L}_{\alpha_0}) & \mathrm{if} \> \tilde{\theta}_i \leq \alpha_0\\
\tilde{\rho}(\tilde{L}_i) \leq\tilde{\rho}({L}_{\tilde{\theta}_i}) & \mathrm{if } \> \tilde{\theta}_i > \alpha_0\
\end{array}
\right.
\end{equation}

\noindent and equality holds only when $\tilde{L}_i = L_{\alpha_0}$ (resp. $L_{\tilde{\theta}_i}$). Therefore, the upper bound estimate of $\rhobarred$ for 5$\Box$-Type-I's will also be the upper bound estimate of $\rhobarred$ for 5$\boxslash$-Type-I's.

\subsection[Hi4]{Area-wise estimates for 5$\triangle$-star clusters of Type-I icosahedra}
\label{subsec:5estimates}

Note that, except some of those small deformations of 5$\Box$-Type-I, all the others are Type-I icosahedra. Therefore, the remaining cases of the proof of Theorem I for Type-I local packings are that of Type-I icosahedra, whose $\rhobarred$ can be expressed as the following weighted average of $\{\tilde{\rho}\left(St(\cdot)\right), \> A_i \in \Sigma\}$. 
\begin{equation}
  \slfrac{\displaystyle\sum_{j=1}^{20}w^j\rho(\tau(\sigma_j))}{\displaystyle\sum_{j=1}^{20}w^j} =
  \slfrac{\displaystyle\sum_{i=1}^{12}\tilde{w}_i\tilde{\rho}(St(A_i))}{\displaystyle\sum_{i=1}^{12}\tilde{w}_i}
\end{equation}
where
\begin{equation}
  \tilde{\rho}(St(A_i)) = \displaystyle\sum_{\sigma_j \in St(A_i)}\slfrac{w^j\rho(\tau(\sigma_j))}{\tilde{w}_i}, \quad
  \tilde{w}_i = \displaystyle\sum_{\sigma_j \in St(A_i)}w^j
\end{equation}
Therefore, one naturally expects that collective \textit{area-wise} estimates of $\rhostar$ for 5$\triangle$-stars of Type-I icosahedra will be a powerful technique for such a proof. Anyhow, this naturally leads to the discovery of Lemma 4 and Lemma 4$'$. This will be the major topic of this section.

\subsubsection{Extremal Triangles and extremal stars}
\label{subsubsec:extremals}

In the study of the problem of collective areawise estimates for stars in a Type-I icosahedra, it is natural to embark our journey by the exploration of possible candidates of ``\textit{extremal stars}", achieving the optimal $\tilde{\rho}(\cdot)$ for such stars with a given total area $|St(\cdot)|$. Recall that the clean-cut organization of area preserving deformations of spherical triangles and the ($k, \> \delta$)-analysis of \S \ref{subsubsec:areapreserve} not only show that the \textit{area-wise maximalities} of $\rho(\sigma)$ (resp. $w(\sigma)$) for those extremal shapes (i.e. $\sigma_\theta$ for areas $\leq 2 \sin^{-1}\slfrac{1}{3}$ and $\tilde{\sigma}_\theta$ for areas $\geq 2\sin^{-1}\slfrac{1}{3}$ up to 0.92), but also provide effective estimates for others in terms of ($k, \> \delta$)-shape invariants. Anyhow, this individual estimates naturally leads to the construction (or rather, discovery) of the following examples, namely 

\begin{example}
Both $|St_5^o (\theta)|$ and $\tilde{\rho}\left(St_5^o (\theta)\right)$, as indicated in Figure~\ref{fig:5star2}, are given by the following explicit formulas as functions of $\alpha_0 \leq \theta \leq \pi - \alpha_0$, namely

\begin{equation} \label{eq:5.4}
\begin{array}{lcl}
|St_5^o (\theta)| & = & 2|\sigma_{\alpha_0}| + 2|\sigma_\theta| + |\sigma_{\bar{\theta}}|, \> \bar{\theta} = 2\left(\pi - \alpha_0 - \tan^{-1}(2\cot \frac{\theta}{2})\right)\\
& = & 6\alpha_0 - 2\pi +4\tan^{-1}\frac{\sin\theta}{3 + \cos\theta}+2\tan^{-1}\frac{\sin(\hat{\theta})}{3+\cos(\bar{\theta})}\\
\tilde{\rho}\left(St_5^o (\theta)\right) & = & \frac{2\alpha_0 - \frac{2}{3}\pi + 2w(\sigma_\theta)\rho\left(\tau(\sigma_\theta,2)\right)+w(\sigma_{\bar{\theta}})\rho\left(\tau(\sigma_{\bar{\theta}},2)\right)}{\frac{\sqrt{2}}{3}+2w(\sigma_\theta) + w(\sigma_{\bar{\theta}})}
\end{array}
\end{equation}

\noindent where $w(\sigma_\theta), \>\rho\left(\tau(\sigma_\theta,2)\right)$ (resp.  $w(\sigma_{\bar{\theta}}), \> \rho\left(\tau(\sigma_{\bar{\theta}},2)\right)$) are given by (\ref{eq:sigisos}) with explicit functions of $\theta$ (resp. $\bar{\theta}$). 
\end{example}

\begin{example}
Both $|St_5(\alpha_0,\theta)|$ and $\tilde{\rho}\left(St_5(\alpha_0, \theta)\right)$ are given by the following explicit formulas as functions of $\alpha_0 \leq \theta \leq \slfrac{\pi}{2}$, namely

\begin{equation} \label{eq:5.5}
\begin{array}{lcl}
|St_5(\alpha_0,\theta)| & = & |\sigma_{\alpha_0}| + |\tilde{\sigma}_{\alpha_0}| + |\sigma_{\theta}| + 2|\tilde{\sigma}_\theta|\\
& = & \alpha + 2\theta - |\sigma_\theta| = \alpha_0 + 2\theta - 2\tan^{-1}\frac{\sin \theta}{3+\cos\theta}\\
\tilde{\rho}\left(St_5(\alpha_0, \theta)\right) & = & \frac{\frac{1}{3}\alpha_0 + w(\sigma_\theta)\rho\left(\tau(\sigma_\theta,2)\right) + 2w(\tilde{\sigma}_\theta)\rho\left(\tau(\tilde{\sigma}_\theta,2)\right)}{\frac{\sqrt{2}}{6}+\frac{\sqrt{8}}{9}+w(\sigma_\theta) + 2w(\tilde{\sigma}_\theta)}
\end{array}
\end{equation}

\noindent where $w(\sigma_\theta), \> \rho\left(\tau(\sigma_\theta,2)\right)$ (resp. $w(\tilde{\sigma}_\theta), \> \rho\left(\tau (\tilde{\sigma}_\theta,2)\right)$) are given by (\ref{eq:sigisos}) (resp. (\ref{eq:sigtisos})).

\end{example}

\begin{lemma}
  Let $St(\cdot)$ be a 5$\triangle$-star with edge lengths of at least $\slfrac{\pi}{3}$, containing their circumcenters and $|St(\cdot)| \le \pi$ (resp. $> \pi$ and at most equal to $\pi + 0.21672$).
  Set $St_5^{\circ}(\theta)$ (resp. $St_5(\alpha_0,\theta)$) to be the one as indicated in Figure~\ref{fig:5star2}-(ii) (resp. (i)) with the same area of $St(\cdot)$.
  Then
  
  \begin{equation} \label{eq:rtless}
    \rhostar \le \tilde{\rho}(St_5^{\circ}(\theta)) \quad \textrm{(resp. } \tilde{\rho}(St_5(\alpha_0,\theta)) \textrm{)}
  \end{equation}
  and equality holds when and only when $St(\cdot)$ consists of the same collection of quintuple triangles as that of $St_5^{\circ}(\theta)$ (resp. $St_5(\alpha_0,\theta)$).
  \label{lem:fivestar}
\end{lemma}

\begin{remarks}
\begin{enumerate}[(i)]
\item It is not difficult to check that $\{St_5(\alpha_0, \theta)\}$ are realizable as stars of Type-I icosahedra only for $\alpha_0 \leq \theta \leq \gamma$.
\item Lemma~\ref{lem:fivestar} proves that those realizable ones of Examples 5.1 and 5.2 are indeed extremal 5$\triangle$-stars.
\item Thus, straightforward parametric graphing based upon the above set of formulas with $\theta$ as the parameter produces the graphs of $\tilde{\rho}\left(St_5^o (\theta)\right)$ (resp. $\tilde{\rho}\left(St_5(\alpha_0,\theta)\right)$ as functions of $|St_5^o(\theta)|$ (resp. $|St_5(\alpha_0,\theta)|$), as indicated by the two branches of graphs in Figure~\ref{fig:rhostar}, matching at the distinguish peak point of ($\pi, \> \slfrac{\pi}{\sqrt{18}}$). One of the main topic of this section will be the proof of this important lemma.
\end{enumerate}
\end{remarks}

\begin{figure}
  \begin{center}
    \includegraphics[width=3.5in]{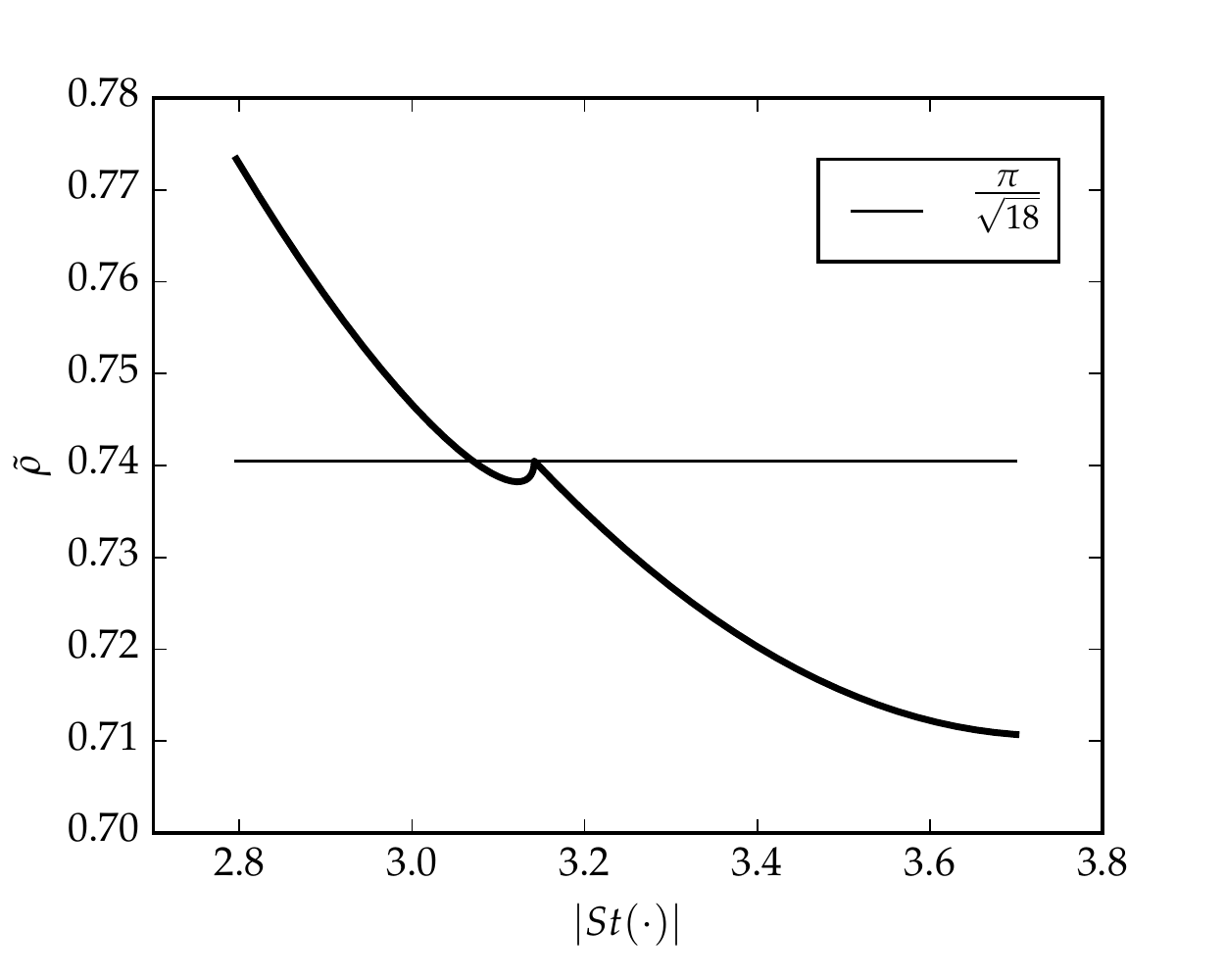}
    \caption{The graph of the area-wise collective estimate of $\tilde{\rho}(St(\cdot))$ as a function of $|St(\cdot)|$.}
    \label{fig:rhostar}
  \end{center}
\end{figure}

\subsubsection[Hi5]{Generalities on the geometry of such 5$\triangle$-stars and some basic geometric ideas}
\label{subsubsec:generalities}

Let $St(\cdot)$ be such a 5$\triangle$-star, $\{r_i\}$ (resp. $\{\theta_i\}, \{b_i\}$), $1\leq i \leq5$, be its quintuples of radial edges (resp. central angles, boundary edges) in cyclic orders. Thus $\{r_i, r_{i+1}, \theta_i\}$ already constitutes the S.A.S. congruence invariants of $\sigma_i$, and hence $\{r_i\}$ together with $\{\theta_i\}$ already constitute a complete set of congruence invariants for such a 5$\triangle$-star, while others such as $\{b_i, |\sigma_i|, u_i, D_i$ and $R_i$ etc. $\}$ can all be computed in terms of $\{r_i\}$ (resp. $\{b_i\}$ and $\{|\sigma_i|\}$) in excess of $\slfrac{\pi}{3}$ (resp. $\slfrac{\pi}{3}$ and $\triangle_{\slfrac{\pi}{3}}$), the set of \textit{radial edge-excesses} (resp. \textit{boundary edge-excesses and area excesses}) of such a $St(\cdot)$, while the correlations between the above triple sets of excesses as well as their manifestations both on densities, weights and their weighted average $\rhostar$ will be the crucial geometric understandings for the proof of Lemma~\ref{lem:fivestar}.

Roughly and intuitively speaking, one naturally expects that the \textit{collective area-wise optimality} of $\rhostar$ should be the result of suitable combination of \textit{individual area-wise optima} (i.e. $\sigma_\theta$ and $\tilde{\sigma}_\theta$). Set $f_1(|\sigma|)$ to be the function of area which records the maximal density for each given $|\sigma|$. As indicated in Figure~\ref{fig:rhosigma}, its graph has a prominent peak at $\frac{1}{2}\Box_{\slfrac{\pi}{3}}$ with the value of $\frac{1}{\sqrt{6}}(3\pi -6\alpha_0)$, which subdivides the function into two branches of convex functions, namely, that of $\sigma_\theta$ and $\tilde{\sigma}_\theta$.

Next let us analyze some of the outstanding geometric features of $St_5^o(\theta)$ (resp. $St_5(\alpha_0,\theta)$), henceforth referred as \textit{extremal stars}. First of all, each of them is an assemblage of those \textit{extremal triangles} which makes the specific weighted average of $f_1(\cdot)$ with respect to its area distribution as high as possible. For example, in the special but also the most critical case of $|St(\cdot)| = \pi$ (i.e. equal to the average value of twelve stars), the area distribution of $St_5^o(\pi - \alpha_0) \simeq St_5(\alpha_0, \alpha_0)$ consists of a pair of $\triangle_{\slfrac{\pi}{3}}$ and a triple of $\frac{1}{2}\Box_{\slfrac{\pi}{3}}$ with $\slfrac{\pi}{\sqrt{18}}$ as their $\mu\{s_i\}$ (cf. \S \ref{subsec:5tstar}).  Anyhow, this makes the proof of Lemma~\ref{lem:fivestar} for the special case of $|St(\cdot)|=\pi$ rather simple. 

Moreover, it is not difficult to see that $St_5^o(\theta)$ (resp. $St_5(\alpha_0, \theta)$) have the most lopsided distributions both in their radial edge-excesses and their boundary edge-excesses. For example, both of them are concentrate in one for $St_5^o(\theta)$.

\medskip

\noindent {\bf Some simple area preserving deformations of 5$\triangle$-stars}

We mention here two kinds of simple area preserving deformations of star configurations, which will be helpful for area-wise estimation of geometric invariants of stars in a way quite similar to that of single triangle in area-wise estimates of individual spherical triangle. 

\begin{enumerate}[(i)]
\item For certain star configurations, the center-vertex still has some kind of degree of freedom to shift. Then, the center shifting will, of course, leave the total area of such a star unchanged, while it is quite straightforward to compute the gradient vector of given geometric invariants such as $\tilde{\rho}(\cdot)$ and $\tilde{w}(\cdot)$ of this kind of area preserving deformations.

\item \textit{Lexell's deformation and a simple kind of area preserving deformation of star configurations}

As indicated in Figure~\ref{fig:lexelldeform}, let $St(\cdot)$ be such a 5$\triangle$-star with $b_1, \> b_2$ and $r_2$ longer than $\slfrac{\pi}{3}$. Then one may use the Lexell's deformation of $\sigma(A_1 A_2 A_3)$ fixing $\overline{A_1A_3}$ to deform it to another 5$\triangle$-star preserving the total area, up until one of $\{b_1^{\prime}, b_2^\prime \}$ becomes $\slfrac{\pi}{3}$. 
\end{enumerate}

Using the parameter representation of \S \ref{lexell} for Lexell deformations, it is quite straightforward to apply Taylor's approximation up to second order for $\tilde{\rho}\left(St(\cdot)\right)$ or other kind geometric invariants of star configurations to determine their minima (resp. maxima) for the intended range. Anyhow, such simple area preserving deformations often provide useful reduction for area-wise estimates of geometric invariants of stars such as the proof of Lemma~\ref{lem:fivestar}. 

\begin{figure}[H]
\begin{center}
    \includegraphics[width=2in]{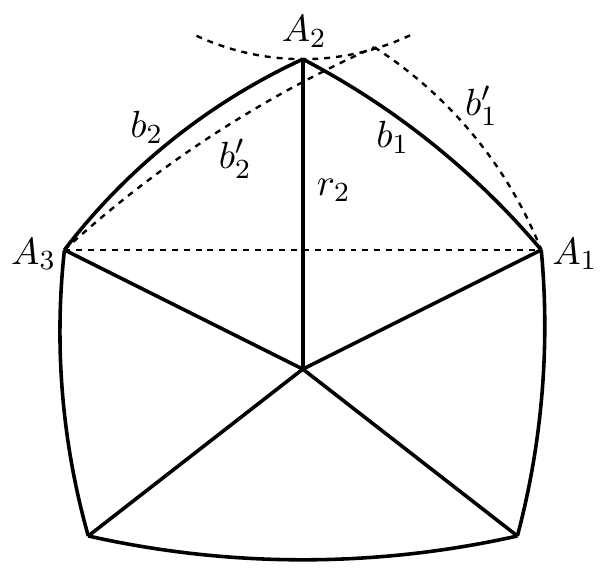}
\caption{}
\label{fig:lexelldeform}
\end{center}
\end{figure}

\subsubsection{A kind of measurement on the lopsidedness of area-distributions of a 5$\triangle$-star} \label{subsec:5tstar}

Just for the purpose of facilitating the proof of Lemma~\ref{lem:fivestar}, we shall introduce the following \textit{measurement} and \textit{terminology}:

Set $\hat{w}(\cdot)$ to be the function of area which records the weight of those extremal triangles as a function of their areas, namely

\begin{equation} \label{eq:omegahead}
\hat{w} = \left\{
\begin{array}{l}
w(\sigma_\theta), \> \> |\sigma_\theta| = s \leq \frac{1}{2} \Box_{\slfrac{\pi}{3}}\\
w(\tilde{\sigma}_\theta), \> \> |\tilde{\sigma}_\theta| = s > \frac{1}{2} \Box_{\slfrac{\pi}{3}}
\end{array}
\right.
\end{equation}

Let $s_i = |\sigma_i|, \> 1 \leq i \leq 5$, be the area distribution of a given $St(\cdot)$. Set

\begin{equation} \label{eq:5mu}
\mu\{s_i\} := \sum_{i=1}^{5} \hat{w}(s_i)f_1(s_i)\bigg/ \sum_{i=1}^{5}\hat{w}(s_i)
\end{equation}

\noindent which will be regarded as the \textit{quantitative measurement} of the lopsidedness of $\{s_i\}$.

\begin{remarks}
\begin{enumerate}[(i)]
\item in case that the quintuple of $\{\sigma_i\}$ of $St(\cdot)$ are all of $\sigma_\theta$ or $\tilde{\sigma}_\theta$ type, $\mu\{s_i\} = \tilde{\rho}\left(St(\cdot)\right)$.

\item Note that $\sigma_\theta$ (resp. $\tilde{\sigma}_\theta$) not only have the maximal density (i.e. $\rho\left(\tau(\sigma,2)\right)$) among such triangles with the same area, but they also have the maximal $w(\sigma)$.

\item The proof of Lemma~\ref{lem:fivestar} will be divided in two cases according to $\mu\left(St(\cdot)\right)$ is at most equal to that of $St_5^o(\theta)$ (resp. $St_5(\alpha_0, \theta)$) or otherwise, namely

\hspace{1cm} \textbf{Case A:} $\mu\left( \stard \right) \leq \tilde{\rho}\left( St_5^o (\theta)\right)$ (resp. $\tilde{\rho}\left(St_5(\alpha_0, \theta)\right)$)

\hspace{1cm} \textbf{Case B:} $\mu\left( \stard \right) > \tilde{\rho}\left( St_5^o (\theta)\right)$ (resp. $\tilde{\rho}\left(St_5(\alpha_0, \theta)\right)$)

\end{enumerate}
\end{remarks}

\noindent
\textbf{On the Correlation between} $\mu\left( \stard \right)$ \textbf{and} $\tilde{\rho}(\stard)$

\noindent Let $\stard$ be such a given 5$\triangle$ -star with $\{\sigma_i, 1 \leq i \leq 5\}$ as its quintuple of triangles, $\{s_i = |\sigma_i|\}$ be its area distribution and $St_5^o(\theta)$ (resp. $St_5(\alpha_0, \theta)$) be the extremal 5$\triangle$-star with the same area of $\stard$, namely $|St_5^o(\theta)|$ (resp. $|St_5(\alpha_0, \theta)|$) equal to $|\stard| \leq \pi$ (resp. $> \pi$). Set $\mu\left(\stard\right) = \mu\{s_i\}$ and regard it as a kind of measurement of the lopsidedness of the area distribution $\{|\sigma_i|\}$ of $\stard$. Note that $\mu\left(St_5^o(\theta)\right)$ (resp. $\mu\left(St_5(\alpha_0, \theta)\right)$) is exactly equal to $\tilde{\rho}\left(St_5^o(\theta)\right)$ (resp. $\tilde{\rho}\left(St_5(\alpha_0, \theta)\right)$), and 

\begin{equation}\label{eq:leqf1}
\rho\left(\tau(\sigma_i, 2)\right) \leq f_1(s_i), \quad w(\sigma_i) \leq \hat{w}(s_i)
\end{equation} 

\noindent while equalities hold when and only when $\sigma_i$ is extremal, i.e. either $\sigma_\theta$ or $\tilde{\sigma}_\theta$  for a certain $\theta$ depending on whether $s_i$ is at most equal to $\frac{1}{2}\Box_{\slfrac{\pi}{3}}$ or exceeding $\frac{1}{2}\Box_{\slfrac{\pi}{3}}$.

One has the following remarkable correlation between $\mu\left( \stard \right)$ and $\tilde{\rho}\left( \stard \right)$, namely

\begin{equation}\label{eq:remarkcorr}
\left(\mu\left(\stard\right) - \tilde{\rho}\left(\stard\right)\right) \sum_{i=1}^{5}w(\sigma_i) = \sum_{i=1}^{5}\{w(\sigma_i)\left(f_1(s_i)-\rho\left(\tau(\sigma_i,2)\right)\right) + \left(f_1(s_i)-\mu\right) \left(\hat{w}(s_i)-w(\sigma_i)\right)\}
\end{equation}

\begin{proof}
By (\ref{eq:omegahead}) and (\ref{eq:5mu}), one has 

\begin{flalign*}
\mu \sumf w(\sigma_i) & = \mu \sumf \hat{w}(s_i) - \mu \sumf (\hat{w}(s_i)-w(\sigma_i))&&
\end{flalign*}

\vspace{-1cm}

\begin{flalign}\label{eq:proofp2}
&= \sumf \hat{w}(s_i)f_1(s_i) - \mu \sumf (\hat{w}(s_i)-w(\sigma_i))&&
\end{flalign}

\vspace{-1cm}

\begin{flalign*}
&= \sumf w(\sigma_i)\rho(\tau(\sigma_i,2))+\sumf\{w(\sigma_i)(f_1(s_i)-\rho(\tau(\sigma_i,2)))+(f_1(s_i)-\mu)(\hat{w}(s_i)-w(\sigma_i))\}&&
\end{flalign*}

\vspace{-1cm}

\begin{flalign}\label{eq:proofp4}
&=\tilde{\rho}(\stard)\sumf w(\sigma_i)+ \hbox{RHS  of} \> (\ref{eq:remarkcorr}) && 
\end{flalign}

\noindent and hence, (\ref{eq:remarkcorr}) follows from (\ref{eq:proofp4}). 
\end{proof}

\noindent \textbf{Sublemma:} \textit{Each summand of the RHS of (\ref{eq:remarkcorr}) is always non-negative and it is equal to zero when and only when $\sigma_i$ is extremal.} 

\begin{proof}
By (\ref{eq:omegahead}), if $\sigma_i$ is extremal, then the $i$-th summand of the RHS of (\ref{eq:remarkcorr}) is equal to zero because both of its two terms are equal to zero. On the other hand, if $\sigma_i$ is non-extremal, then both ($f_1(s_i)-\rho(\tau(\sigma_i,2))$) and ($\hat{w}(s_i)-w(\sigma_i)$) are positive. Thus, the $i$-th summand will be obviously positive in the case that ($f_1(s_i)-\mu$) is non-negative. On the other hand, in the case that ($f_1(s_i)-\mu$) is negative, $w(\sigma_i)$ is always many times larger than $|f_1(s_i)-\mu|$. Hence, it is not difficult to show that the $i$-th summand is, again, positive. 
\end{proof}

\begin{corollary}
\textit{If $\mu(\stard)$ is at most equal to that of the extremal star with the same area, then $\tilde{\rho}(\stard)$ is at most equal to that of the extremal star, and it is equal to when and only when $\stard$ consists of the same collection of triangles as that the extremal star (cf. Examples \ref{eq:5.1} and \ref{eq:5.2}).}
\end{corollary}

\begin{remarks}
\begin{enumerate}[(i)]
\item Suppose that $\mu(\stard)$ is larger than that of the extremal star with the same area. Then $\stard$ must contain at least one non-extremal triangle, because it is quite simple to check that the assemblages of extremal triangles of those extremal stars are the \textit{only} such assemblages of extremal triangles with the \textit{highest} $\tilde{\rho}(\stard)$ for the same area.

\item The above corollary already reduces the proof of Lemma~\ref{lem:fivestar} to the case that $\mu(\stard)$ exceeds that of the extremal star with the same area, which amounts to provide a lower bound estimate of the RHS of (\ref{eq:remarkcorr}) exceeding
\begin{equation*}
(\mu(\stard)-\tilde{\rho}(\hbox{extremal  star with equal area}))\sumf w(\sigma_i))
\end{equation*}
\end{enumerate}
\end{remarks}

\subsubsection{The proof of Lemma \ref{lem:fivestar}}

\begin{proof}
In this subsection, we shall proceed to provide a proof of Lemma~\ref{lem:fivestar}, based upon the $(k,\delta)$-analysis and the optimal estimates of $\rho(\sigma)$(resp. $w(\sigma)$) of \S \ref{subsec:areadeform}, together with the geometric insights of \S \ref{subsubsec:generalities} and \S \ref{subsec:5tstar}. For example, the sublemma of \S \ref{subsec:5tstar} shows that $\mu\{s_i\}$ provides a simple and advantageous upper bound of $\tilde{\rho}(\stard)$, which will be optimal when and only when the quintuple $\{\sigma_i\}$ of $\stard$ are all extremal ones. In particular, such a simple upper bound solely in terms of area distribution $\{s_i\}$ already proves Lemma~\ref{lem:fivestar} for the case $\mu\{s_i\}$ are at most equal to that of the \textit{extremal star} with the same total area $A=|\stard|$. Just for this subsection, we shall denote an extremal star by $St^*(A)$ and set 
\begin{equation}
\hat{\rho}(A)=\mu(St^*(A))=\tilde{\rho}(St^*(A)), \> A=|St^*(A)|
\end{equation}

Note that, for a given value of the total area $A$, the area distribution $\{s_i^\ast\}$ of the extremal star of total area $A$ is already rather lopsided (i.e. with rather high $\mu\{s_i^\ast\}$ among that of stars with the same $A$). Here, the very particular shapes of the graphs of $f_1(s)$ (resp. $\hat{w}(s)$) will, of course, play an important role in the proof of Lemma~\ref{lem:fivestar}, especially, the peak of $\Gamma(f_1)$ in the vicinity of $s_o = 2\sin^{-1}\frac{1}{3}$. Mainly for this reason we shall subdivide the proof of Lemma~\ref{lem:fivestar} into the following cases, namely

\textbf{Case 1:} $A \geq \pi - 0.04$ and \textbf{Case 2:} $A < \pi - 0.04$

\noindent
\textbf{The Proof of Case 1:} ($\pi - 0.04 \leq A \leq \pi + 0.21672)$)

First of all, in the range of $A$ lying between ($\pi - 0.04$) and ($\pi + 0.21672$), the area distribution of $St^*(A)$ have a pair of $\{s_i\}$ very close to the minimum together with a triple of $\{s_i\}$ rather close to $s_o = 2\sin^{-1}\frac{1}{3}$. Due to the special features of $f_1(s)$ and $\hat{w}(s_i)$, it is not difficult to see that $\mu(St^*(A))$ is almost as large as possible for $\mu(\stard)$ with the same $A$, and moreover, any other $\stard$ of total area $A$ with $\mu(\stard)$ at least equal to $\hat{\rho}(A)$ (or even slightly smaller than $\hat{\rho}(A)$) must be consisting of such a pair and a triple of $\{s_i\}$, such as $\{\stfive, \> \alpha_0 < \theta_1, \theta_2 < \gamma\}$. Furthermore, the $(k,\delta)$-analysis of $\rho(\sigma)$ shows that those $\sigma$ with $\rho(\sigma)$ only slightly smaller than $f_1(|\sigma|)$ can only be a quite small deformation of the extremal triangle with the same area.

Now, suppose that $\stard$ is such a 5$\triangle$-star of Type-I icosahedron with $\tilde{\rho}(\stard) \geq \hat{\rho}(|\stard|)$. Then, it follows from (\ref{eq:remarkcorr}) that the triple of $\{\sigma_i\}$ with $s_i=|\sigma_i|$ in the vicinity of $s_o$ must be all just small deformations of extremal triangles of the same area. Therefore $\stard$ must be a small deformation of $St_5(\theta_1,\theta_2)$, or the other way of assemblage in the case of $A\leq \pi$, and hence $\stard$ must be, in fact, equal to $St^*(A)$ (i.e. $St(\alpha_0, \theta)$ of total area $A$).

This proves Case 1.

\noindent
\textbf{The Proof of Case 2:} ($A < \pi -0.04$)

For 5$\triangle$-stars with $A=|\stard|$ less than $(\pi - 0.04)$, the extremal stars $St^*(A)$ are $St_5^o(\theta)$ or its other way of assemblage; their radial (resp. boundary) edge-excesses are concentrated in a single radial (resp. boundary) edge, while its quintuple of triangles consists of a pair of $\sigma_{\alpha_0}$, a pair of $\sigma_\theta$ and another $\sigma_{\tilde{\theta}}$, 
\begin{equation}
\alpha_0 \leq \theta <\arccos(-\frac{1}{3}) - \varepsilon_o,\quad  \tilde{\theta} = 2\left(\pi-\alpha_0-\tan^{-1}\left( 2\cot\frac{\theta}{2}\right)\right)
\end{equation}

Note that $\tilde{\theta}=\gamma$ as $ \theta = \alpha_0$ and are always larger than $\theta$, while their difference ($\tilde{\theta}-\theta$) is steadily decreasing from ($\gamma-\alpha_0$) and becoming equal to zero at $\theta=\arccos(-\frac{1}{3})$. Anyhow, let us start with the proof of the seemingly rather special but actually more critical subcase that the radial edge excesses are concentrated in a single radial edge.

\noindent (1) \textit{The proof of the subcase 1 of only one radial edge excesses, i.e. $\stard$ with quadruple $\slfrac{\pi}{3}$-radial edges.}

Set $r_5$ to be the radial edge of such a $\stard$ which may be longer than $\slfrac{\pi}{3}$. First of all, the very special beginning case of $r_5=\slfrac{\pi}{3}$ follows easily from the convexity of $\rho(\sigma_\theta)$, $\alpha \leq \theta \leq \gamma$, as a function of $|\sigma_\theta|$. Therefore, we shall from now on assume that $r_5 > \slfrac{\pi}{3}$ and $A=|\stard|>A_o=|St_5^o(\alpha_0)|$. Let $St_5^o(\theta)$ be the  extremal star $St^*(A)$ with total area $A$. Then it is easy to show that 
\begin{equation*}
\theta_1 + \theta_2 + \theta_3 \leq 2\alpha_0 + \tilde{\theta}
\end{equation*}
and equality holds when and only when $\{\theta_1, \theta_2, \theta_3\}$ contains two $\alpha_0$, namely, $\stard$ is itself extremal. Again, by the convexity of $\rho(\sigma_\theta)$ as a function of $s=|\sigma_\theta|$, it is easy to see that , for a given value of $\sum_{i=1}^3 \theta_i$ less than $(2\alpha_0+\tilde{\theta})$ , the more critical geometrical situations are those $\stard$ with $\{\theta_i, 1 \leq i \leq 3\}$ containing a pair of $\alpha_0$; and moreover, by the application of the Lexell's deformations of such stars, starting at the one with $b_4=b_5$ up until that one of $\{b_4,b_5\}$ become $\slfrac{\pi}{3}$, it is straightforward to check that such a Lexell's deformation is density increasing, thus reducing to the study of the following specific family of stars, namely, with quadruple of extremal triangles including a pair of $\sigma_{\alpha_0}$ and one $\slfrac{\pi}{3}$-isosceles with the only larger radial edge as its base, say $\sigma_4$ as indicated in Figure~\ref{fig:5stardeform}.

\begin{figure}
\begin{center}
    \includegraphics[width=2in]{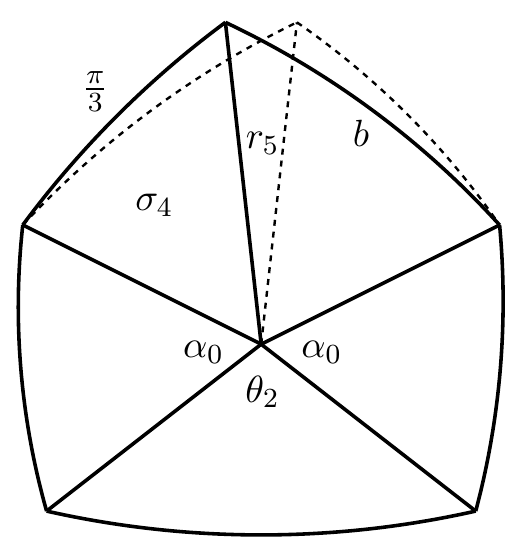}
\caption{}
\label{fig:5stardeform}
\end{center}
\end{figure}

Geometrically, such a $\stard$ is a specific kind of area preserving deformation of $St_5^o(\theta)$ with $\sigma_5$ to be the only non-extremal which makes $b_5$ longer and $\theta_2$ smaller than $\tilde{\theta}$. These kind of stars are completely determined by $r_5 $ and $\theta_2$ whose area and $\tilde{\rho}(\cdot)$ can be easily computed by explicit formula in terms of $r_5$ and $\theta_2$, say denoted by $A(r^\prime,\theta^\prime)$ and $\tilde{\rho}(r^\prime, \theta^\prime)$. Therefore, straightforward computation by the methods of implicit differentiation will show that the initial value of $\tilde{\rho}$, i.e. $\hat{\rho}(A)$, is the maximal value of $\tilde{\rho}(\cdot)$ along the above kind of area preserving deformations. This proves Lemma~\ref{lem:fivestar} for such a subcase. 

\noindent (2) \textit{The proof of the subcase 2 of two non-zero radial edge-excesses:}

Let $\stard$ be such a star with $\mu\{\cdot\}$ exceeding $\hat{\rho}(|\stard|)$. Then, one of the non-zero radial excesses must be very small. For otherwise, it is easy to show that $\mu\{\cdot\}$ is at most equal to $\hat{\rho}(A)$. Therefore, such a star can be regarded as a very small deformation of that of subcase 1, and moreover, it is straightforward to check that such small deformations are density decreasing, thus proving $\tilde{\rho}(\stard)$ less than $\hat{\rho}(A)$. 

\noindent (3) \textit{The proof of remaining cases:}

We may assume without loss of generality of the proof that such a $\stard$ has a triple of non-zero radial edge-excesses. For otherwise, we may simply replace it by a density increasing deformation of shifting its center. And moreover, the smallest non-zero radial edge-excesses must be very small. For otherwise, $\mu(\cdot)$ will be at most $\hat{\rho}(A)$. Therefore, such a star is, again, just a very small density decreasing deformation of a star of subcase 2, and hence $\tilde{\rho}(\cdot)<\hat{\rho}(A)$.

This completes the proof of Lemma~\ref{lem:fivestar}.
\end{proof}

\subsubsection{Extension of Lemmas~\ref{lem:fivestarchar} and \ref{lem:fivestar} to $5\triangle$-stars of areas exceeding $(\pi+0.21672)$: Lemma~\ref{lem:fivestar}$'$}

\begin{itemize}
\item[(1)] Note that both Lemma~\ref{lem:fivestarchar} and Lemma~\ref{lem:fivestar} essentially only cover the case of $5\triangle$-stars in Type-I icosahedra with areas at most equal to $(\pi+0.21672)$. In this subsection, we shall proceed to study their extensions beyond that upper limit. The main result here will be stated as Lemma~\ref{lem:fivestar}$'$. Recall that, for the range of areas lying between $\pi$ and $(\pi+0.21672)$, Lemma~\ref{lem:fivestar} proves that the {\it optimal stars} (i.e. the ones with optimal $\tilde{\rho}(\cdot)$ for such a given area $|St(\cdot)|$) are uniquely realized by $St_5(\alpha_0,\beta)$, $\alpha_0\leq \beta\leq \gamma$; while Lemma~\ref{lem:fivestarchar} proves that those Type-I configurations containing small deformations of such extremal stars are necessarily of $5\boxslash$-type, namely, small deformations of $5\square$-Type-I. In particular, the $5\square$-Type-I with angular distribution of quadruple $\alpha_0$ and $\gamma$ is, in fact, the {\it unique} one containing $St_5(\alpha_0,\gamma)$. Actually, if one {\it ignores the realizability condition}, it is not difficult to show that $St_5(\alpha_0,\beta)$, $\beta>\gamma$, will also be the unique extremal star for areas considerably larger than $(\pi+0.21672)$. However, the problem here is that they are {\it no longer realizable} as stars of Type-I configurations. Therefore, the geometric problem here is a kind of {\it tight combination} of {\it optimality} and {\it realizability} and what we are seeking should be a kind of joint generalization of Lemmas~\ref{lem:fivestarchar} and \ref{lem:fivestar} for the case of areas larger than $(\pi+0.21672)$.

\item[(2)] Let us begin our study in the case of areas only slightly larger than that of $St_5(\alpha_0,\gamma)$, say up to $(\pi+0.25)$. Intuitively, one expects that such stars (i.e. realizable and with almost optimal $\tilde{\rho}(\cdot)$ for such a given area) should have quite similar geometric structures as that of $St_5(\alpha_0,\gamma)$, namely, with a pair of small almost extremal triangles and a triple of small deformations of half rectangles. Moreover, we may assume without loss of generality that such a star has at least a pair of $\slfrac{\pi}{3}$-radial edges, for otherwise, one may simply reduce to such a star by a {\it density increasing} deformation of {\it center-shifting}. Therefore, such a star should have an adjacent pair of small triangles, each with at least one $\slfrac{\pi}{3}$-radial edge, together with a triple of small deformations of half rectangles, or more precisely, one small deformation of half rectangle and another $\boxslash$, in order to minimizing their $(k,\delta)$-decrements to achieve higher $\tilde{\rho}(\cdot)$ among those realizable ones. Note that the {\it realizability} condition of such a star can be conveniently {\it reformulated} as follows, namely, there exists a star including the pair of small triangles, such as $St(N)$ in Example~\ref{ex:smallstar}, whose {\it complementary region} accommodates another star {\it adjacent} to the one we constructed. Anyhow, this is the basic geometric idea that naturally leads to the construction of Example~\ref{ex:smallstar} as well as the formulation of the following lemma, namely
\end{itemize}

\noindent
{\bf Lemma~\ref{lem:fivestar}$'$}: {\it Let $\sconf$ be a Type-I icosahedron containing a star with $\tilde{\rho}(\cdot)$ at least almost as high as the optimal one of the family of Example~\ref{ex:smallstar} with the same area. Then $\sconf$ must be a small deformation of some of Example~\ref{ex:smallstar} containing a star of the same area with optimal $\tilde{\rho}(\cdot)$.}

\begin{proof}
(1) First of all, we may assume without loss of generality for the purpose of proof, that such a star in $\sconf$ contains at least a pair of $\slfrac{\pi}{3}$-radial edges, simply by a reduction of density increasing center shifting. Thus, we shall first divide the proof of this lemma into two major cases, namely
\medskip

{\bf Case I}: Such a star contains an adjacent pair of $\slfrac{\pi}{3}$-radial edges.
\medskip

{\bf Case II}: Such a star only contains a separate pair of $\slfrac{\pi}{3}$-radial edges.
\medskip

\noindent
We shall denote the pair of small triangles of such a star by $\sigma_1$ and $\sigma_2$, and in the case of I, $\sigma_2$ is the one with both of its radial edges equal to $\slfrac{\pi}{3}$ (or the smaller one if such a star containing a triple of $\slfrac{\pi}{3}$-radial edges), while the boundary edges of $\{\sigma_1,\sigma_2\}$ will be denoted by $\{b_1,b_2\}$.
\medskip

Moreover, we shall further subdivide Case~I into the following subcases, namely
\medskip

{\bf Case I$_0$}: $\{\sigma_1,\sigma_2\}$ are the same as that of Example~\ref{ex:smallstar}, namely $b_1=b_2=\slfrac{\pi}{3}$ and $\sigma_1=\sigma_\gamma$, $\sigma_2=\sigma_{\alpha_0}$,
\medskip

{\bf Case I$_1$}: $b_1=b_2=\slfrac{\pi}{3}$ but $\sigma_1$ is not $\sigma_\gamma$,
\medskip

{\bf Case I$_2$}: At least one of $\{b_1,b_2\}$ is longer than $\slfrac{\pi}{3}$.
\medskip

\noindent
(2) {\bf The proof of Case I$_0$}:
\medskip

As it turns out, the proof of this very special subcase is the most critical and revealing on the geometric insight of how the areawise optimality and realizability interplay with each other. Suppose that $\sconf$ is such a Type-I icosahedron and $St(A_2)$ is the larger star with $\tilde{\rho}(\cdot)$ at least almost as high as that of the family of Example~\ref{ex:smallstar} with the same area.
\begin{itemize}
\item[(i)] Let us first analyze the geometric situation that the star containing $\{\sigma_1,\sigma_2\}$, say again denoted by $St(N)$, is the same as that of Example~\ref{ex:smallstar}, namely, with uniform $\slfrac{\pi}{3}$-radial edges and the most lopsided angular distribution. Then the opposite star of $St(N)$ in $\sconf$, say again denoted by $St(N')$, will, of course, be a star whose vertices, say again denoted by $\{A_i',\>1\leq i\leq 5\}$, are inside of its complementary region, as indicated in Figure~\ref{fig:smallstar} by the region below those $\slfrac{\pi}{3}$-circular arcs tangent to $y=-\frac{\pi}{6}$ at $\{T_i\}$, which can be  represented via the stereographic projection as indicated in Figure~\ref{fig:largepentagon}.

\begin{figure} [H]
\begin{center}
    \includegraphics[width=3in]{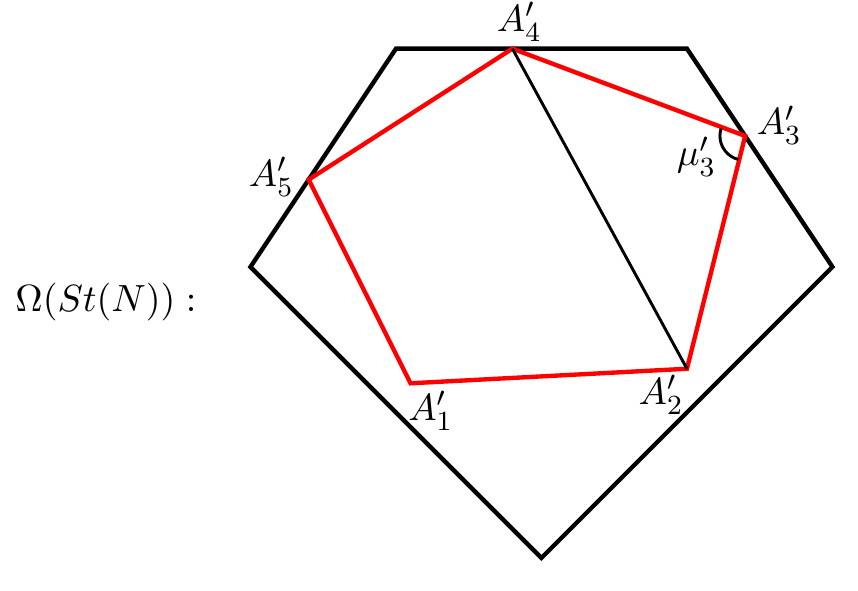}
\caption{\label{fig:largepentagon}}
\end{center}
\end{figure}

Note that the geometry of $St(A_2)$ in $\sconf$ is already determined by the positions of $A_1'$ and $A_2'$, while that of $St(N')$ inside of $\Omega(St(N))$ mainly providing an example of realization, as far as $\tilde{\rho}(\cdot)$ and $|St(\cdot)|$ is concerned. Therefore, one may choose some special kind of $St(N')$ just for the purpose of proving Case~I$_0$. In particular, it is not difficult to show that there exists such a $St(N')$ with $\{A_3',A_4'\}$ situated on $\partial \Omega$ and $\overline{A_2'A_3'}=\overline{A_3'A_4'}=\slfrac{\pi}{3}$. Set $\theta$ to be the parameter determined by $A_4'$ as in Figure~\ref{fig:largepentagon} (cf. Example~\ref{ex:smallstar}). For the purpose of proving this lemma, we shall study the following special kind of extension problem, namely, for a given position of $A_2'$, inside of $\Omega$ in the vicinity of $T_2$, what are those extensions of an $St(N')$ with rather large $|St(A_2)|$ and of areawise higher $\tilde{\rho}(\cdot)$.

Set $\mu_3'$ to be the top angle of the $\slfrac{\pi}{3}$-isosceles $\sigma(A_2'A_3'A_4')$ and $\frac{1}{2}\lambda_2$ to be its base angle, namely
\begin{equation}
\lambda_2=2\arctan(2\cot\frac{\mu_3'}{2})
\end{equation}
If $\lambda_2$ is at most equal to $(f(\theta)-\alpha_0)$, then there exist one of Example~\ref{ex:smallstar} with $A_2'$, $A_3'$, $A_4'$ as its vertices, while such an $St(N')$ must have its $A_1'$ outside of that of the one of Example~\ref{ex:smallstar}, if it is not coincide with its vertex. Moreover, we may assume without loss of generality for the proof that $\overline{A_5'A_1'}=\slfrac{\pi}{3}$. For otherwise, we may perform an density increasing Lexell's deformation to $St(A_2)$ to reduce $St(N')$ to such a $St(N')$, namely, $St(N')$ and one of Example~\ref{ex:smallstar} have identical $\{A_i',\>2\leq i\leq 5\}$, while the only difference is that its $\mu_5'$ is slightly larger than that of the comparing one of Example~\ref{ex:smallstar}. Therefore, it is quite straightforward to check that its $\tilde{\rho}(\cdot)$ is, in fact, smaller than that of $St(A_2)$ in some of Example~\ref{ex:smallstar} with the same area.

\item[(ii)] Next, let us consider the other possibility that $\lambda_2$ exceeds $(f(\theta)-\alpha_0)$. Then the kind of tight extensions of Example~\ref{ex:smallstar} with $\{A_2',A_3',A_4'\}$ as their vertices are no longer available. Set $\Gamma$ to be the $\slfrac{\pi}{3}$-circle passing through $\{A_2',A_4'\}$ and $A_5'$ to be the point on $\Gamma\cap \partial \Omega$ in the vicinity of $T_5$. It is easy to check that $\overline{A_4'A_5'}>\slfrac{\pi}{3}$ and we may assume that $N'$ is actually the center of $\Gamma$. Therefore it is again straightforward to check that $\tilde{\rho}(\cdot)$ of $St(A_2)$ is again smaller than that one of Example~\ref{ex:smallstar} with the same area. This proves the subcase of Case~I$_0$ that $St(N)$ in $\sconf$ has uniform $\slfrac{\pi}{3}$-radial edges.

\item[(iii)] Now, let us consider the remaining subcase of Case~I$_0$ that $St(N)$ contains some radial edges longer than $\slfrac{\pi}{3}$. First of all, such radial edge-excesses, even just a very small total amount will already make the complementary region $\Omega$ considerably smaller than that of the previous subcase, thus making some of those {\it critical} extensions of $St(N')$ with the same $\{A_2',A_1'\}$ no longer available. Therefore, it is not difficult to follow the same kind of geometric analysis to provide the kind of areawise estimates of $\tilde{\rho}(St(A_2))$ as that of the previous subcases.

\end{itemize}
\medskip

\noindent
(3) {\bf The proof of Case~I$_1$}

This case is naturally consisting of two subcases according to the angle of $\sigma_1$ at $N$ is less than $\gamma$ or larger than $\gamma$.

\begin{itemize}
\item[(i)] {\bf Subcase 1}: $\sigma_1=\sigma_\theta$, $\theta<\gamma$:

Let us first consider the geometric analysis for the situation of $\sconf$ that $St(N)$ again has no radial edge-excesses. Then, the same kind geometric construction similar to that of Example~\ref{ex:smallstar}, in particular to the specific case that the central angles of $St(N)$ consisting of $\{\theta_1,\alpha_0,\pi-\alpha_0-\frac{1}{2}\theta_1,\pi-\alpha_0-\frac{1}{2}\theta_1,\alpha_0\}$, say denoted by Example~$\tilde{\hbox{\ref{ex:smallstar}}}$.

Note that one of the critical fact for such an $\sconf$ is that $\overline{A_1A_2}$ in $St(A_2)$ is shorter than that of Case~I$_0$ (i.e. $r_0=2\arcsin(\sin\frac{\gamma}{2})$). Therefore, $\tilde{\rho}(\cdot)$ of $St(A_2)$ in such $\sconf$ of Example~$\tilde{\hbox{\ref{ex:smallstar}}}$ will always be lower than that of Example~\ref{ex:smallstar} with the same area. Anyhow, the same kind of comparison areawise estimate of $\tilde{\rho}(\cdot)$ will always provide a smaller upper bound on $\tilde{\rho}(\cdot)$.

\item[(ii)] {\bf Subcase 2}: $\sigma_1=\sigma_\theta$, $\theta>\gamma$:

In this subcase, the $St(N)$ in such an $\sconf$ is necessarily to have some radial edge-excesses depending on the size of $(\theta-\gamma)$. Therefore, on the one hand, the radial edge $\overline{A_2A_1}$ in $St(A_2)$ is now longer than $r_0=2\arcsin(\sin\frac{\gamma}{2})$, which is advantageous for producing larger $St(\cdot)$, but on the other hand the minimal amount of total radial edge-excesses of $St(N)$ makes the size of its complementary region substantially smaller than that of the minimal star in Case~I$_0$, which is rather disadvantageous for producing larger $St(\cdot)$ with higher density. Anyhow, the same kind of geometric analysis as that of Case~I$_0$ will show that the disadvantage far outweighing the advantage. Therefore, essentially the same kind of comparison upper bound areawise estimate will again show that $\tilde{\rho}(\cdot)$ of $St(A_2)$ in this subcase will always be lower than that of the Case~I$_0$.

\end{itemize}

\noindent
(4) {\bf The proofs of Case~I$_2$ as well as Case~II}

Let us first discuss the geometric situation of Case~I$_2$.  Note that $\{b_1,b_2\}$ are the build in radial edges of $St(N)$. Therefore in this case, $St(N)$ always has some radial edge-excesses. On the other case of Case~II, $St(N)$ has an adjacent pair of $\slfrac{\pi}{3}$-boundary edges. Thus, unless in the very special geometric situation of $\sigma_1=\sigma_2=\sigma_{\alpha_0}$ which has already been included in Case~I$_1$, the $St(N)$ of such an $\sconf$ of Case~II always has both the disadvantage of having radial edge-excesses and $\overline{A_2A_1}=\frac{\pi}{3}=\overline{A_2A_3}$. Anyhow, it is quite clear that the same kind of proof for the previous cases will also provide simple proofs of Case~I$_2$ and Case~II.

This completes the proof of Lemma~\ref{lem:fivestar}$'$.

\end{proof}

\begin{remarks}
\begin{enumerate}[(i)]
\item For the purpose of proving Theorem~I for Type-I local packings, Lemma~\ref{lem:fivestar}$'$ already provides an easy to use upper bound areawise estimate for $\tilde{\rho}(\cdot)$ of $5\triangle$-stars of areas up to $(\pi+0.6)$, namely, that of those $\tilde{\rho}(\cdot)$ of $St(A_2)$ of Examples~\ref{ex:smallstar}.
\item The proof of Lemma~\ref{lem:fivestar}$'$ also shows that, for those realizable $5\triangle$-stars with areas exceeding $(\pi+0.6)$, they can {\it not} be the assemblages of a pair of almost extremal small triangles and a triple of small deformations of half rectangles, because such $5\triangle$-stars with areas exceeding $(\pi+0.6)$ are no longer realizable. Therefore, it follows from the $(k,\delta)$-estimates that $\tilde{\rho}(\cdot)$ of those realizable $5\triangle$-stars with areas larger than $(\pi+0.6)$ must be considerably lower than that of $St(A_2)$ of Examples~\ref{ex:smallstar}. In particular, it is not difficult to show that the following areawise slope of $\tilde{\rho}(\cdot)$, namely
\begin{equation}
\slfrac{(\tilde{\rho}(\cdot)-{\pi}/{\sqrt{18}})}{(\hbox{total area} -\pi)}
\end{equation}
must be smaller than that of $St(A_2)$ in Examples~\ref{ex:smallstar} with maximal area.

\end{enumerate}

\end{remarks}

%% file: 6-proof-of/6-proof-of.tex
\section{The proof of Theorem I for the case of Type-I local packings}
\label{sec:proof}

Geometrically speaking, the moduli space of congruence classes of Type-I configurations, say denoted by ${\cal M}_I$, constitutes a 21-dimensional semi-algebraic set on which the specific geometric invariant $\bar{\rho}:{\cal M}_I \rightarrow \mathcal{R}^+$ is defined to be the following weighted average
\begin{equation}
\bar{\rho}(\sconfp) = \slfrac{\sum_{\sigma_j \in \sconfp} w(\sigma_j) \rho(\sigma_j)}{\sum_{\sigma_j \in \sconfp} w(\sigma_j)}
\end{equation}

First of all, the pair of 6$\Box$-ones (i.e. the f.c.c. and the h.c.p.) and the family of 5$\Box$-ones constitute two particular kinds of outstanding examples, namely, a pair of singular points and a 4-dimensional singular subvariety with $\rhobarred$ of the f.c.c. and the h.c.p. equal to $\spet$ while that of the latter are slightly lower than $\spet$. Of course, one of the critical, important steps in the proof of Theorem I for the Type-I case should, naturally, be the establishment of suitable neighborhoods of the 6$\Box$-ones (resp. the 5$\Box$-ones) in ${\cal M}_I$ that they are indeed the unique maxima (resp. a kind of ridge subvariety) in such neighborhoods, which consists of those $6\boxslash$-Type-I's (resp. $5\boxslash$-Type-I's).

Thus, we shall divide the proof of Theorem~I for the case of Type-I local packings into the following three subcases, namely
\begin{labeling}{cases}
\item[\textbf{Case~I$_1$:}] $6\boxslash$-Type-I's.
\item[\textbf{Case~I$_2$:}] $5\boxslash$-Type-I's.
\item[\textbf{Case~I$_3$:}] Others
\end{labeling}
while the proofs of Case~I$_1$ (resp. Case~I$_2$) have already been given in \S\ref{subsec:proofthrm1} (resp. \S\ref{subsec:5estimates}) as a kind of demonstrations of the advantages of collective areawise estimates for $\boxslash$-clusters and lune-clusters. Therefore it suffices to prove the remaining case of Case~I$_3$.

\medskip

\noindent
{\bf The proof of Case~I$_3$}:

Note that all non-icosahedra Type-I's are belonging to Cases~I$_1$ or I$_2$ and hence Case~I$_3$ only consists of Type-I icosahedra. Therefore, one has
\begin{equation}
\overline{\rho}({\cal S}'(\Sigma))=\sum_{A_i\in\Sigma}\tilde{w}(St(A_i))\cdot\tilde{\rho}(St(A_i))\Big/ \sum_{A_i\in \Sigma}\tilde{w}(St(A_i))
\end{equation}
and hence, it suffices to prove that
\begin{equation} \label{eqn:densitysm0}
\sum_{A_i\in\Sigma}\tilde{w}(St(A_i))(\tilde{\rho}(St(A_i))-\pi/\sqrt{18})<0
\end{equation}
by a direct application of Lemmas~\ref{lem:fivestar} and \ref{lem:fivestar}$'$.

For a given Type-I $5\triangle$-star $St(\cdot)$ with $|St(\cdot)|\neq \pi$), set
$$ \tilde{m}(St(\cdot))=\tilde{w}(St(\cdot))\cdot(\tilde{\rho}(St(\cdot))-\pi/\sqrt{18})\Big/(\pi-|St(\cdot)|)
$$
which will be referred to as the {\it weighted slope} of $St(\cdot)$. It follows from Lemma~\ref{lem:fivestar}, Lemma~\ref{lem:fivestar}$'$ and the remarks following Lemma~\ref{lem:fivestar}$'$, one has the following upper bound of $\tilde{m}(\cdot)$, namely

(i) for $|St(\cdot)|<\pi$,
\begin{equation}
\tilde{m}(St(\cdot))\leq \tilde{m}(\hbox{the smallest $5\triangle$-star}):=\hat{m}
\end{equation}

(ii) for $|St(\cdot)|>\pi$,
\begin{equation}
\tilde{m}(St(\cdot))\leq \tilde{m}(St_5(\alpha_0,\gamma))=\check{m}
\end{equation}
and moreover, $\hat{m}<(-\check{m})$. 

Applying the above simple estimates to the LHS of ($\ref{eqn:densitysm0}$), one has the following estimates of its positive terms (resp. its negative terms) as follows, namely
\begin{itemize}
\item[(i)] The sum of positive terms is bounded above by
$$\begin{array}{l}
 \hat{m}\cdot \left\{ \hbox{the sum of $(\pi-|St(\cdot)|)$ of those with $\tilde{\rho}(\cdot)>\pi/\sqrt{18}$}\right\}\\
 \leq \hat{m}\cdot\left\{\hbox{the total area defects of those stars of ${\cal S}'(\Sigma)$}\right\}
\end{array}$$
\item[(ii)] The sum of negative terms is bounded above by
$$\check{m}\cdot \left\{ \hbox{the sum of area excesses of those stars of  ${\cal S}'(\Sigma)$ of areas larger than $\pi$}\right\}$$
\end{itemize}
On the other hand, it is quite obvious that the total area defects is equal to the total area excesses because the average of areas of the twelve stars is always equal to $\pi$. Therefore, the proof of Case~I$_3$ follows simply from the above straightforward estimate, thanks to Lemma~\ref{lem:fivestar} and Lemma~\ref{lem:fivestar}$'$.
\medskip

This completes the proof of Theorem~I for the case of Type-I local packings.
\hfill $\square$

%% file: 7-proof-case-2/7-proof-case-2.tex
\section{The proof of Theorem I for Case II: non-Type-I local packings}

In this section, we shall proceed to complete the proof of Theorem I for those local packings other than those Type-I local packings, which will be simply referred to as Case II.

\subsection{Some pertinent generalities and a brief overview}

First of all, non-Type-I local packings consist of all those local packings which have at most eleven touching neighbors, thus encompassing great amount of probabilities and varieties. However, what we're are going to prove will be that their locally averaged density, $\rhobarred$, are always less than $\spet$! Of course, such a proof should be quite simple for those such as local packings with neighbors all of them having certain amount of buckling heights. Therefore, the real challenge in the proof of Case II is not just to have a proof, but rather, to achieve a clean-cut simple proof. Anyhow, let us begin here with some pertinent generalities.

\subsubsection{Some generalities on non-Type-I local packings}

(1) {\bf A concept of close neighbors and reduced local packings}

Technically, it is quite clear that almost touching neighbors and touching neighbors will play just the same role in their effect towards the $\rhobarred$ or $w(\cdot)$. Intuitively, the subset of a given local packing $\mathcal{L}$ consisting those neighbors with small buckling heights, $h_i$, will play a major role in the estimation of $\rhobarred$. This motivates the following definition of close neighbors and reduced local packings.

\begin{definition}
A neighbor $S_j \in \mathcal{L}(\cdot)$ with $h_i$ at most equal to 0.07 will be called a \textit{close neighbor}; the sub-local packing of $\mathcal{L}(\cdot)$ consisting of its close neighbors will be henceforth referred to as the \textit{reduced sub-packing of $\mathcal{L}(\cdot)$} and denoted by $\mathcal{L}_0(\cdot)$
\end{definition}

\begin{remarks}
\begin{enumerate}[(i)]

\item Of course, there are local packings  $\lof$ with no close neighbor at all; such local packings will have no reduced sub-packing, but it should be very easy to estimate their $\rhobarred$ to be much less than $\spet$.

\item For the purpose of proving Case II of Theorem I, we may assume without loss of generalities of the proof that $\lof$, actually, have quite a few close neighbors, and we shall regard $\lof$ as an extension of its reduced sub-packing, i.e. $\lof\supseteq \mathcal{L}_0(\cdot)$.

\item A given reduced local packing ${\cal L}_0(\cdot)$ will be called a {\it saturated} reduced local packing if it can {\it not} be extended to another local packing with additional close neighbors. Let ${\cal L}_0(\cdot)$ be a saturated reduced local packing, an extension ${\cal L}(\cdot)\supseteq{\cal L}_0(\cdot)$ will be referred to as a tightest extension of ${\cal L}_0(\cdot)$ if $\bar{\rho}({\cal L}(\cdot))$ already achieves the maximality among all possible extensions of ${\cal L}_0(\cdot)$ and $\bar{\rho}({\cal L}_0(\cdot))$ is defined to be the $\bar{\rho}(\cdot)$ of its tightest extensions.

\end{enumerate}
\end{remarks}

\noindent
(2) {\bf Some pertinent generalities}

\medskip

For a given saturated reduced local packing ${\cal L}_0(\cdot)=\{S_0;\; S_j,\>1\leq j\leq\ell\}$, set
$$ \Sigma({\cal L}_0)=\{A_j,\>1\leq j\leq \ell\}\quad \hbox{and}\quad \{h_j,\>1\leq j\leq \ell\}$$
to be the direction profile of $\{\overrightarrow{OO_j},\>1\leq j\leq \ell\}$ and the set of buckling heights of $\{S_j\}$, namely
$$\overrightarrow{OO_j}=2(1+h_j)\overline{OA_j},\quad 1\leq j\leq \ell.$$
Then, ${\cal S}(\Sigma({\cal L}_0))$ will be referred as the associated spherical configuration of ${\cal L}_0(\cdot)$.

Let ${\cal L}_0(\cdot)$ be a given {\it saturated} reduced local packing and ${\cal L}(\cdot)$ be one of its {\it tightest} extension and $\Sigma_0$ (resp. $\Sigma$) be the directional profiles of ${\cal L}_0(\cdot)$ (resp. ${\cal L}(\cdot)$), ${\cal S}(\Sigma_0)$ (resp. ${\cal S}(\Sigma)$) be their associated spherical configurations. Set $\{V_k\}$ to be that of $\Sigma\setminus\Sigma_0$, if any, and $\{\Gamma_i\}$ to be the connected components of the sub-graph in ${\cal S}(\Sigma)$ spanned by $\{V_k\}$. To each $\Gamma_i$, if any, the union of those stars of $V_k\in\Gamma_i$ will be referred to as the {\it big hole} in ${\cal S}(\Sigma_0)$ associated to $\Gamma_i$, namely
\begin{equation}
B_i={\textstyle \bigcup}\{St(V_k),\>V_k\in\Gamma_i\}={\textstyle\bigcup}\{\sigma_{ij}\}
\end{equation}
while $B_i$ will be referred to as the base of the mountain range with those peaks $V_k$ in $\Gamma_i$, which is, of course, union of triangles in ${\cal S}'(\Sigma_0)$. Just for the convenience of areawise estimation, we shall assign the densities of those triangles, say $\rho(\sigma_{ij})$, $\sigma_{ij}\subseteq B_i$, all equal to the weighted average of that of faces of ${\cal S}'(\Sigma)$ inside of $B_i$. Then, it is not difficult to apply the $(k,\delta)$-analysis to provide areawise estimates for such clusters of triangles of ${\cal S}'(\Sigma_0)$, thus proving a kind of averaged areawise optimal estimate, namely, bounded above by that of the special case in which the buckling heights at those vertices of $\{\sigma_{ij}\}$ are all equal to $0$. For example, in the simplest case of quadrilateral big holes, their averaged densities are bounded above by that of the extremal triangles with half of their area.

\subsubsection{The basic ideas and crucial steps of the proof of Theorem~I for the remaining case of non-Type-I local packings (i.e. Case~II)}
\label{subsubsec:pfcase2}

In fact, the major reason that makes a simple clean-cut proof of Theorem~I for Case~II, at all, possible is the simple idea of using strategy of proof by contradiction, namely, proving the {\it non-existence} of a non-Type-I local packing ${\cal L}(\cdot)$ with $\bar{\rho}\geq \slfrac{\pi}{\sqrt{18}}$. Moreover, the following are two crucial steps along the journey of such a proof by contradiction.

\begin{proposition} \label{prop:case2.1}
Let ${\cal L}(\cdot)$ be a non-Type-I local packing with $\bar{\rho}(\cdot)\geq \slfrac{\pi}{\sqrt{18}}$. Then ${\cal L}(\cdot)$ contains at least twelve close neighbors.
\end{proposition}

\begin{proposition} \label{prop:case2.2}
Let ${\cal L}(\cdot)$ be a non-Type-I local packing with $\bar{\rho}(\cdot)\geq \slfrac{\pi}{\sqrt{18}}$. Then ${\cal S}'(\Sigma_0)$ is an icosahedron.
\end{proposition}

Therefore, we shall begin with the proofs of Proposition~\ref{prop:case2.1} and Proposition~\ref{prop:case2.2} again by the method of proof by contradiction.

\noindent
(1) {\it The proof of Proposition~\ref{prop:case2.1}}

Suppose that ${\cal L}(\cdot)$ is a (non-Type-I) local packing with at most eleven close neighbors. Let us first consider the most critical case that it actually has eleven close neighbors. Then ${\cal S}'(\Sigma_0)$ consists of eighteen triangles, say $\{\sigma_j,\>1\leq j\leq 18\}$. Note that the average of $\{|\sigma_j|,\>1\leq j\leq 18\}$ is always equal to $\frac{2}{9}\pi$ which is already slightly larger than $\frac{1}{2}\square_{\frac{\pi}{3}}=2\arcsin\frac{1}{3}$. Therefore, it is quite straightforward to apply the triangular areawise estimate to show that $\bar{\rho}(\cdot)$ of such an ${\cal L}(\cdot)$ must be considerably smaller than $\slfrac{\pi}{\sqrt{18}}$, contradicting to the assumption that $\bar{\rho}({\cal L}(\cdot))\geq \slfrac{\pi}{\sqrt{18}}$. Moreover, it is quite clear that the same kind of triangular areawise estimate will show that $\bar{\rho}(\cdot)$ must be smaller than $\slfrac{\pi}{\sqrt{18}}$ by a bigger margin if $\#(\Sigma_0)<11$. This completes the proof of Proposition~\ref{prop:case2.1}.
\hfill $\square$

\medskip

\noindent
(2) {\it The proof of Proposition~\ref{prop:case2.2}}

Suppose the contrary that there exists a non-Type-I local packing ${\cal L}(\cdot)$ with at least twelve close neighbors, with $\bar{\rho}(\cdot)\geq \slfrac{\pi}{\sqrt{18}}$ and having at least one $6\triangle$-star, say denoted by $St(A_1)$. Then, the condition of $\bar{\rho}(\cdot)\geq\slfrac{\pi}{\sqrt{18}}$ implies that the total amount of buckling heights of those vertices of $St(A_1)$ must be quite small, say at most $0.1$. For otherwise, there must be some $5\triangle$-stars with $\tilde{\rho}(\cdot)$ considerably higher than $\slfrac{\pi}{\sqrt{18}}$ to counter balance $\tilde{\rho}(St(A_1))$, and moreover, such $5\triangle$-stars can only be either adjacent or overlapping to $St(A_1)$. Therefore, $St(A_1)$ must be of the $(\triangle,\triangle,\boxslash,\boxslash)$-type and ${\cal S}'(\Sigma_0)$ must be  a small deformation of $5\boxslash$-type, and hence, it is quite easy to show that $\bar{\rho}(\cdot)$ is smaller than $\slfrac{\pi}{\sqrt{18}}$ with substantial margin, contradicting to the assumption that $\bar{\rho}(\cdot)\geq\slfrac{\pi}{\sqrt{18}}$.

On the other hand, the complementary region of the $6\triangle$-star $St(A_1)$ with at most $0.1$ total buckling heights is a small deformation of that of a Type-I $6\triangle$-star we have already analyzed in the proof of Lemma~\ref{lem:nonico}. Therefore, the condition of $\bar{\rho}(\cdot)\geq \slfrac{\pi}{\sqrt{18}}$ again implies that ${\cal S}'(\Sigma_0)$ can not have more than twelve close neighbors and moreover, ${\cal S}'(\Sigma_0)$ must be of $6\boxslash$-type or $5\boxslash$-type. Thus, straightforward application of the collective areawise estimates of $\boxslash$-pairs and lune clusters, {\it generalized} to the geometric situation with small amount of buckling heights, will produce an upper bound of $\bar{\rho}(\cdot)$ less than $\slfrac{\pi}{\sqrt{18}}$, again contradicting to the assumption that $\bar{\rho}(\cdot)\geq \slfrac{\pi}{\sqrt{18}}$. This completes the proof of Proposition~\ref{prop:case2.2}. \hfill $\square$

\subsection{The proof of Theorem~I for Case~II}

Thanks to Proposition~\ref{prop:case2.1} and Proposition~\ref{prop:case2.2}, the proof of Theorem~I for Case~II can now be reduced to the proof of the following upper bound, namely
\begin{equation}
\bar{\rho}(\cdot)<\slfrac{\pi}{\sqrt{18}}
\end{equation}
for the remaining case of non-Type-I local packings with icosahedra ${\cal S}'(\Sigma_0)$. Similar to the proof for the case of Type-I icosahedra (cf. \S\ref{sec:proof}), we shall also subdivide the proof of such an estimate into the following subcases, namely
\begin{labeling}{cases}
\item[\textbf{Case~II$_1$:}] with ${\cal S}'(\Sigma_0)$ of $6\boxslash$-type.
\item[\textbf{Case~II$_2$:}] with ${\cal S}'(\Sigma_0)$ of $5\boxslash$-type.
\item[\textbf{Case~II$_3$:}] Others
\end{labeling}
Note that the same upper bound for areawise estimates of $\tilde{\rho}(\cdot)$ clearly still hold for the more general geometric situations of $\boxslash$-pairs (resp. lune clusters) in ${\cal S}'(\Sigma_0)$. Therefore, Case~II$_1$ (resp. Case~II$_2$) can again be readily proved by the same kind of areawise estimates for their $\boxslash$-pairs (resp. lune clusters). 

Finally, let us proceed to prove Case~II$_3$ by means of areawise estimates of $\tilde{\rho}(\cdot)$ for $5\triangle$-star-clusters. It is easy to see that the upper bound estimate of Lemma~\ref{lem:fivestar} still holds without modification for the more general situation of $5\triangle$-star-clusters in ${\cal S}'(\Sigma_0)$. However, for the range of $|St(\cdot)|$ exceeding $(\pi+0.21672)$, Lemma~\ref{lem:fivestar}$'$ certainly needs some kind of proper modification for the more general geometric situation of $5\triangle$-star-clusters of ${\cal S}'(\Sigma_0)$, because the realizability condition as stars in Type-I configuration is considerably stronger than that of ${\cal S}'(\Sigma_0)$, or in other words, for a given total area exceeding $(\pi+0.21672)$, certain kind of $5\triangle$-star of that total area and with their $\tilde{\rho}(\cdot)$ slightly or even considerably higher than that of the optimal ones in Example~\ref{ex:smallstar} can not be extended to Type-I configurations, but still can be extended to ${\cal S}'(\Sigma_0)$. The pertinent remark here is that the $\bar{\rho}(\cdot)$ of such ${\cal S}'(\Sigma_0)$ will always be lower than the $\bar{\rho}(\cdot)$ of those in Example~\ref{ex:smallstar} containing $St(A_2)$ with the same area and optimal $\tilde{\rho}(\cdot)$, whose $\bar{\rho}(\cdot)$ are already lower than $\slfrac{\pi}{\sqrt{18}}$ with a comfortable margin. Therefore, just for the purpose of proving ``$\bar{\rho}(\cdot)<\slfrac{\pi}{\sqrt{18}}$'' for Case~II$_3$, one may assume without loss of generality that the $\tilde{\rho}(\cdot)$ of stars with such large total areas, if any, are bounded above by that we used in the proof of Case~I$_3$, simply because the $\bar{\rho}(\cdot)$ of those configurations with stars of large total areas and with $\tilde{\rho}(\cdot)$ higher than such upper bound are already substantially smaller than $\slfrac{\pi}{\sqrt{18}}$. Hence, Case~II$_3$ can again be proved by the same estimates as that of Case~I$_3$.

This completes the proof of Case~II$_3$. \hfill $\square$

%% file: 8-conclusion/8-conclusion.tex
\section{Concluding remarks, retrospects and prospects}

\subsection{Retrospects}

First of all, the sphere packing problem should be, more naturally, regarded as a problem of \textit{geometric understanding on crystal formations of dense type} (cf. Theorem~\ref{leastaction}, \S \ref{sec:majortheorems}), a problem suggested by the nature, on the remarkable geometric feature of such crystals for a majority of monatomic elements, rather than whatever is commonly attributed to be the seeking of a proof of Kepler's conjecture.

In retrospect, both conceptually as well as technically, it is a kind of problem on in depth understanding the underlying profound interplays between the wonderful symmetry of the physical space and the least action principle (i.e. optimality) of crystal formation of dense type as formulated in \cite{hsiang}, roughly speaking, a kind of philosophical belief that the remarkable geometric precision and regularity of such crystals should be the consequence of optimality in packing density. Anyhow, this naturally leads to the proof of Theorem~\ref{leastaction} (cf. \S \ref{sec:majortheorems}). 

Note that one needs firstly to give a \textit{proper definition} of the concept of global density for packing of the second kind and then proceed to study its optimality such as the one we give in (\S \ref{subsubsec:relativedensity}). Such a concept of \textit{global density}, no matter how one chooses to define it will be inevitably rather complicated and very difficult to compute in general. Thus, it is extremely remarkable that the global density in (\S \ref{subsubsec:relativedensity}) actually has a clean-cut geometric structure for its optimality (cf. Theorem~\ref{leastaction}). Here, one will naturally wonder about: What is the underlying geometric reason for such an almost magical reality? And how to prove it?

In retrospect, the local geometric structure of those crystals of dense type actually already reveals the clue of Nature's secret, namely, the whole structure of such immense aggregates are always the repetition of only two kinds of local packings, i.e. the f.c.c. and the h.c.p., and such a reality of nature strongly suggests that \textit{global optimality} of the second kind of packings \textit{should be} the consequence of a kind of \textit{local optimization everywhere}! Thus, the entire journey starting with the proper definition of global density and ending up with the proof of such a least action principle of crystal formation of dense type amounts to find the proper \textit{local invariant} (i.e. the \textit{locally averaged density $\bar{\rho}(\mathcal{L})$}) and the proof of its \textit{optimal estimate} with the f.c.c. and the h.c.p. as the \textit{unique two of optimality}, namely, Theorem~I. This is exactly what we are presenting in this paper in the setting of geometric invariant theory of space. 

In summary and from the overall point of view of solid geometry (the geometry of our physical space), the sphere packing problem is an outstanding natural problem which, on the one hand, tests and challenges our depth of geometric understanding and techniques, and on the other hand, it also inspires and leads to further improvements of geometric understanding. Indeed, it is a wonderful geometric problem of the nature that certainly deserve a clean-cut proof in the classical tradition of solid geometry. I am glad that I have finally succeeded in paying my due respect to this problem in the proper setting of geometric invariant theory of space.

\subsection{Prospects}

Note that the techniques we developed in \S \ref{sec:summary} and \S \ref{sec:review} on geometric invariant theory and spherical geometry are naturally applicable to the study of many other problems in solid geometry, while we are only focusing on their applications towards the proof of Theorem~I. Of course, the broader prospect would naturally be further developing solid geometry in the setting of geometric invariant theory of the space by studying many other natural geometric problems, especially those problems naturally arise in the study of physics, such as the problem of sphere packing and crystal formations in the realm of solid state physics.

However, due to the limited scope of this paper, we shall mention here two specific such problems still in the realm of sphere packings.

\begin{description}
\item[Problem~1:] What should be the optimal shape for large scale sphere packings of the first kind, namely, with container?
\item[Problem~2:] Problem of sphere packings with two sizes, say of radii $1$ and $r<1$ respectively.
\end{description}

\noindent
{\bf Acknowledgements}: This is the final version of a manuscript that has been refined over time with many helpful discussions. In particular I would like to acknowledge the great help and many inspiring discussions with Lee-Ping Wang, which led to important improvement and substantial simplification of the proof of Theorem I. I also would like to thank Eldar Straume, Lars Sydney and Wing-Lung Lee for their generous help and in-depth discussions.